\documentclass{article}

\usepackage[utf8]{inputenc}
\usepackage{amsfonts,amsmath,amssymb,amsthm}
\usepackage{mathtools}
\usepackage{fancyhdr}
\usepackage[left=1in,right=1in,top=1in,bottom=1in]{geometry}
\usepackage{tabularx}
\usepackage{array}
\usepackage[hidelinks]{hyperref}
\newtheorem{definition}{Definition}[section]
\newtheorem{theorem}[definition]{Theorem}
\newtheorem{proposition}[definition]{Proposition}
\newtheorem{lemma}[definition]{Lemma}

\newtheorem{remark}[definition]{Remark}
\title{Uniformization of complex projective klt varieties by bounded symmetric domains}
\author{Aryaman Patel }
\date{}

\begin{document}

\maketitle

\begin{abstract}
    Using classical results from Hodge theory and more contemporary ones valid for complex projective varieties with Kawamata log terminal (klt) singularities, we deduce necessary and sufficient conditions for such varieties to be uniformized by each of the four classical irreducible Hermitian symmetric spaces of non compact type. We also deduce necessary and sufficient conditions for uniformization by a polydisk in the klt setting, which generalizes a classical result of Simpson. 
\end{abstract}

\tableofcontents

\vspace{0.5cm}

{\noindent \small \textit{Keywords}: Uniformization, klt singularities, Higgs bundles, Hermitian symmetric sapces, stability, principal
bundles, systems of Hodge bundles.}

\section{Introduction}


The problem of finding necessary and sufficient conditions for a complex projective variety to have a Hermitian symmetric space of non-compact type (also realized as a bounded symmetric domain) as its universal cover has now been studied extensively and for a long time. For example, in \cite{yau}, Yau showed that surfaces of general type which satisfy equality in the Miyaoka-Yau inequality are uniformized by the unit ball $\mathbb{B}^2\subset\mathbb{C}^2$. This result has since been generalized to quasi projective varieties (see \cite{yadeng}), and varieties with Kawamata log terminal (klt) singularities (see \cite{gkpt,gkpt2}). In \cite{gkpt}, the authors formulated an orbifold version of the Miyaoka-Yau inequality in terms of orbifold Chern classes (also called $\mathbb{Q}$-Chern classes), called the $\mathbb{Q}$-Miyaoka-Yau inequality, which is satisfied by varieties of general type with klt singularities. They showed that if a complex projective variety of general type with klt singularities and nef canonical divisor satisfies equality in the $\mathbb{Q}$-Miyaoka-Yau inequality, then the associated canonical model is the quotient of the unit ball by the action of a discrete group.\\
In \cite{simp}, Simpson formulated necessary and sufficient conditions for a smooth complex projective variety to have an arbitrary Hermitian symmetric space of non-compact type as its universal cover. These conditions require the existence of a uniformizing variation of Hodge structure (see Definition \ref{def8}) on the variety. A sufficient condition for uniformization by the polydisk, and the classical result of uniformization by the unit ball are some consequences of Simpson's results. In this note our goal is to generalize the main results of \cite{simp} to the klt setting, and as a consequence derive necessary and sufficient conditions for a complex projective klt variety with ample canonical divisor to be uniformized by one of the classical Hermitian symmetric spaces of non compact type. \\
In \cite[Corollary 9.7]{simp}, Simpson proved that if $X$ is a smooth complex projective variety of dimension $n$ whose tangent bundle $\mathcal{T}_X$ splits as a direct sum of line bundles of negative degree, and if the Chern classes of $X$ satisfy the equality $[c_1(X)^2-2c_2(X)]\cdot[K_X]^{n-2}=0$, then the universal cover of $X$ is the polydisk $\mathbb{H}^n$. This result was slightly improved by Beauville (see \cite[Theorem B]{bea}), who showed that the assumption on the Chern classes of $X$ is not necessary. Simpson's result of uniformization by the complex ball (see \cite[Corollary 9.8]{simp}) has been recently generalized to the klt setting (see \cite[Theorem 1.2]{gkpt} and \cite[Theorem 1.5]{gkpt2}). We work in a similar setting and prove an analogous result for uniformization by the polydisk. Our result is a slight generalization of Beauville's, in that we only assume that the tangent bundle $\mathcal{T}_{X_{reg}}$ of the smooth locus $X_{reg}$ of $X$ splits as a direct sum of line bundles (see Theorem \ref{thm1}). Due to the semistability of the tangent sheaf $\mathcal{T}_X$ with respect to the canonical divisor $K_X$, which was shown in \cite{gkpt}, we do not need to assume that the line bundles appearing in the decomposition of $\mathcal{T}_{X_{reg}}$ have negative degree. Our proof uses classical techniques developed by Simpson in \cite{simp}, and those developed by the authors in \cite{gkpt} and \cite{gkpt2}. Although the splitting of the tangent sheaf $\mathcal{T}_X$ is not a necessary condition for $X$ to be a polydisk quotient, it does hold on a finite, quasi-\'{e}tale cover of $X$.\\
We formulate the following analog of Simpson's result of uniformization by Hermitian symmetric spaces of non compact type, for a complex projective variety with klt singularities.
\begin{theorem}\label{genthm}
Let $X$ be a complex projective klt variety with ample canonical divisor $K_X$. Let $\mathcal{D}$ be a Hermitian symmetric space of non-compact type. Then $X\cong\mathcal{D}/\Gamma$, where $\Gamma$ is a discrete cocompact subgroup of $\textrm{Aut}(\mathcal{D})$, whose action on $\mathcal{D}$ is fixed point free in codimension one, if and only if:
\begin{enumerate}
    \item The smooth locus $X_{reg}$ admits a uniformizing system of Hodge bundles $(P,\theta)$ for any Hodge group $G_0$ of which $Aut(\mathcal{D})$ is a quotient by a discrete central subgroup, and
    \item $X$ satisfies the $\mathbb{Q}$-Chern class equality $\widehat{ch}_2(\mathcal{E}')\cdot[K_X]^{n-2}=0$, where $\mathcal{E}'$ is the reflexive extension of the system of Hodge bundles $P\times_K\mathfrak{g}$ to $X$.
\end{enumerate}
\end{theorem}
The notation appearing in Theorem \ref{genthm} will be made clear in the subsequent sections. As a consequence of Theorem \ref{genthm}, we derive necessary and sufficient conditions for a projective klt variety $X$ with ample canonical divisor to be uniformized by each of the four classical Hermitian symmetric spaces of noncompact type. In each case the conditions consist of the tangent bundle of the smooth locus $\mathcal{T}_{X_{reg}}$ admitting a reduction in structure group, and $X$ satisfying a $\mathbb{Q}$-Chern class equality. For example, the statement for uniformization by the Hermitian symmetric space $\mathcal{H}_n$ of type $CI$ (also known as the Siegel upper half space), is as follows.
\begin{theorem}\label{equivthm}
Let $X$ be a projective klt variety of dimension $n(n+1)/2$ such that the canonical divisor $K_X$ is ample. Then $X\cong\mathcal{H}_n/\Gamma$, where $\Gamma$ is a discrete cocompact subgroup of $\textrm{Aut}(\mathcal{H}_n)=PSp(2n,\mathbb{R})$, whose action on $\mathcal{D}$ is fixed point free in codimension one, if and only if $X$ satisfies
\begin{itemize}
    \item $\mathcal{T}_{X_{reg}}\cong\textrm{Sym}^2(\mathcal{E})$
    \item $[2\widehat{c}_2(X)-\widehat{c}_1(X)^2+2n\widehat{c}_2(\mathcal{E}')-(n-1)\widehat{c}_1(\mathcal{E}')^2]\cdot[K_X]^{n-2}=0$,
\end{itemize}
where $\mathcal{E}$ is a vector bundle of rank $n$ on $X_{reg}$, and $\mathcal{E}'$ denotes the reflexive extension of $\mathcal{E}$ to $X$.
\end{theorem}
Theorem \ref{equivthm} is a consequence of combining Propositions \ref{prop6} and \ref{propn} in Section 6. Analogous results to Theorem \ref{equivthm} for Hermitian symmetric spaces of types $DIII$, $BDI$, and $AIII$ for $p\neq q$ are Theorems \ref{equivthmb}, \ref{equivthmd}, and \ref{equivthmc} in Sections 7, 8, and 9 respectively. Necessary and sufficient conditions for uniformization by the polydisc and Hermitian symmetric space of type $AIII$ for $p=q$ are formulated separately in Sections 5 and 9 respectively.

\section*{Acknowledgements}

I would like to thank my advisor Prof. Daniel Greb firstly for introducing me to the problem and suggesting it as the topic of my thesis, and secondly for the many insightful discussions leading up to this article. I would also like to thank my co-advisor Prof. Ulrich G\"{o}rtz, and Prof. Jochen Heinloth for many discussions that were crucial to proving our results. I am grateful to Matteo Costantini for his detailed comments and feedback on the article, which significantly improved it. I am also grateful to Adrian Langer for his helpful comments on the first version of this article, and for pointing me to his paper from which I learned a lot. Finally, I would like to thank the RTG 2553 of ESAGA at the University of Duisburg-Essen for financial support throughout my PhD.

\section{Preliminaries}

In this section we introduce objects and notation that will be used to formulate our main results. We follow the conventions of \cite[Section 2]{gkpt}, and of \cite{simp}.

\subsection{Higgs sheaves and stability}

A closed subset $Z$ of a normal, quasi projective variety is called \emph{small} if the codimension of $Z$ in $X$ is at least two. An open subset $U$ of $X$ is called \emph{big} if the complement $X\setminus U$ is small. 

\begin{definition}[Quasi \'{e}tale morphism]
A morphism $f:X\to Y$ between normal varieties is called \emph{quasi \'{e}tale} if $f$ is of relative dimension zero, and \'{e}tale in codimension one. In other words, $f$ is quasi \'{e}tale if $\textrm{dim}X=\textrm{dim}Y$ and if there exists a subset $Z\subset X$ of codimension at least two such that the restricted map $f|_{X\setminus Z}:X\setminus Z\to Y$ is \'{e}tale.
\end{definition}

\begin{definition}[Sheaf with an operator, {\cite[Definition 4.1]{gkpt}}]\label{def0}
Let $X$ be a normal, quasi-projective variety, and $\mathcal{W}$ a coherent sheaf of $\mathcal{O}_X$-modules. A \emph{sheaf with a $\mathcal{W}$-valued operator} is a pair $(\mathcal{E},\theta)$, where $\mathcal{E}$ is a coherent sheaf on $X$ and $\theta:\mathcal{E}\to\mathcal{E}\otimes\mathcal{W}$ is an $\mathcal{O}_X$-linear sheaf morphism.
\end{definition}

\begin{definition}[Higgs sheaves, {\cite[Definition 5.1]{gkpt}}]\label{def1}
Let $X$ be a normal variety. A \emph{Higgs sheaf} is a pair $(\mathcal{E},\theta)$, where $\mathcal{E}$ is a coherent sheaf of $\mathcal{O}_X$-modules, and $\theta:\mathcal{E}\to\mathcal{E}\otimes\Omega^{[1]}_X$ is an $\Omega^{[1]}_X$-valued operator called the \emph{Higgs field}, such that the composed morphism
\begin{align*}
\mathcal{E}\xrightarrow{\theta}\mathcal{E}\otimes\Omega^{[1]}_X\xrightarrow{\theta\otimes \textrm{Id}}\mathcal{E}\otimes\Omega^{[1]}_X\otimes\Omega^{[1]}_X\xrightarrow{\textrm{Id}\otimes[\wedge]}\mathcal{E}\otimes\Omega^{[2]}_X
\end{align*}
is zero. The composed morphism is usually denoted $\theta\wedge\theta$.
\end{definition}

A Higgs sheaf is hence an instance of a sheaf with an $\Omega^{[1]}_X$-valued operator. Note that we use $\Omega^{[1]}_X$ instead of $\Omega^1_X$ because this turns out to be the correct notion when $X$ has klt singularities, as explained in \cite{gkpt}.\\
\\
We remark that a more general definition of Higgs sheaves on normal varieties was introduced in \cite[Section 4]{lan}. The advantage of this definition is that one can define define duals of Higgs sheaves, and also extend reflexive Higgs sheaves on the smooth locus of a klt variety to reflexive Higgs sheaves on the whole variety. This does not seem immediately possible with Higgs sheaves in the sense of Definition \ref{def1}. However, in this article we use Definition \ref{def1} because the proof of the main theorem relies on some results of \cite{gkpt2}, where this definition is used, and we work in the same generality as \cite{gkpt2}. \\
The constructions in \cite[Section 4]{lan} are however of independent interest, and we point the interested reader to the article \cite{lan} for details.

\begin{definition}[Morphism of Higgs Sheaves, {\cite[Definition 5.2]{gkpt}}]\label{def2}
In the setting of Definition \ref{def1}, a \emph{morphism of Higgs sheaves} $f:(\mathcal{E}_1,\theta_1)\to(\mathcal{E}_2,\theta_2)$ is a morphism $f:\mathcal{E}_1\to\mathcal{E}_2$ of sheaves which is compatible with the Higgs fields, i.e., $(f\otimes \textrm{Id}_{\Omega^{[1]}_X})\circ\theta_1=\theta_2\circ f$.
\end{definition}

An important notion that will be used later in this note is that of a Higgs subsheaf. To make a definition we must first define an invariant subsheaf of a sheaf with an operator.

\begin{definition}[Invariant subsheaf, {\cite[Definition 4.8]{gkpt}}]\label{def3}
Let $X$ be a normal quasi-projective variety and $(\mathcal{E},\theta)$ a sheaf with a $\mathcal{W}$-valued operator, as in Definition \ref{def0}. A coherent subsheaf $\mathcal{F}\subset\mathcal{E}$ is called \emph{$\theta$-invariant} if the map $\theta:\mathcal{F}\to\mathcal{E}\otimes\mathcal{W}$ factors through $\mathcal{F}\otimes\mathcal{W}$. We call $\mathcal{F}$ \emph{generically $\theta$-invariant} if the restriction $\mathcal{F}|_U$ is invariant with respect to $\theta|_U$, where $U\subset X$ is the maximal dense open subset of $X$ where $\mathcal{W}$ is locally free. 
\end{definition}

Let $(\mathcal{E},\theta)$ be a Higgs sheaf on $X$, as in Definition \ref{def1}. A coherent subsheaf $\mathcal{F}\subset\mathcal{E}$ is a \emph{Higgs subsheaf} of $\mathcal{E}$ if $\mathcal{F}$ is generically $\theta$-invariant, and if $(\mathcal{F},\theta|_{\mathcal{F}})$ is a Higgs sheaf in the sense of Definition \ref{def1}.\\
Following \cite{gkpt}, on a normal, quasi projective variety $X$, we denote by $N^1(X)_{\mathbb{Q}}$ the $\mathbb{Q}$-vector space of numerical Cartier divisor classes. For any sheaf $\mathcal{F}$ on $X$ whose determinant is $\mathbb{Q}$-Cartier, we denote the corresponding element of $N^1(X)_{\mathbb{Q}}$ by $[\mathcal{F}]=[\textrm{det}\mathcal{F}]$. Using \cite[Construction 2.17]{gkpt}, we can define the degree of a Weil divisorial sheaf $\mathcal{F}$ on a normal projective variety $X$ with respect to a system of $n-1$ line bundles, where $n=\textrm{dim}(X)$. 

\begin{definition}[Slope with respect to a nef divisor]
Let $X$ be a normal projective variety of dimension $n$, and let $H$ be a nef, $\mathbb{Q}$-Cartier divisor on $X$. If $\mathcal{F}$ is a torsion free, coherent sheaf on $X$, the \emph{slope} of $\mathcal{F}$ with respect to $H$ is given by 
\begin{align*}
    \mu_H(\mathcal{F})=\frac{[\mathcal{F}]\cdot[H]^{n-1}}{\textrm{rank}(\mathcal{F})}.
\end{align*}
\end{definition}

We call $\mathcal{F}$ \emph{semistable} with respect to $H$ if for any subsheaf $\mathcal{E}\subset\mathcal{F}$ with $0<\textrm{rank}(\mathcal{E})<\textrm{rank}(\mathcal{F})$, we have $\mu_H(\mathcal{E})\le\mu_H(\mathcal{F})$. We call $\mathcal{F}$ \emph{stable} with respect to $H$ if the strict inequality holds.

\begin{definition}[Stability of sheaves with operator, \cite{gkpt}, Definition 4.13]
Let $X$ be a normal, projective variety and $H$ be any nef, $\mathbb{Q}$-Cartier divisor on $X$. Let $(\mathcal{E},\theta)$ be a sheaf with an operator, as in Definition \ref{def0}, where $\mathcal{E}$ is torsion free. We say $(\mathcal{E},\theta)$ is \emph{semistable} with respect to $H$ if the inequality $\mu_H(\mathcal{F})\le\mu_H(\mathcal{E})$ holds for all generically $\theta$-invariant subsheaves $\mathcal{F}$ of $\mathcal{E}$ with $0<rank(\mathcal{F})<rank(\mathcal{E})$. The pair $(\mathcal{E},\theta)$ is called \emph{stable} with respect to $H$ if the strict inequality $\mu_H(\mathcal{F})<\mu_H(\mathcal{E})$ holds. Direct sums of stable sheaves with operator are called \emph{polystable}.
\end{definition}

A Higgs sheaf is \emph{stable} (resp. \emph{semistable}, \emph{polystable}) if it is stable (resp. semistable, polystable) as a sheaf with an $\Omega^{[1]}_X$-valued operator.\\
\\
For situations that arise later in the article, we would like to use a more general notion of stability, which works for sheaves over the smooth locus of a normal projective variety. This is developed in \cite[Section 2]{naht}. For convenience, we state the following useful definitions which appear therein. 

\begin{definition}[Slope for sheaves on the smooth locus]\label{slreg} Let $X$ be a normal, projective variety of dimension $n$, and let $\mathcal{E}$ be a torsion free, coherent sheaf of rank $r$ on the smooth locus $X_{reg}$. If $H$ is a nef, $\mathbb{Q}$-Cartier divisor on $X$, define the slope of $\mathcal{E}$ with respect to $H$ as 
\begin{align*}
    \mu_H(\mathcal{E})=\frac{c_1(j_*\mathcal{E})\cdot[H]^{n-1}}{r}
\end{align*}
    where $j:X_{reg}\to X$ is the inclusion.
\end{definition}

Note that the sheaf $\mathcal{E}$ is assumed to algebraic. For coherent analytic sheaves $\mathcal{F}$ on the analytic space associated to $X_{reg}$, the pushforward $j_*\mathcal{F}$ is in general not analytic, as remarked in \cite[Remark 2.24]{naht}. We will view $X_{reg}$ as a quasi projective variety over $\mathbb{C}$, hence coherent sheaves on $X_{reg}$ will always be algebraic.\\
\\
As earlier, we call $\mathcal{E}$ \emph{semistable} with respect to $H$ if for any subsheaf $\mathcal{F}\subset\mathcal{E}$ on $X_{reg}$ with $0<\textrm{rank}(\mathcal{F})<\textrm{rank}(\mathcal{E})$, we have 
$\mu_H(\mathcal{F})\le\mu_H(\mathcal{E})$. Call $\mathcal{E}$ \emph{stable} with respect to $H$ if the strict equality holds.\\
These notions can be extended to sheaves with operator on $X_{reg}$ as follows.

\begin{definition}[Stability for sheaves on the smooth locus]\label{streg}
    Let the setting be as in Definition \ref{slreg}. Let $\mathcal{W}$ be a coherent sheaf on $X_{reg}$, and let $\theta:\mathcal{E}\to\mathcal{E}\otimes\mathcal{W}$ be a $\mathcal{W}$-valued operator. We say that $(\mathcal{E},\theta)$ is \emph{semistable} with respect to $H$ if the inequality $\mu_H(\mathcal{F})\le\mu_H(\mathcal{E})$ holds for all generically $\theta$-invariant subsheaves $\mathcal{F}$ of $\mathcal{E}$ with $0<rank(\mathcal{F})<rank(\mathcal{E})$. Call $(\mathcal{E},\theta)$ \emph{stable} with respect to $H$ if the strict inequality $\mu_H(\mathcal{F})<\mu_H(\mathcal{E})$ holds. Direct sums of stable sheaves with operator are called \emph{polystable}.
\end{definition}

We can analogously define notions of stability, semistabilty, and polystability for Higgs sheaves over the smooth locus of a normal variety.
Lemmas 2.26 and 2.27 in \cite{naht} give useful properties of the above more general notions of slope and stability. 

\subsection{Uniformization}

The notions of a system of Hodge bundles, principle bundles, and uniformising bundles play a central role in the proof of the main theorem. We state the definitions here as they are in Simpson's paper \cite{simp}. The variety $X$ is assumed to be smooth through Definitions \ref{def4}-\ref{def8}.

\begin{definition}[System of Hodge bundles]\label{def4}
A \emph{system of Hodge bundles} on $X$ is a direct sum of holomorphic vector bundles $E=\bigoplus_{p,q}E^{p,q}$ together with maps $\theta:E^{p,q}\to E^{p-1,q+1}\otimes\Omega^1_X$, such that the composed map $\theta\wedge\theta:E\to E\otimes\Omega^2_X$ is zero. In particular, a system of Hodge bundles on $X$ is a Higgs bundle equipped with an action of the group $A=U(1)\times U(1)$. 
\end{definition}

Let $E$ be a system of Hodge sheaves. A subsystem of Hodge sheaves of $E$ is a sub-Higgs sheaf of $E$ which is preserved by the action of the group $A$.

\begin{definition}[Complexification of a Lie group]\label{wcom}
A \emph{complexification} of a real Lie group $G$ is a complex Lie group $G_{\mathbb{C}}$ containing $G$ as a real Lie subgroup such that the Lie algebra $\mathfrak{g}_{\mathbb{C}}$ of $G_{\mathbb{C}}$ is a complexification of the Lie algebra $\mathfrak{g}$ of $G$ i.e., $\mathfrak{g}_{\mathbb{C}}=\mathfrak{g}\otimes_{\mathbb{R}}\mathbb{C}$. The group $G$ is then a real form of the group $G_{\mathbb{C}}$.
\end{definition}

Not every real Lie group $G$ admits a complexification in the above sense. For example, the universal cover of $SL(2,\mathbb{R})$ does not admit such a complexification. In fact, $G$ admits such a complexification if and only if $G$ is a linear group. If a complexification exists, it is not necessarily unique. Compact Lie groups always admit complexifications (see for example \cite{bour} for details). \\
A more general and satisfying notion of complexification is the following.

\begin{definition}[Universal complexification]\label{ucom}
The \emph{universal complexification} of a real Lie group $G$ is a complex Lie group $G_{\mathbb{C}}$ together with a continuous homomorphism $\varphi:G\to G_{\mathbb{C}}$ such that, for any continuous homomorphism $f:G\to H$ to a complex Lie group $H$, there is a unique complex analytic homomorphism $g:G_{\mathbb{C}}\to H$ such that $f=g\circ\varphi$. 
\end{definition}

Universal complexifications of Lie groups always exist and are unique up to unique isomorphism. If the group $G$ is linear then so is the universal complexification $G_{\mathbb{C}}$ and the homomorphism $\varphi:G\to G_{\mathbb{C}}$ is an inclusion. If the Lie group $G$ has identity component $G^o$, and component group $\Gamma=G/G^o$, then the extension 
\begin{align*}
    1\to G^o\to G\to\Gamma\to1
\end{align*}
induces an extension 
\begin{align*}
    1\to G^o_{\mathbb{C}}\to G_{\mathbb{C}}\to\Gamma\to1,
\end{align*}
where $G_\mathbb{C}$ is the universal complexification of $G$. In particular, $G_{\mathbb{C}}$ is not connected if $G$ is not connected.\\
Universal complexifications are often difficult to compute in practice, but for us it suffices to work with complexifications in the sense of Definition \ref{wcom}. Since the Lie groups we work with are linear, such complexifications exist.

\begin{definition}[Hodge group]\label{def5}
A \emph{Hodge group} is a semisimple real algebraic Lie group $G_0$, together with a Hodge decomposition of the complexified Lie algebra 
\begin{align*}
    \mathfrak{g}=\bigoplus_p\mathfrak{g}^{p,-p}
\end{align*}
such that $[\mathfrak{g}^{p,-p},\mathfrak{g}^{r,-r}]\subset\mathfrak{g}^{p+r,-p-r}$ and such that $(-1)^{p+1}\textrm{Tr}(ad(U)\circ ad(\Bar{V}))>0$ for $U,V\in\mathfrak{g}^{p,-p}\setminus\{0\}$. Here $ad:\mathfrak{g}\to\textrm{Der}(\mathfrak{g})$ denotes the adjoint representation of the Lie algebra $\mathfrak{g}$.
\end{definition}

Let $G_0$ be a Hodge group as in the above definition, and let $K_0\subset G_0$ be the subgroup corresponding to the Lie algebra $\mathfrak{k}=\mathfrak{g}_0^{0,0}$. It is the subgroup of elements $k$ such that $ad(k)$ preserves the Hodge decomposition of $\mathfrak{g}$. In particular $ad(k)$ preserves the positive definite form $(-1)^{p+1}\textrm{Tr}(ad(U)\circ ad(\Bar{V}))$ so $K_0$ is compact.\\
\\
Henceforth, we assume that $G_0$ and $K_0$ are linear groups, so that they always admit complexifications in the sense of Definition \ref{wcom} (see \cite[Chapter VII, Section 9]{knapp}). Let $G$ and $K$ denote such complexifications of $G_0$ and $K_0$ respectively.

\begin{definition}[Principal system of Hodge bundles]\label{def6}
A \emph{principal system of Hodge bundles} on $X$ for a Hodge group $G_0$ is a principal holomorphic $K$-bundle $P$ on $X$ together with a holomorphic map 
\begin{align*}
    \theta:\mathcal{T}_X\to P\times_K\mathfrak{g}^{-1,1}
\end{align*}
such that $[\theta(u),\theta(v)]=0$ for all local sections $u,v$ of $\mathcal{T}_X$.
\end{definition}

\begin{definition}[Hodge group of Hermitian type]\label{def7}
A Hodge group $G_0$ of \emph{Hermitian type} is a Hodge group such that the Hodge decomposition of $\mathfrak{g}$ has only types $(1,-1)$, $(0,0)$, and $(-1,1)$, and such that $G_0$ has no compact factors.
\end{definition}

In the case of a Hodge group of Hermitian type, $K_0\subset G_0$ is a maximal compact subgroup, and the quotient $\mathcal{D}=G_0/K_0$ is a \emph{Hermitian symmetric space of non-compact type}, and it can also be realized as a \emph{bounded symmetric domain}. Moreover, all bounded symmetric domains arise in this way.

\begin{definition}[Uniformizing system of Hodge bundles]\label{def8}
Let $G_0$ be a Hodge group of Hermitian type. A \emph{uniformizing system of Hodge bundles} on $X$ for $G_0$ is a principal system of Hodge bundles $(P,\theta)$ on $X$ such that map $\theta:\mathcal{T}_X\to P\times_K\mathfrak{g}^{-1,1}$ is an isomorphism of sheaves.
\end{definition}

A uniformizing system of Hodge bundles on $X$ corresponds to a holomorphic reduction of structure group for $\mathcal{T}_X$ to $K\to GL(n,\mathbb{C})$, where the map $K\to GL(n,\mathbb{C})$ is given by the adjoint representation of $K$ on $\mathfrak{g}^{-1,1}$. 

\begin{definition}[Uniformizing variation of Hodge structure]
A \emph{uniformizing variation of Hodge structure} is a uniformizing system of Hodge bundles $(P,\theta)$ equipped with a flat metric.    
\end{definition}

The table below shows groups $G_0$, $K_0$, and their respective complexifications $G$, $K$ associated to the bounded symmetric domains which appear later in this article.\\
\\
\begin{tabularx}{1.0\textwidth} { 
  | >{\centering\arraybackslash}X
  | >{\centering\arraybackslash}X
  | >{\centering\arraybackslash}X
  | >{\centering\arraybackslash}X 
  | >{\centering\arraybackslash}X | }
 \hline
 Bounded symmetric domain & $G_0$ & $K_0$ & $G$ & $K$ \\
 \hline
 Polydisk ($\mathbb{H}^n$)  & $SL(2,\mathbb{R})^n$  & $U(1)^n$ & $SL(2,\mathbb{C})^n$ & $(\mathbb{C}^*)^n$ \\
\hline
Type CI ($\mathcal{H}_n$)  & $Sp(2n,\mathbb{R})$  & $U(n)$ & $Sp(2n,\mathbb{C})$ & $GL(n,\mathbb{C})$ \\
\hline
Type DIII ($\mathcal{D}_n$)  & $SO^*(2n)$  & $U(n)$ & $SO^*(2n,\mathbb{C})$ & $GL(n,\mathbb{C})$ \\
\hline
Type BDI ($\mathcal{B}_n$)  & $SO(2,n)$  & $SO(2)\times SO(n)$ & $SO(2+n,\mathbb{C})$ & $SO(2,\mathbb{C})\times SO(n,\mathbb{C})$ \\
\hline
Type AIII ($\mathcal{A}_{pq}$)  & $SU(p,q)$  & $S(U(p)\times U(q))$ & $SL(p+q,\mathbb{C})$ & $S(GL(p,\mathbb{C})\times GL(q,\mathbb{C}))$ \\
\hline
\end{tabularx}

\subsection{A weak characterization}

The following observation is a weaker version of \cite[Proposition 4.1]{gkp2}, in a more general setting.

\begin{lemma}\label{lem1}
Let $X$ be normal, complex quasi-projective variety of dimension $n$ with klt singularities. Assume that the sheaf of reflexive differentials $\Omega^{[1]}_X$ is of the form 
\begin{align}\label{split}
    \Omega^{[1]}_X\cong\mathcal{L}_1\oplus\dots\oplus\mathcal{L}_n.
\end{align}
where $\mathcal{L}_i$ is a rank one $\mathbb{Q}$-Cartier sheaf for all $1\le i\le n$. Then $X$ has only quotient singularities.
\end{lemma}
\begin{proof}
Since each $\mathcal{L}_i$ is assumed to be $\mathbb{Q}$-Cartier, there is a minimal number $N_i\in\mathbb{N}$ for each $i$, such that $\mathcal{L}_i^{[\otimes N_i]}$ is locally free. Note that for any point $x\in X$, for each $1\le i\le n$, we can find an open neighbourhood $U_i=U_i(x)$ of $x$ over which $\mathcal{L}_i^{[\otimes N_i]}$ is trivial, i.e. such that $(\mathcal{L}_i^{[\otimes N_i]})|_{U_i}\cong\mathcal{O}_{U_i}$. Let $V=\bigcap_iU_i$. Then $V$ is a non-empty open neighbourhood of $x$ as it is the intersection of finitely many non-empty open neighbourhoods of $x$.\\
Consider the restriction of $\mathcal{L}_1^{[\otimes N_1]}$ to $V$ and observe that it is trivial. Let $\gamma_1:V_1\to V$ be the associated index one quasi-\'{e}tale cover which is cyclic of order $N_1$. The complex space $V_1$ again has klt singularities, and we have $\gamma_1^{[*]}\mathcal{L}_1\cong\mathcal{O}_{V_1}$. In particular, it follows that $\Omega^{[1]}_{V_1}\cong\gamma_1^{[*]}\Omega^{[1]}_V$, i.e., $\Omega^{[1]}_{V_1}\cong\gamma_1^{[*]}(\mathcal{L}_1|_V)\oplus\dots\oplus\gamma_1^{[*]}(\mathcal{L}_n|_V)\cong\mathcal{O}_{V_1}\oplus\dots\oplus\gamma^{[*]}_1(\mathcal{L}_n|_V)$. Note that since $\mathcal{L}_2^{[\otimes N_2]}$ is trivial over $V$, $\gamma_1^{[*]}\mathcal{L}_2^{[\otimes N_2]}$ is trivial over $V_1$. Now let $\gamma_2:V_2\to V_1$ be the associated index one quasi-\'{e}tale cover which is cyclic of order $N_2$. We again have that $V_2$ is klt and $\gamma_2^{[*]}\gamma_1^{[*]}\mathcal{L}_2\cong\mathcal{O}_{V_2}$. Thus it follows that $\Omega^{[1]}_{V_2}\cong\gamma_2^{[*]}\Omega^{[1]}_{V_1}\cong\gamma_2^{[*]}\gamma_1^{[*]}(\mathcal{L}_1|_V)\oplus\dots\oplus\gamma_2^{[*]}\gamma_1^{[*]}(\mathcal{L}_n|_V)\cong\mathcal{O}_{V_2}\oplus\mathcal{O}_{V_2}\dots\oplus\gamma_2^{[*]}\gamma^{[*]}_1(\mathcal{L}_n|_V)$.\\
Continuing in this way $n-2$ more times we arrive at an index one quasi-\'{e}tale cover $\gamma_n:V_n\to V_{n-1}$ of order $N_n$, such that $V_n$ is klt and $\gamma_n^{[*]}\dots\gamma_1^{[*]}\mathcal{L}_n\cong\mathcal{O}_{V_n}$. This implies that $\Omega_{V_n}^{[1]}\cong\gamma_n^{[*]}\dots\gamma_1^{[*]}\Omega^{[1]}_V\cong\mathcal{O}_{V_n}^{\oplus n}$ i.e., $\Omega_{V_n}^{[1]}$ and $\mathcal{T}_{V_n}$ are both free. The solution for the Lipman-Zariski conjecture for spaces with klt singularities (\cite[Theorem 16.1]{gkkp}) then asserts that $V_n$ is smooth.\\
Let $G_n$ and $G_{n-1}$ be the Galois groups corresponding to the maps $V_n\to V_{n-1}$ and $V_{n-1}\to V_{n-2}$ respectively, and let $\widetilde{G}_n$ denote the Galois closure of $G_n$ and $G_{n-1}$. Then there is a variety $W_n$ such that $W_n\to V_n$ is quasi-\'{e}tale and Galois, and the composed map $W_n\to V_{n-2}$ is quasi-\'{e}tale and Galois with group $\widetilde{G}_n$. Moreover, it follows that $W_n$ is smooth because $V_n$ is smooth. Repeating this argument with the maps $W_n\to V_{n-2}$ and $V_{n-2}\to V_{n-3}$, we obtain a Galois, quasi-\'{e}tale cover $W_{n-2}\to V_{n-3}$ with $W_{n-2}$ smooth, the associated Galois group being $\widetilde{G}_{n-2}$, the Galois closure of $G_{n-2}$ and $\widetilde{G}_n$. Eventually, we arrive at a smooth, quasi-\'{e}tale Galois cover $W_1\to V$ with Galois group $\widetilde{G}_1$. In other words, $W_1\to V$ is a quotient map for the group $\widetilde{G}_1$. Hence we conclude that $X$ has quotient singularities only.
\end{proof}


\subsection{Another remark on stability}

We now make an observation about the stability of certain vector bundles associated to a semistable bundle. 

\begin{lemma}\label{lemsym}
Let $X$ be a $n$-dimensional algebraic variety over $\mathbb{C}$ and let $H$ be an ample divisor on $X$. Let $\mathcal{E}$ be a rank $r$ vector bundle on $X$ such that $\textrm{Sym}^2(\mathcal{E})$ is semistable with respect to $H$. Then $\mathcal{E}$ and $\mathcal{E}nd(\mathcal{E})$ are semistable with respect to $H$.
\end{lemma}
\begin{proof}
Note that $\textrm{Sym}^2(\mathcal{E})$ is a vector bundle of rank $r(r+1)/2$. Let $c_t(\mathcal{E})=\sum_{i=0}^rc_i(\mathcal{E})t^i$ be the Chern polynomial of $\mathcal{E}$, and let $\alpha_1,...,\alpha_r$ be the Chern roots of $\mathcal{E}$. Then the Chern polynomial of the $p$-th symmetric power $\textrm{Sym}^p(\mathcal{E})$ is given by
\begin{align*}
    c_t(\textrm{Sym}^p(\mathcal{E}))=\prod_{i_1\le\dots\le i_p}(1+(\alpha_i+\alpha_j)t).
\end{align*}
In particular, we have that $c_1(\textrm{Sym}^2(\mathcal{E}))$ is the coefficient of $t$ in the expression $\prod_{i\le j}(1+(\alpha_i+\alpha_j)t)$. A quick computation shows that $c_1(\textrm{Sym}^2(\mathcal{E}))=(r+1)\sum_{i=1}^r\alpha_i=(r+1)c_1(\mathcal{E})$. The slope of $\textrm{Sym}^2(\mathcal{E})$ with respect to $H$ is given by
\begin{align*}
    \mu_H(\textrm{Sym}^2(\mathcal{E}))=\frac{c_1(\textrm{Sym}^2(\mathcal{E}))\cdot[H]^{n-1}}{r(r+1)/2}=\frac{2c_1(\mathcal{E})\cdot[H]^{n-1}}{r}=2\mu_H(\mathcal{E}).
\end{align*}
Suppose $\mathcal{E}$ is not $H$-semistable. Let $\mathcal{F}\subset\mathcal{E}$ be a subsheaf of rank $r'$ such that $\mu_H(\mathcal{F})>\mu_H(\mathcal{E})$. Consider the short exact sequence
\begin{align*}
    0\to\mathcal{F}\to\mathcal{E}\to\mathcal{G}\to0
\end{align*}
where $\mathcal{G}=\mathcal{E}/\mathcal{F}$. Note that taking $\textrm{Sym}^2$ preserves surjective maps, so the map $\textrm{Sym}^2(\mathcal{E})\to\textrm{Sym}^2(\mathcal{G})$ is surjective. This implies that $\textrm{Sym}^2(\mathcal{G}^\vee)\subset \textrm{Sym}^2(\mathcal{E}^\vee)$. Since $\textrm{rank}(\mathcal{G})=r-r'$ and $c_1(\mathcal{G})=c_1(\mathcal{E})-c_1(\mathcal{F})$, we have 
\begin{align*}
    \mu_H(\mathcal{G})=\frac{(c_1(\mathcal{E})-c_1(\mathcal{F}))\cdot H^{n-1}}{r-r'}=\frac{r}{r-r'}\mu_H(\mathcal{E})-\frac{r'}{r-r'}\mu_H(\mathcal{F}).
\end{align*}
Thus $\mu_H(\mathcal{F})>\mu_H(\mathcal{E})$ implies that $\mu_H(\mathcal{G})<\mu_H(\mathcal{E})$, which further implies $\mu_H(\mathcal{G^\vee})=-\mu_H(\mathcal{G})>-\mu_H(\mathcal{E})=\mu_H(\mathcal{E}^\vee)$. But this means that $\mu_H(\textrm{Sym}^2(\mathcal{G}^\vee))>\mu_H(\textrm{Sym}^2(\mathcal{E}^\vee))$, which implies that $\textrm{Sym}^2(\mathcal{E}^\vee)$ is not $H$-semistable, and hence neither is $\textrm{Sym}^2(\mathcal{E})$, a contradiction.\\
Since $\mathcal{E}$ is semistable with respect to $H$, so is the dual bundle $\mathcal{E}^{\vee}$, and the endomorphism bundle $\mathcal{E}nd(\mathcal{E})\cong\mathcal{E}\otimes\mathcal{E}^{\vee}$, because the tensor product of semistable bundles is semistable.
\end{proof}

An analogous statement to Lemma \ref{lemsym} holds for the second wedge power.

\begin{lemma}\label{lemwedge}
Let $X$ be a $n$-dimensional algebraic variety and let $H$ be an ample divisor on $X$. Let $\mathcal{E}$ be a vector bundle of rank $r\ge3$ on $X$ such that $\bigwedge^2(\mathcal{E})$ is semistable with respect to $H$. Then $\mathcal{E}$ and $\mathcal{E}nd(\mathcal{E})$ are semistable with respect to $H$.
\end{lemma}

The proof of this Lemma is essentially the same as that of Lemma \ref{lemsym}, because applying $\bigwedge^2$ also preserves surjective maps, and a simple computation shows that $\mu_H(\bigwedge^2\mathcal{E})=2\mu_H(\mathcal{E})$, just as in the $\textrm{Sym}^2$ case.\\
It was pointed out to us by A. Langer that this Lemma does not work when the rank of $\mathcal{E}$ is 2. This is because $\bigwedge^2(\mathcal{E})$ will be a line bundle in this case, which is always semistable, but this does not imply that $\mathcal{E}$ is semistable.

\section{Revisiting Simpson's results}

In this section, our goal is to study the structure of the tangent bundle of a Hermitian symmetric space $\mathcal{D}$ of non-compact type. This structure descends to the tangent bundle of a smooth projective quotient $X=\mathcal{D}/\Gamma$, where $\Gamma$ is a discrete, cocompact group of automorphisms of $\mathcal{D}$. We use the notation of Sections 8 and 9 of \cite{simp}, and we use the theory developed in Chapter 12 of \cite{cmsp}.

\subsection{The tangent bundle of \texorpdfstring{$\mathcal{D}$}{}}

Let $G_0$ be a Hodge group, and $K_0$ the subgroup corresponding to the Lie algebra $\mathfrak{k}_0=\mathfrak{g}^{0,0}_0$, as in Definition \ref{def5}. Recall that since $G_0$ and $K_0$ are assumed to be linear, they admit complexifications in the sense of Definition \ref{wcom}. Let $G$ and $K$ be complexifications of $G_0$ and $K_0$ respectively.\\
\\
Let $X$ be a smooth projective variety of dimension $n$ over $\mathbb{C}$, and let $(P,\theta)$ be a principal system of Hodge bundles on $X$ for $G$. This means that $P$ is a principal holomorphic $K$-bundle on $X$, and $\theta$ is a holomorphic map
\begin{align*}
    \theta:\mathcal{T}_X\to P\times_K\mathfrak{g}^{-1,1}
\end{align*}
such that $[\theta(u),\theta(v)]=0$ for all local sections $u,v$ of $\mathcal{T}_X$.
\begin{definition}
A \emph{metric} $H$ for a principal system of Hodge bundles is a $C^\infty$ reduction of structure group of $P$ from $K$ to $K_0$, i.e., a principal $K_0$-bundle $P_H\subset P$.  
\end{definition}

Let $(P,\theta)$ be a principal system of Hodge bundles with metric $P_H\subset P$. This reduction in structure group corresponds to a Hermitian metric $H$ on the associated system of Hodge bundles $E=P\times_K\mathfrak{g}\cong P_H\times_{K_0}\mathfrak{g}$. Let $d'_H\in\mathcal{A}^1(\textrm{End}(E))$ denote the Chern connection on $E$ with respect to $H$, and let $d_H$ be the $K_0$-connection on $P_H$ from which $d'_H$ is induced. Now we view $E$ as a $G_0$-bundle $E'=R_H\times_{G_0}\mathfrak{g}^{-1,1}$, where $R_H=P_H\times_{K_0}G_0$, and $G_0$ acts on $\mathfrak{g}$ via the adjoint action. The map $\theta:\mathcal{T}_X\to P\times_K\mathfrak{g}^{-1,1}$ gives an $\textrm{End}(E')$-valued one-form $\theta'\in\mathcal{A}^{1,0}(\textrm{End}(E'))$. Let $\Bar{\theta}'\in\mathcal{A}^{0,1}(\textrm{End}(E'))$ be the adjoint of $\theta'$ with respect to $H$, i.e., $\langle{\theta'u,v\rangle}_H=\langle{u,\Bar{\theta}'v\rangle}_H$ for local sections $u,v$ of $E$. Let $\sigma$ and $\Bar{\sigma}$ be the $G_0$-connections on $R_H$ which induce $\theta'$ and $\Bar{\theta}'$ respectively. Then $D'_H=d'_H+\theta'+\Bar{\theta}'\in\mathcal{A}^1(\textrm{End}(E'))$ is a connection on $E$. Let $D_H=d_H+\sigma+\Bar{\sigma}$ denote the $G_0$-connection on $R_H$ which induces $D'_H$.   

\begin{definition}
A \emph{principal variation of Hodge structure} for a Hodge group $G_0$ is a principal system of Hodge bundles $(P,\theta)$ together with a metric $P_H$ such that the curvature of the associated connection $D_H$ is zero. 
\end{definition}

Let $(P,\theta)$ together with metric $(P_H,D_H)$ be a principal variation of Hodge structure on $X$, and let $\widetilde{X}$ be a universal cover of $X$. We denote by $\pi$ the quotient map $\pi:\widetilde{X}\to X$. Let $\widetilde{R}_H$ be the pullback of the $G_0$-bundle $R_H=P_H\times_{K_0} G_0$ to $\widetilde{X}$, then $\widetilde{R}_H=\pi^*(P_H\times_{K_0} G_0)=\widetilde{P}_H\times_{K_0} G_0$, where $\widetilde{P}_H=\pi^*P_H$ is a principal $K_0$-bundle on $\widetilde{X}$. The flat connection $D_H$ on $R_H$ pulls back to a flat connection on $\widetilde{R}_H$. Since $\widetilde{R}_H$ is a principal bundle with a flat connection on a simply connected space $\widetilde{X}$, we have a trivialisation $\phi:\widetilde{R}_H=\widetilde{P}_H\times_{K_0}G_0\cong\widetilde{X}\times G_0$, hence $\widetilde{R}_H$ admits a global section. A global section of $\widetilde{P}_H\times_{K_0}G_0$ corresponds to a $K_0$-equivariant map $\varphi:\widetilde{P}_H\to G_0$, which induces a map $\widetilde{P}_H/K_0\to G_0/K_0$. Since $\widetilde{X}$ is diffeomorphic to $\widetilde{P}_H/K_0$, we get a map $\widetilde{X}\to G_0/K_0$, which sends a point $x\in\widetilde{X}$ to a right $K_0$-coset $\varphi\widetilde{P}_{H,x}\subset G_0$.\\
\\
Let $H$ be a Lie group. Then, corresponding to every flat principal $H$-bundle $P\to X$, there is a group homomorphism $\rho:\pi_1(X)\to H$, known as the \emph{holonomy morphism} or the \emph{monodromy morphism}. More precisely, the following is true.
\begin{theorem}[\cite{morita}, Theorem 2.9]\label{monthm}
    The correspondence, which sends each flat principal $H$-bundle over $X$ to its holonomy morphism, induces a bijection
    $$
    \{\textrm{isomorphism classes of flat $H$-bundles over } X\}\cong \{\textrm{conjugacy classes of homomorphism } \rho:\pi_1(X)\to H\}.
    $$
\end{theorem}
In view of this, the flat $G_0$-bundle $R_H$ on $X$ corresponds to a homomorphism $\sigma:\pi_1(X)\to G_0$. Recall from the preceding paragraph that we have a trivialisation $\phi:\pi^*R_H=\widetilde{R}_H\cong\widetilde{X}\times G_0$. From the proof of \cite[Theorem 2.9]{morita}, it follows that the flat bundle $R_H$ can be expressed as 
\begin{align*}
    R_H=\widetilde{R}_H/\pi_1(X)=(\widetilde{X}\times G_0)/\pi_1(X),
\end{align*}
where $\pi_1(X)$ acts on $\widetilde{X}\times G_0$ by automorphisms. This action is described explicitly in \cite{morita}, p.58-59. Setting $\phi=(\phi_1,\phi_2)$, we have that $\phi(r\cdot\gamma)=(\phi_1(r)\gamma,\sigma(\gamma)^{-1}\phi_2(r))$ for all local sections $r\in \widetilde{R}_H$ and all $\gamma\in\pi_1(X)$. The map $\widetilde{X}\to G_0/K_0$, given by $x\mapsto\varphi \widetilde{P}_{H,x}$ is equivariant under the representation $\sigma$, i.e., 
\begin{align*}
    \varphi \widetilde{P}_{H,\gamma x}=\varphi(\gamma\cdot \widetilde{P}_{H,x})=\sigma(\gamma)(\varphi\widetilde{P}_{H,x})
\end{align*}
for all $\gamma\in\pi_1(X)$, and $x\in \widetilde{X}$.\\
Note that $\mathcal{D}=G_0/K_0$ is a homogeneous space and its tangent bundle $T_{\mathcal{D}}$ is a homogeneous vector bundle on $\mathcal{D}$. Since $G_0$ is a Hodge group, $G_0/K_0$ is in fact a \emph{reductive domain}. Thus there is an $\textrm{Ad}(K_0)$-invariant decomposition of the Lie algebra of $G_0$ as $\mathfrak{g}_0=\mathfrak{k}_0\oplus\mathfrak{m}$. From the discussion in \cite[Chapter 12.2]{cmsp}, there is an isomorphism $G_0\times_{K_0}\mathfrak{m}\cong T_{\mathcal{D}}$. To obtain the holomorphic tangent bundle $\mathcal{T}_{\mathcal{D}}$ of $\mathcal{D}$, consider the splitting $\mathfrak{m}\otimes\mathbb{C}=\mathfrak{m}^+\oplus\mathfrak{m}^-$, where $\mathfrak{m}^+=\bigoplus_{p>0}\mathfrak{g}^{p,-p}$, and $\mathfrak{m}^-=\bigoplus_{p<0}\mathfrak{g}^{p,-p}$. By \cite[Lemma 12.2.2]{cmsp}, this splitting defines a complex structure on $\mathfrak{m}$. By \cite[Lemma-Definition 12.2.3]{cmsp}, the holomorphic tangent bundle of $\mathcal{D}$ is given by $\mathcal{T}_{\mathcal{D}}\cong G_0\times_{K_0}\mathfrak{m}^-$, where $K_0$ acts on $\mathfrak{m}^-$ via the adjoint action. Let $H_0$ be the maximal compact subgroup of $G_0$ containing $K_0$. Then $G_0/H_0$ is a \emph{symmetric space}. The associated \emph{Cartan decomposition} of $\mathfrak{g}_0$ is $\mathfrak{g}_0=\mathfrak{h}_0\oplus\mathfrak{p}_0$. The \emph{Cartan involution} $\iota:\mathfrak{g}_0\to\mathfrak{g}_0$ is defined such that $\iota|_{\mathfrak{h}_0}=\textrm{id}$, and $\iota|_{\mathfrak{p}_0}=-\textrm{id}$. The Cartan decomposition is reflected in the tangent bundle of $\mathcal{D}$ as follows. There is a canonical projection $\omega:\mathcal{D}=G_0/K_0\to G_0/H_0$ with respect to which the tangent space at any point $x\in\mathcal{D}$ splits into vertical and horizontal tangent spaces given by
$$
T^v_{\mathcal{D},x}=\textrm{fiber of $\omega$ through }x,\;\;\;
T^h_{\mathcal{D},x}=\textrm{orthogonal complement of }T^v_{\mathcal{D},x}.
$$
From \cite[Lemma 12.5.2]{cmsp}, we know that the $\pm1$-eigenspaces of the Cartan involution on $\mathfrak{g}$ are given by $\mathfrak{h}=\mathfrak{h}_0\otimes\mathbb{C}=\bigoplus_{j\textrm{ even}}\mathfrak{g}^{-j,j}$, and $\mathfrak{p}=\mathfrak{p}_0\otimes\mathbb{C}=\bigoplus_{j\textrm{ odd}}\mathfrak{g}^{-j,j}$. Moreover, there are canonical identifications $T^v_{\mathcal{D},x}=\mathfrak{h}_0/\mathfrak{k}_0$, and $T^h_{\mathcal{D},x}=\mathfrak{p}_0$. By \cite[Proposition 12.5.3]{cmsp}, the tangent bundle $T_{\mathcal{D}}$ canonically decomposes as $T_{\mathcal{D}}=T^v_{\mathcal{D}}\oplus T^h_{\mathcal{D}}$ into vertical and horizontal components. Both components are homogeneous vector bundles on $\mathcal{D}$ for the adjoint action of $K_0$ on $\mathfrak{g}_0$. They can be written as
\begin{align*}
T^v_{\mathcal{D}}=G_0\times_{K_0}\mathfrak{h}_0/\mathfrak{k}_0,\;\;\; T^h_{\mathcal{D}}=G_0\times_{K_0}\mathfrak{p}_0.
\end{align*}
The vertical tangent bundle $T^v_{\mathcal{D}}$ is holomorphic, because its fibers are complex submanifolds of $\mathcal{D}$ (see \cite[Problem 4.4.2]{cmsp}). The horizontal tangent bundle is in general not holomorphic. The subbundle $G_0\times_{K_0}\mathfrak{g}^{-1,1}$ of the complexification $T^h_{\mathcal{D}}\otimes\mathbb{C}$ has the structure of a holomorphic vector bundle, and is called the \emph{holomorphic horizontal tangent bundle} of $\mathcal{D}$.\\
Now suppose that $G_0$ is a Hodge group of Hermitian type. Then $K_0$ is a maximal compact subgroup of $G_0$, and $\mathcal{D}=G_0/K_0$ is the associated Hermitian symmetric space of non-compact type. While there may be several $G_0$ corresponding to a fixed $\mathcal{D}$, their connected components are all isogenous, so the complexified Lie algebra $\mathfrak{g}$ and its Hodge decomposition are determined (see \cite[Section II, Corollary 3.30]{milne}). 
\begin{lemma}\label{lemtan}
    The holomorphic tangent bundle $\mathcal{T}_{\mathcal{D}}$ of the Hermitian symmetric space $\mathcal{D}$ can be written as
    \begin{align*}
        \mathcal{T}_{\mathcal{D}}\cong P'\times_K\mathfrak{g}^{-1,1},
    \end{align*}
    where $P'$ is a principal $K$-bundle on $\mathcal{D}$.
\end{lemma}
\begin{proof}
Let $Q^+$ and $Q^-$ denote the Lie subgroups of $G$ corresponding to the Lie subalgebras $\mathfrak{g}^{-1,1}$ and $\mathfrak{g}^{1,-1}$ of $\mathfrak{g}$ respectively. Then $Q^+$ and $Q^-$ are abelian unipotent subgroups of $G$ stabilized by conjugation by $K$ (\cite[Section 2.4]{herm}). There is a Zariski open embedding \begin{align*}
    j:\mathcal{D}\cong G_0/K_0\hookrightarrow G/(K\ltimes Q^-)\cong\mathcal{D}^*
\end{align*}
called the Borel embedding (see \cite[Chapter VIII, Section 7]{helga}), where $\mathcal{D}^*$ is a complex homogeneous projective variety known as the \emph{compact dual} of the Hermitian symmetric space $\mathcal{D}$. By the discussion in the preceding paragraph, the holomorphic tangent bundle of $\mathcal{D}^*$ can be expressed as $\mathcal{T}_{\mathcal{D}^*}\cong G\times_{(K\ltimes Q^-)}\mathfrak{g}^{-1,1}$, where $K\ltimes Q^-$ acts on $\mathfrak{g}^{-1,1}$ via the adjoint action. The adjoint action of $Q^-$ on $\mathfrak{g}^{-1,1}$ is trivial, i.e., $qXq^{-1}=0$ for all $q\in Q^-$, $X\in\mathfrak{g}^{-1,1}$. Thus the adjoint action of $K\ltimes Q^-$ on $\mathfrak{g}^{-1,1}$ factors through $K$. Hence $\mathcal{T}_{\mathcal{D}^*}$ admits a reduction in structure group from $K\ltimes Q^-$ to $K$, and we can write $\mathcal{T}_{\mathcal{D}^*}\cong P\times_K\mathfrak{g}^{-1,1}$, where $P$ is a principal $K$-bundle on $\mathcal{D}^*$ such that $P\times_K(K\ltimes Q^-)\cong G$. It follows that $P\cong G/Q^-$ as principal $K$-bundles over $\mathcal{D}^*$. Expressing the holomorphic tangent bundle $\mathcal{T}_{\mathcal{D}}$ of $\mathcal{D}$ as the restriction of $\mathcal{T}_{\mathcal{D}^*}$ to $\mathcal{D}$, we get $\mathcal{T}_{\mathcal{D}}\cong P'\times_K\mathfrak{g}^{-1,1}$, where $P'=P|_{\mathcal{D}}$ is principal $K$-bundle on $\mathcal{D}$.
\end{proof}
Hence, it is clear that the holomorphic tangent bundle of a Hermitian symmetric space of non-compact type is horizontal.\\
If $X$ admits a uniformizing variation of Hodge structure for a Hodge group $G_0$ of Hermitian type, then we know from the proof of \cite[Proposition 9.1]{simp} (Proposition \ref{simprop} below), that the $\pi_1(X)$-equivariant map $\widetilde{X}\to G_0/K_0=\mathcal{D}$ is an isomorphism. The differential of this map is an isomorphism $\theta:\mathcal{T}_{\widetilde{X}}\cong P'\times_K\mathfrak{g}^{-1,1}$.

\subsection{The tangent bundle of \texorpdfstring{$X=\mathcal{D}/\Gamma$}{}}

We would now like to invert the previous construction. Let $X$ be a smooth projective variety with universal cover $\widetilde{X}$. Suppose there is a holomorphic isomorphism $\phi:\widetilde{X}\to\mathcal{D}$, where $\mathcal{D}$ is a Hermitian symmetric space of noncompact type. Let $\pi:\widetilde{X}\to X$ denote the projection map. Let $G_0=\textrm{Aut}(\mathcal{D})$ be the full automorphism group of $\mathcal{D}$, then $G_0$ is a Hodge group of Hermitian type with maximal compact subgroup $K_0$, and we have $\mathcal{D}=G_0/K_0$.
\begin{lemma}\label{lemuni}
    In this situation, the variety $X$ admits a uniformizing variation of Hodge structure $(P,\theta)$ for the Hodge group $G_0$.
\end{lemma}
\begin{proof}
The differential of the holomorphic isomorphism $\phi:\widetilde{X}\cong\mathcal{D}$ is the isomorphism $d\phi:\mathcal{T}_X\cong\mathcal{T}_{\mathcal{D}}$ of holomorphic tangent bundles. From Lemma \ref{lemtan} it follows that $\mathcal{T}_{\widetilde{X}}\cong \widetilde{P}\times_K\mathfrak{g}^{-1,1}$, where $\widetilde{P}$ is a prinicipal $K$-bundle on $\widetilde{X}$. Recall that, by \cite[Lemma-Definition 12.2.3]{cmsp}, we have $\mathcal{T}_{\mathcal{D}}\cong G_0\times_{K_0}\mathfrak{g}^{-1,1}$. Thus it follows that $\mathcal{T}_{\widetilde{X}}\cong\widetilde{P}'\times_{K_0}\mathfrak{g}^{-1,1}$, where $\widetilde{P}'=\phi^*G_0$ is a principal $K_0$-bundle on $\widetilde{X}$. Hence, we may view $\widetilde{P}'$ as a reduction in structure group of $\widetilde{P}$ from $K$ to $K_0$.\\
There is natural principal $G_0$ bundle $\widetilde{R}$ on $\mathcal{D}$ called the \emph{Higgs principal bundle} (see \cite[Definition 12.4.1]{cmsp}). It is given by 
\begin{align*}
    \widetilde{R}=G_0\times_{K_0}G_0\cong\mathcal{D}\times G_0,
\end{align*}
where the isomorphism is a $G_0$-equivariant map given by $[g,g']\mapsto([g],gg')$. The \emph{Higgs connection} $\omega_H$ on $\widetilde{R}$ is obtained by pulling back the Maurer-Cartan form on $G_0$ to $\mathcal{D}\times G_0$, and $\omega_H$ is flat. Thus we get a principal $G_0$-bundle $\widetilde{R}'$ on $\widetilde{X}$ given by $\widetilde{R}'=\phi^*\widetilde{R}=\phi^*(G_0\times_{K_0}G_0)\cong \widetilde{P}'\times_{K_0}G_0$. The connection $\omega'_H=\phi^*\omega_H$ on $\widetilde{R}'$ is flat, and there is a trivialisation $\widetilde{R}'\cong\widetilde{X}\times G_0$.
The fundamental group $\pi_1(X)$ acts on $\widetilde{X}\cong\mathcal{D}$ by holomorphic automorphisms, so we have a representation $\sigma:\pi_1(X)\to G_0$. The isomorphism $\phi:\widetilde{X}\cong G_0/K_0$ is equivariant under $\sigma$ i.e., we have $\phi(\gamma x)=\sigma(\gamma)\phi(x)$ for all $x\in\widetilde{X}$ and $\gamma\in\pi_1(X)$. It follows that the differential $d\phi$ is also equivariant under $\sigma$. The action of $\pi_1(X)$ on $\widetilde{X}$ lifts to a left action of $\pi_1(X)$ on the principal $K$-bundle $\widetilde{P}$. Thus there is a left action of $\pi_1(X)$ on the associated bundle $\mathcal{T}_{\widetilde{X}}$, and we have $\mathcal{T}_X\cong\mathcal{T}_{\widetilde{X}}/\pi_1(X)$. By the equivariance of $d\phi$ under the representation $\sigma$, we have an isomorphism $\theta:\mathcal{T}_X\cong P\times_K\mathfrak{g}^{-1,1}$, where $P=\widetilde{P}/\pi_1(X)$ is a principal $K$-bundle on $X$. The pair $(P,\theta)$ is a uniformizing system of Hodge bundles on $X$ for the Hodge group $G_0$. Since we also have $\mathcal{T}_{\widetilde{X}}\cong\widetilde{P}'\times_{K_0}\mathfrak{g}^{-1,1}$, it follows that $\mathcal{T}_X\cong P'\times_{K_0}\mathfrak{g}^{-1,1}$, where $P'=\widetilde{P}'/\pi_1(X)$ is a principal $K_0$ bundle on $X$, and $P'\times_{K_0}K\cong P$. Note that $P'$ is a metric for $(P,\theta)$.\\
Recall that $\widetilde{R}'=\widetilde{P}'\times_{K_0}G_0$ is a flat principal $G_0$-bundle on $\widetilde{X}$, and we can associate to it the system of Hodge bundles $\widetilde{E}=\widetilde{R}'\times_{G_0}\mathfrak{g}$ via the adjoint action of $G_0$ on $\mathfrak{g}$. We can view $\widetilde{E}$ as a $K_0$-bundle $\widetilde{E}'=\widetilde{P}'\times_{K_0}\mathfrak{g}$. As a $K_0$-bundle $\widetilde{E}'$ decomposes as a direct sum $\widetilde{E}'=\bigoplus_{i\in\{-1,0,1\}}\widetilde{P}'\times_{K_0}\mathfrak{g}^{i,-i}$. The flat connection $\omega_H$ on $\widetilde{R}'$ induces a flat connection $\omega'_H$ on the associated bundle $\widetilde{E}$. From \cite[Proposition 13.1.1]{cmsp}, there is a decomposition $\omega'_H=\widetilde{d}_H+\sigma+\Bar{\sigma}$, where $\widetilde{d}_H$ is the Chern connection for the Hodge metric on $\widetilde{E}'$, $\sigma\in\mathcal{A}^{1,0}(\textrm{End}(\widetilde{E}))$, and $\Bar{\sigma}$ is the adjoint of $\sigma$ with respect to the Hodge metric. By slight abuse of notation, let $\omega_H=\widetilde{d}_H+\sigma+\Bar{\sigma}$ be the associated splitting as principal connections.   \\
There is an action of $\pi_1(X)$ on $\widetilde{R}$ by automorphisms, which comes from the representation $\sigma:\pi_1(X)\to G_0$. This action is described explicitly in the proof of \cite[Theorem 2.9]{morita}. The quotient $R'=(\widetilde{P}'\times_{K_0} G_0)/\pi_1(X)=P'\times_{K_0}G_0$ is a flat principal $G_0$-bundle on $X$ by the correspondence of Theorem \ref{monthm}.\\ 
The system of Hodge bundles $E=P\times_K\mathfrak{g}$ on $X$ admits a Hermitian metric $H$ corresponding to the reduction of structure group $P'$ of $P$ from $K$ to $K_0$. Let $d'_H$ be the $K_0$-connection on $P'$ which induces the Chern connection on $E$ for $H$. Let $\theta'$ denote the connection on $E$ corresponding to $\theta$, $\Bar{\theta}'$ the adjoint of $\theta'$ with respect to $H$, and $\sigma'$, $\Bar{\sigma}'$ the $G_0$ connections on $R'$ which induce $\theta'$, $\Bar{\theta}'$ respectively. Then $D_H=d'_H+\sigma'+\Bar{\sigma}'$ is a $G_0$-connection on $R'$, and the connections $D_H,d'_H,\sigma',\Bar{\sigma}'$ pull back to $\widetilde{D}_H,\widetilde{d}_H,\sigma,\Bar{\sigma}$ respectively. Recall that $\widetilde{D}_H$ is a flat connection on $\widetilde{R}'$. Since flatness can be checked locally, and $X$ and $\widetilde{X}$ are locally diffeomorphic, it follows that $D_H$ is a flat connection on $R'$. \\
Thus, $(P,\theta)$ together with the metric $P'$ is a principal variation of Hodge structure on $X$ for the Hodge group $G_0$. It is in fact a uniformizing variation because the differential $\theta$ is an isomorphism.
\end{proof}

\subsection{Classical results of Simpson}

In the case when $X$ is a smooth compact complex manifold, Simpson derives the following necessary and sufficient conditions for $X$ to be uniformized by a Hermitian symmetric space $\mathcal{D}$ of noncompact type.
\begin{proposition}[{\cite[Proposition 9.1]{simp}}]\label{simprop}
Let $X$ be a smooth compact complex manifold and let $\widetilde{X}$ be the universal cover of $X$. Then $\widetilde{X}$ is isomorphic to the bounded symmetric domain $\mathcal{D}$ if and only if $X$ admits a uniformizing variation of Hodge structure for some Hodge group $G_0$ with $\mathcal{D}=G_0/K_0$.
\end{proposition}
A more algebraic formulation of the above Proposition, which is more usefult in practice, is the following result. 

\begin{theorem}[{\cite[Theorem 2]{simp}}]\label{simpthm}
Let $X$ be a smooth compact complex manifold. Then $\widetilde{X}\cong\mathcal{D}$ if and only if there is a uniformizing system of Hodge bundles $(P,\theta)$ for a Hodge group $G_0$ of Hermitian type corresponding to $\mathcal{D}$, such that $(P,\theta)$ is stable and $c_2(P\times_K\mathfrak{g})\cdot[K_X]^{n-2}=0$.
\end{theorem}

From \cite[Corollary 9.4]{simp}, it follows that $(P,\theta)$ being stable is equivalent to the system of Hodge bundles $P\times_K\mathfrak{g}$ being polystable with respect to $K_X$ as a Higgs bundle. Theorem \ref{simpthm} turns out to be useful in formulating explicit necessary and sufficient conditions for a complex projective variety with klt singularities to be uniformized by each of the four classical Hermitian symmetric spaces $\mathcal{D}$ of non-compact type. 

\begin{remark}\label{remgp}
(a) To determine necessary conditions for a projective variety $X$ over $\mathbb{C}$ to be unifomized by a Hermitian symmetric space $\mathcal{D}$ of non-compact type, we must fix a Hodge group $G_0$ associated to $\mathcal{D}$. In general, choosing $G_0$ to be the connected component $\textrm{Aut}^0(\mathcal{D})$ of the automorphism group of $\mathcal{D}$ will not give necessary condiditons.\\ 
(b) However, we can choose the $G_0$ to be a cover of $Aut(\mathcal{D})$, such that the kernel of the covering map $\varphi:G_0\to\textrm{Aut}(\mathcal{D})$ is a discrete central subgroup of $G_0$. To why we can make this choice, note that $G_0$ and $Aut(\mathcal{D})$ have the same Lie algebra, and the isomorphism $G_0/K_0\cong Aut(\mathcal{D})/M_0$ is compatible with $\varphi$. Here $K_0$ and $M_0$ denote maximal compact subgroups of $G_0$ and $Aut(\mathcal{D})$ respectively. This means that the image of $K_0$ via $\varphi$ is $M_0$, and the kernel of $\varphi|_{K_0}$ is a discrete central subgroup $K_0$. Let $K$ and $M$ denote the complexifications of $K_0$ and $M_0$ respectively. Then, $M$ is a quotient of $K$ by a discrete central subgroup. Note that $K$ acts on the complexified Lie algebra $\mathfrak{g}$ of $G_0$ via the adjoint representation $Ad: K\to Aut(\mathfrak{g})$, $k\mapsto Ad_k=kXk^{-1}$, $X\in\mathfrak{g}$. The kernel of the adjoint representation $Ad$ is the center of $K$. It follows that the action of $K$ on $\mathfrak{g}$ factors through $M$. By Lemma \ref{lemuni}, the tangent bundle of $X$ can be expressed as $\mathcal{T}_X\cong P\times_M\mathfrak{g}^{-1,1}$, where $P$ is a principal $M$-bundle on $X$. Since the action of $K$ on $\mathfrak{g}^{-1,1}$ factors through $M$, the tangent bundle of $X$ can also be expressed as $\mathcal{T}_X\cong P'\times_K\mathfrak{g}^{-1,1}$, where $P'\cong P\times_MK$ is a principal $K$- bundle on $X$. 
\end{remark}

\section{Extending Simpson's result to the klt case}

In this section, our goal is to prove Theorem \ref{genthm}. Henceforth, we work in the klt setting instead of the smooth setting. We first make some observations which will help us. 

\subsection{Auxiliary remarks}

The following observation about the semistability of the tangent sheaf of a variety of general type is a slight generalization of a result of H.Guenancia (\cite[Theorem A]{guen}). It allows us to do away with the assumption that the tangent sheaf has negative degree, which was made by Simpson in his uniformisation result (\cite[Corollary 9.7]{simp}). 

\begin{proposition}[\cite{gkpt}, Theorem 7.1]\label{tanprop}
Let $X$ be a projective, klt variety of general type whose canonical divisor $K_X$ is nef. Then the sheaves $\mathcal{T}_X$ and $\Omega^{[1]}_X$ are semistable with respect to $K_X$.
\end{proposition}

Note that by \cite[Lemma 2.26]{naht}, this is equivalent to the vector bundles $\mathcal{T}_{X_{reg}}$ and $\Omega^1_{X_{reg}}$ on $X_{reg}$ being semistable with respect to $K_X$.\\
\\
An important observation due to A. Langer is that the stability of a system of Hodge sheaves is equivalent to the stability of the underlying Higgs sheaf. This was first proved for systems of Hodge sheaves over a smooth projective variety $X$ (see \cite[Proposition 8.1]{langer}). The proof uses that systems of Hodge sheaves are fixed points of the $\mathbb{C}^*$-action on the moduli space of Higgs sheaves (see \cite[Lemma 4.1]{simp2}) on $X$. Moreover, \cite[Lemma 6.8]{simpmod2} says that a torsion free Higgs sheaf $E$ on $X$ is the same thing as a pure coherent sheaf $\mathcal{E}$ of dimension $n=\textrm{dim}(X)$ on a projective completion $Z$ of the cotangent bundle, such that the support of $\mathcal{E}$ does not meet the divisor at infinity. Then $\mathcal{E}$ being a $\mathbb{C}^*$-fixed point implies that the Quot scheme $Quot_Z(\mathcal{E})$ inherits a $\mathbb{C}^*$- action. Since $Quot_Z(\mathcal{E})$ is projective, the limit of the orbit of any point $\mathcal{F}\in Quot_Z(\mathcal{E})$ under the $\mathbb{C}^*$-action exists in $Quot_Z(\mathcal{E})$, and corresponds to a quotient system of Hodge sheaves of $E$ on $X$ that has the same numerical invarants as $\mathcal{F}$. \\
This result was generalized to the setting where $X$ is normal and projective (see \cite[Corollary 3.5]{langer2}). It is shown that even in this more general setting, a torsion free Higgs sheaf is a fixed point of the $\mathbb{C}^*$-action if and only if it is a system of Hodge sheaves. To prove the result it is sufficient to show that the maximally destabilizing Higgs subsheaf of a system of Hodge sheaves is a system of Hodge sheaves. This follows from the fact that the maximally destabilizing Higgs subsheaf is unique, so it must be a fixed point of the $\mathbb{C}^*$-action.\\    
We give the following independent proof of this result for systems of Hodge sheaves on the smooth locus of a projective variety $X$ with klt singularities. 

\begin{lemma}\label{lemsta}
Let $X$ be a projective variety with klt singularities. Then the stability conditions for a torsion free system of Hodge sheaves on the smooth locus $X_{reg}$ and the stability conditions of the underlying Higgs sheaf in the sense of Definition \ref{streg}, are equivalent.
\end{lemma}

\begin{proof}
Let $E=\bigoplus_{p=0}^nE^p$ be a torsion free system of Hodge sheaves on $X_{reg}$ with Higgs field $\theta$, i.e., we have $\theta:E^p\to E^{p-1}\otimes\Omega^1_{X_{reg}}$ for all $0\le p\le n$. Let $F\subset E$ be a Higgs subsheaf. The idea is to construct a subsystem of Hodge sheaves $F'=\bigoplus_{p=0}^nF'^p$ of $E$ such that $F$ and $F'$ have the same rank and first Chern class. For each $i$, define $G_i=F\cap\bigoplus_{p<i}E^p$, where $0\le i\le n+1$. So for example $G_0=0$, $G_1=F\cap E_0$, and $G_{n+1}=F\cap\bigoplus_{p<n+1}E^p=F$. Thus the $G_i's$ give a filtration of $F$
\begin{align*}
    0=G_0\subset G_1\subset\dots\subset G_n\subset G_{n+1}=F.
\end{align*}
Since $\theta(F)\subset F\otimes\Omega^1_{X_{reg}}$, and $\theta(\bigoplus_{p<i}E^p)\subset(\bigoplus_{p<i-1}E^p)\otimes\Omega^1_{X_{reg}}$, it follows that $\theta(G_i)\subset G_{i-1}\otimes\Omega^1_{X_{reg}}$ for all $0\le i\le n+1$. Let $\{F'^i\}_{i=0}^n$ be the quotients of this filtration, i.e., $F'^i=G_{i+1}/G_i$ for all $i$. Consider the sequence of maps
\begin{align*}
    G_{i+1}\hookrightarrow\bigoplus_{p<i+1}E^p\to E^i
\end{align*}
where the first map is the inclusion and the second is projection. The kernel of the composed map is $G_i$, so the image of $G_{i+1}$ in $E^i$ is isomorphic to the quotient $F'^i=G_{i+1}/G_i$. Thus $F'^i\subset E^i$ for all $0\le i\le n$. Moreover, since $\theta(G_{i})\subset G_{i-1}\otimes\Omega^1_{X_{reg}}$, and tensoring with $\Omega^1_{X_{reg}}$ is exact, it follows that $\theta(F'^i)\subset F'^{i-1}\otimes\Omega^1_{X_{reg}}$ for all $0\le i\le n$. Hence $F'=\bigoplus_{p=0}^nF'^p\subset\bigoplus_{p=0}^nE^p$ is a subsystem of Hodge sheaves. \\
To see that $F$ and $F'$ have the same numerical invariants, we look at the series of short exact sequences
\begin{align*}
    &0\to G_0=0\to G_1\to F'^0\to0\\
    &0\to G_1\to G_2\to F'^1\to0\\
    &\vdots\\
    &0\to G_{n-1}\to G_n\to F'^{n-1}\to0\\
    &0\to G_n\to G_{n+1}=F\to F'^n\to0
\end{align*}
From the first exact sequence it follows that $\textrm{rank}(G_1)=\textrm{rank}(F'^0)$. This implies that $\textrm{rank}(G_2)=\textrm{rank}(F'^0)+\textrm{rank}(F'^1)$, and repeating this gives $\textrm{rank}(G_i)=\textrm{rank}(F'^0)+\dots+\textrm{rank}(F'^{i-1})$. From the last exact sequence we get $\textrm{rank}(F)=\textrm{rank}(F'^0)+\dots+\textrm{rank}(F'^n)=\textrm{rank}(F')$. The same computation holds for first Chern classes, so we have $c_1(F)=c_1(F')$. 
\end{proof}

The above result is true also for reflexive systems of Hodge sheaves on a normal variety in the sense of \cite[Secton 4]{lan}. More precisely, the stability conditions for a reflexive system of Hodge sheaves and the underlying reflexive Higgs sheaf in the sense of \cite[Section 4]{lan} are equivalent. This is remarked in \cite[Section 4.10]{lan}.\\
\\
The semistability of the bundle $P\times_K\mathfrak{g}^{0,0}$ is necessary to show that the system of Hodge bundles $P\times_K\mathfrak{g}$ is polystable as a Higgs bundle in the sense of Definition \ref{streg}. This will play a central role in the proofs of statements that appear later. We observe the following.

\begin{lemma}\label{lem00}
Let $X$ be a projective, klt variety of general type, and let $(P,\theta)$ define a uniformizing system of Hodge bundles for a Hodge group $G_0$ on $X_{reg}$. Then the vector bundle $P\times_K\mathfrak{g}^{0,0}$ is semistable with respect to $K_X$. 
\end{lemma}

\begin{proof}
Let $\mathcal{T}_{X_{reg}}=\bigoplus V_i$ be a decomposition corresponding to irreducible representations of $K\to\textrm{Aut}(\mathfrak{g}^{-1,1})$. First suppose that the group $G$ is connected. Let $\mathfrak{g}=\bigoplus\mathfrak{g}_i$ be the decomposition into simple ideals. Then $\mathfrak{g}^{-1,1}=\bigoplus\mathfrak{g}_i^{-1,1}$ is the decomposition into irreducible representations of $K$. Since $\mathcal{T}_{X_{reg}}$ is semistable, it follows that the $V_i\cong P\times_K\mathfrak{g}_i^{-1,1}$ and their duals $P\times_K\mathfrak{g}_i^{1,-1}$ are also semistable. Since the Lie bracket $\mathfrak{g}_i^{-1,1}\otimes\mathfrak{g}_i^{1,-1}\to\mathfrak{g}_i^{0,0}$ is surjective (see \cite[proof of Proposition 9.6]{simp}), we get a surjective map of vector bundles 
\begin{align*}
(P\times_K\mathfrak{g}_i^{-1,1})\otimes(P\times_K\mathfrak{g}_i^{1,-1})\to P\times_K\mathfrak{g}_i^{0,0},
\end{align*}
where the left hand side is a semistable vector bundle of degree zero and the right hand side is a vector bundle of degree zero. It follows that $P\times_K\mathfrak{g}_i^{0,0}$ is also semistable. Indeed, if it is not, then neither is the dual bundle $(P\times_K\mathfrak{g}_i^{0,0})^\vee$. Note that $(P\times_K\mathfrak{g}_i)^\vee$ is also of degree zero, and is a subbundle of $(P\times_K\mathfrak{g}_i^{-1,1})\otimes(P\times_K\mathfrak{g}_i^{1,-1})$. Thus, if some subsheaf of $(P\times_K\mathfrak{g}_i)^\vee$ destabilizes it, then it also destabilizes $(P\times_K\mathfrak{g}_i^{-1,1})\otimes(P\times_K\mathfrak{g}_i^{1,-1})$, which is a contradiction. Since $P\times_K\mathfrak{g}^{0,0}=\bigoplus_iP\times_K\mathfrak{g}_i^{0,0}$, it follows that $P\times_K\mathfrak{g}^{0,0}$ is also a semistable vector bundle of degree zero.\\
\\
Now suppose that $G$ is not connected. Let $G'$ be the connected component of $G$, let $K'=K\cap G'$, and let $f':Y'\to X_{reg}$ be a finite \'{e}tale cover which the structure group of $P$ can be reduced to $K'$. Note that $Y'$ can be completed to a projective, klt variety $Y$ such that $K_Y$ is ample, and there is a finite quasi-\'{e}tale map $f:Y\to X$, which restricts to $f'$ on $Y'$. Since the $V_i$'s are semistable with respect to $K_X$, the pullbacks $f'^*V_i$ are semistable with respect to $f^*K_X=K_Y$. If we decompose $\mathcal{T}_Y=\bigoplus_kV'_k$ corresponding to irreducible components of the representation $K'\to \textrm{Aut}(\mathfrak{g}^{-1,1})$, the $V'_k$'s are direct summands of the $f'^*V_i$, so they are also semistable with respect to $K_Y$. Since $V'_k\cong f'^*(P\times_K\mathfrak{g}_k^{-1,1})=f'^*P\times_{K'}\mathfrak{g}_k^{-1,1}$, it follows from the argument in the case that $G$ is connected, that $f'^*(P\times_K\mathfrak{g}^{0,0})$ is semistable with respect to $K_Y$ of degree zero. Since $f'$ is finite \'{e}tale, it follows that $P\times_K\mathfrak{g}^{0,0}$ is also semistable with respect to $K_X$ of degree zero. 
\end{proof}

The following observation is at the heart of the proof of Theorem \ref{genthm}, and we again split the proof into two cases, namely $G$ connected and disconnected.

\begin{proposition}\label{propstab}
    Let $X$ be a projective, klt variety with ample canonical divisor $K_X$ and let $(P,\theta)$ define a uniformizing system of Hodge bundles for a Hodge group $G_0$ on $X_{reg}$. Then the system of Hodge bundles $P\times_K\mathfrak{g}$ is $K_X$-polystable as a Higgs bundle on $X_{reg}$. If $\mathfrak{g}$ is a simple Lie agebra, then it is $K_X$-stable. 
\end{proposition}
\begin{proof}
First, suppose that the chosen complexification $G$ of $G_0$ is connected. By assumption, there is an isomorphism of vector bundles $\theta:\mathcal{T}_{X_{reg}}\cong P\times_K\mathfrak{g}^{-1,1}$ on $X_{reg}$. Write $\mathcal{T}_{X_{reg}}=\bigoplus_iV_i$, where each $V_i$ corresponds to an irreducible component of the representation $K\hookrightarrow \textrm{Aut}(\mathfrak{g^{-1,1}})$. Recall from the comment following Proposition \ref{tanprop} that $\mathcal{T}_{X_{reg}}$ is semistable with respect to $K_X$, and since $K_X$ is ample by assumption, $\mathcal{T}_{X_{reg}}$ has negative slope. It follows that all direct summands $V_i$ are semistable with respect to $K_X$ and have the same negative slope as $\mathcal{T}_{X_{reg}}$.\\
Let $\mathfrak{g}=\bigoplus_i\mathfrak{g}_i$ be a decomposition of $\mathfrak{g}$ as a direct sum of simple ideals. Then each $P\times_K\mathfrak{g}_i$ is a subsystem of Hodge bundles of $P\times_K\mathfrak{g}$ (see \cite[p.903]{simp}), and we have a decomposition $P\times_K\mathfrak{g}=\bigoplus_iP\times_K\mathfrak{g}_i$. Let $\mathfrak{g}^{-1,1}=\bigoplus_i\mathfrak{g}_i^{-1,1}$ be the decomposition into irreducible representations of $K\subset \textrm{Aut}(\mathfrak{g^{-1,1}})$. Since $V_i\cong P\times_K\mathfrak{g}_i^{-1,1}$ for all $i$, the bundles $P\times_K\mathfrak{g}_i^{-1,1}$ and their duals $P\times_K\mathfrak{g}_i^{1,-1}$ are semistable with respect to $K_X$. From the proof of Lemma \ref{lem00}, we know that the bundle  $P\times_K\mathfrak{g}_i^{0,0}$ is semistable with respect to $K_X$ of degree zero. Thus each $P\times_K\mathfrak{g}_i$ is a system of Hodge bundles on $X_{reg}$ of degree zero.\\
By Lemma \ref{lemsta}, to show that $P\times_K\mathfrak{g}$ is polystable as a Higgs bundle, it is sufficient to show that it is polystable as a system of Hodge bundles on $X_{reg}$. Let $W\subset P\times_K\mathfrak{g}_i$ be a subsystem of Hodge sheaves. Then we know from \cite[Corollary 9.3]{simp}, that $\textrm{rank}(W^{-1,1})\ge\textrm{rank}(W^{1,-1})$, and if equality holds then $W=P\times_K\mathfrak{g}_i$. This implies that $\mu_{K_X}(W)\le0$, and if equality holds then $W=P\times_K\mathfrak{g}^{-1,1}$. Thus each $P\times_K\mathfrak{g}_i$ is $K_X$-stable as a Higgs bundle, and since $P\times_K\mathfrak{g}=\bigoplus_iP\times_K\mathfrak{g}_i$, it follows that $P\times_K\mathfrak{g}$ is $K_X$-polystable as a Higgs bundle on $X_{reg}$, of degree zero. Note that if $\mathfrak{g}$ is a simple Lie algebra, then the only non trivial ideal of $\mathfrak{g}$ is itself. So in this case, $P\times_K\mathfrak{g}$ is even $K_X$-stable as a Higgs bundle on $X_{reg}$.\\
\\
Now suppose $G$ is not connected.  Let $G'$ denote the connected component of $G$ and let $K'=K\cap G'$. Let $f:Y'\to X_{reg}$ be a finite \'{e}tale cover over which the structure group of $P$ can be reduced from $K$ to $K'$, and let $Y$ be a normal projective completion of $Y'$ such that $f$ can be extended to a quasi-\'{e}tale map $\gamma:Y\to X$. Note that $Y$ is again klt, $Y'\subset Y_{reg}$ is a big open subset, and $K_Y=\gamma^*K_X$ is still ample. Thus $\mathcal{T}_{Y'}$ is semistable with respect to $K_Y$ of negative degree, and if we write $\mathcal{T}_{Y'}=\bigoplus_kV'_k$ corresponding to irreducible components of the representation $K'\subset\textrm{Aut}(\mathfrak{g}^{-1,1})$, then each $V'_k$ is semistable with respect to $K_Y$, of negative degree.\\
Let $P'$ be the unique extension of $f^*P$ to $Y_{reg}$ as a holomorphic principal $K'$-bundle, which we know exists by \cite[Theorem 1.1]{biswas}. Then by the $G$ connected case, the system of Hodge bundles $P'\times_{K'}\mathfrak{g}$ is $K_Y$-polystable as a Higgs bundle on $Y_{reg}$. Therefore, $f^*(P\times_K\mathfrak{g})$ is $K_Y$-polystable as a Higgs bundle on $Y'$.\\
If $W\subset P\times_K\mathfrak{g}$ is any saturated subsystem of Hodge sheaves with $\textrm{deg}(W)\ge0$, then $f^*W$ is a subsystem of Hodge sheaves of $f^*(P\times_K\mathfrak{g})$ of degree zero, and is therefore a direct summand of $f^*(P\times_K\mathfrak{g})$. By \cite[Corollary 9.3]{simp}, $f^*W$ is locally a direct sum of simple ideals of $\mathfrak{g}$, thus the same holds for $W$ on $X_{reg}$. From the discussion in \cite[p.903]{simp}, there is a unique finest global decomposition 
\begin{align}\label{decomp}
P\times_K\mathfrak{g}=\bigoplus_j\mathcal{E}_j,
\end{align}
where each $\mathcal{E}_j$ is a subsystem of Hodge bundles of $P\times_K\mathfrak{g}$ of degree 0. Moreover, each $\mathcal{E}_j$ is locally a direct sum of simple ideals of $\mathfrak{g}$. Now if $W\subset\mathcal{E}_j$ is any saturated subsystem of Hodge sheaves of degree $\ge0$, then we know that $W$ is a locally a direct sum of simple ideals of $\mathfrak{g}$, so we have $W=\mathcal{E}_j$ by minimality of the decomposition (\ref{decomp}). Thus $\mathcal{E}_j$ is stable with respect to $K_X$ as a system of Hodge bundles. It follows that $P\times_K\mathfrak{g}$ is $K_X$-polystable as a system of Hodge bundles, which is equivalent to $P\times_K\mathfrak{g}$ being $K_X$-polystable as a Higgs bundle on $X_{reg}$, by Lemma \ref{lemsta}.
\end{proof}

We are now ready to prove Theorem \ref{genthm}.

\subsection{Proof of Theorem \ref{genthm}}

We split the proof into two steps, in order to make it more readable. The first step is to show that if $X\cong\mathcal{D}/\Gamma$, then $X$ satisfies the two conditions of Theorem \ref{genthm}. The second step is to prove the converse. 

\begin{proof}
\textbf{Step I}.
Suppose we can write $X\cong\mathcal{D}/\Gamma$, for a Hermitian symmetric space $\mathcal{D}$ of noncompact type. Let $G_0=\textrm{Aut}(\mathcal{D})$ be the full automorphism group of $\mathcal{D}$, then $\Gamma\subset G_0$. In this case $G_0$ is a Hodge group of Hermitian type, and $K_0$ is a maximal compact subgroup of $G_0$. Let $G$ and $K$ denote complexifications of $G_0$ and $K_0$ respectively. Since the smooth locus $X_{reg}$ is (the analytic space associated to) a quasi-projective variety, its fundamental group $\pi_1(X_{reg})$ is finitely generated, and isomorphic to $\Gamma$. Then from Selberg's lemma (see \cite{alperin}), it follows that $\Gamma$ admits a normal torsion free subgroup $\widehat{\Gamma}$ of finite index. Thus the quotient map $\mathcal{D}\to\mathcal{D}/\Gamma=X$ factors as 
\begin{align}\label{mapfac}
    \mathcal{D}\to\mathcal{D}/\widehat{\Gamma}\to\mathcal{D}/\Gamma=X,
\end{align}
where $\mathcal{D}/\widehat{\Gamma}=Y$ is smooth and projective because $\widehat{\Gamma}$ acts freely and cocompactly on $\mathcal{D}$. The map $\gamma:Y\to X$ is quasi-\'{e}tale and Galois with group $H=\Gamma/\widehat{\Gamma}$. Since $Y$ is uniformized by $\mathcal{D}$, it follows from the classical Theorem \ref{simpthm} that $Y$ admits a uniformizing system of Hodge bundles $(P',\theta)$ for the Hodge group $G_0$, such that $c_1(P'\times_K\mathfrak{g})\cdot[K_Y]^{n-1}=c_2(P'\times_K\mathfrak{g})\cdot[K_Y]^{n-2}=0$. There is hence an isomorphism of vector bundles $\mathcal{T}_Y\cong P'\times_K\mathfrak{g}^{-1,1}$ on $Y$.\\
By the purity of branch locus, $\gamma:Y\to X$ branches only over the singular locus of $X$. Let $Y^o=\gamma^{-1}(X_{reg})$. Then $\gamma|_{Y^o}:Y^o\to X_{reg}$ is \'{e}tale, and $Y^o\subset Y$ is a big open subset. The action of $H$ on $Y$ restricts to a free action on $Y^o$ and lifts to a left action on $\mathcal{T}_{Y^o}\cong P''\times_K\mathfrak{g}^{-1,1}$, where $P''=P'|_{Y^o}$. On $X_{reg}$, we have $\mathcal{T}_{X_{reg}}\cong\mathcal{T}_{Y^o}/H$, i.e., there is an isomorphism $\theta:\mathcal{T}_{X_{reg}}\cong P\times_K\mathfrak{g}^{-1,1}$, where $P\cong P''/H$ is a principal $K$-bundle on $X_{reg}$. Thus $(P,\theta)$ is a uniformizing system of Hodge bundles on $X_{reg}$.\\
Let $\mathcal{E}$ be the system of Hodge bundles $P\times_K\mathfrak{g}$ on $X_{reg}$, and let $\mathcal{E}'$ denote the unique extension of $\mathcal{E}$ to $X$ as a reflexive sheaf. Indeed, such an extension exists because, in view of the factorization \ref{mapfac}, $(\gamma|_{Y^o})^*\mathcal{E}$ extends to $Y=\mathcal{D}/\widehat{\Gamma}$ and is hence algebraic. Then, $\gamma^{[*]}\mathcal{E}'\cong P'\times_K\mathfrak{g}$ on $Y$, because $\gamma^{[*]}\mathcal{E}$ and $P'\times_K\mathfrak{g}$ are both reflexive and agree on the big open subset $Y^o$ of $Y$. Moreover, we have $K_Y=\gamma^*K_X$. From the behaviour of $\mathbb{Q}$-Chern classes under quasi-\'{e}tale covers (\cite[Proposition 3.17]{gkpt}), it follows that
\begin{align*}
&(a)\;c_1(\mathcal{E}')\cdot[K_X]^{n-1}=(\textrm{deg}(\gamma))^{-1}c_1(\gamma^{[*]}\mathcal{E}')\cdot[\gamma^*K_X]^{n-1}=(\textrm{deg}(\gamma))^{-1}c_1(P'\times_K\mathfrak{g})\cdot[K_Y]^{n-1} \\
&(b)\;\widehat{c}_2(\mathcal{E}')\cdot[K_X]^{n-2}=(\textrm{deg}(\gamma))^{-1}\widehat{c}_2(\gamma^{[*]}\mathcal{E}')\cdot[\gamma^*K_X]^{n-2}=(\textrm{deg}(\gamma))^{-1}c_2(P'\times_K\mathfrak{g})\cdot[K_Y]^{n-2}.
\end{align*}
The right hand sides of equalities $(a)$ and $(b)$ are zero, again by Theorem \ref{simpthm}. Thus it follows that $\widehat{ch}_2(\mathcal{E}')\cdot[K_X]^{n-2}=0$. This proves one implication of Theorem \ref{genthm}.\\ 
\\
\textbf{Step II}.
Conversely, suppose that $X$ is a projective klt variety which satisfies the two conditions of Theorem \ref{genthm}. Let $G_0$, $K_0$, $G$, and $K$ be as earlier. By assumption, we have an isomorphism of vector bundles $\theta:\mathcal{T}_{X_{reg}}\cong P\times_K\mathfrak{g}^{-1,1}$ on $X_{reg}$. \\
Let $\gamma:Y\to X$ be a Galois, maximally quasi-\'{e}tale cover, the existence of which is known (see \cite[Theorem 1.14]{gkp1}). Note that $Y$ is again klt (see \cite[Proposition 5.20]{kolmori}), and $K_Y=\gamma^*K_X$ is ample. Since $\gamma$ branches only over the singular locus of $X$, the restricted map $\gamma|_{Y^o}:Y^o\to X_{reg}$ is \'{e}tale. Here $Y^o=\gamma^{-1}(X_{reg})\subset Y_{reg}$, and note that $Y^o$ is a big open subset of $Y$. Then $\mathcal{T}_{Y^o}=(\gamma|_{Y^o})^*\mathcal{T}_{X_{reg}}\cong(\gamma|_{Y^o})^*(P\times_K\mathfrak{g}^{-1,1})=P'\times_K\mathfrak{g}^{-1,1}$, where $P'=(\gamma|_{Y^o})^*P$ is a principal $K$-bundle on $Y^o$. We can uniquely extend $P'$ to $Y_{reg}$ as a holomorphic principal $K$-bundle (see \cite[Theorem 1.1)]{biswas}),  which by abuse of notation, we denote again by $P'$. There is an isomorphism of vector bundles $\theta':\mathcal{T}_{Y_{reg}}\cong P'\times_K\mathfrak{g}^{-1,1}$ on $Y_{reg}$. Note that $\mathcal{T}_{Y_{reg}}$ has negative degree, and is semistable with respect to $K_Y$ again by the comment following Remark 4.1.\\
Let $\mathcal{F}=P'\times_K\mathfrak{g}$, then $\mathcal{F}$ is a system of Hodge bundles on $Y_{reg}$ with Higgs field coming from $\theta'$. By Proposition \ref{propstab}, we know that $\mathcal{F}$ is $K_Y$-polystable as a Higgs bundle on $Y_{reg}$, and of degree zero. Let $\mathcal{F}'$ and $\mathcal{E}'$ be the unique extensions of $\mathcal{F}$ to $Y$ and $P\times_K\mathfrak{g}$ to $X$ respectively, as reflexive sheaves. Since $\gamma^{[*]}\mathcal{E}'$ and $\mathcal{F}'$ agree on the big open subset $Y^o$, we have $\mathcal{F}'\cong\gamma^{[*]}\mathcal{E}'$. By assumption, $\widehat{ch}_2(\mathcal{E}')\cdot[K_X]^{n-2}=0$ holds, and by the behaviour of $\mathbb{Q}$-Chern classes under quasi-\'{e}tale covers, we have $\widehat{ch}_2(\mathcal{F}')\cdot[K_Y]^{n-2}=0$. From \cite[Theorem 5.1]{gkpt2}, it follows that $(\mathcal{F},\theta')\in\textrm{TPI-Higgs}_{Y_{reg}}$. Then \cite[Proposition 3.17]{gkpt2}, together with $Y$ being maximally quasi-\'{e}tale implies that $\mathcal{F}'$ is locally free. Since $\mathcal{T}_Y$ is a direct summand of $\mathcal{F}'$, it follows that $\mathcal{T}_Y$ is locally free. The solution to the Lipman-Zariski conjecture for klt spaces (see \cite[Theorem 16.1]{gkkp}) asserts that $Y$ is smooth. We may thus apply Simpson's classical Theorem \ref{simpthm} to conclude that $Y$ is uniformized by $G_0/K_0=\mathcal{D}$.\\
This means that $Y\cong\mathcal{D}/\Gamma'$, where $\Gamma'$ is a discrete, cocompact, torsion free group of automorphisms of $\textrm{Aut}(\mathcal{D})$, and $X\cong Y/H$, where $H$ is a finite group of automorphisms of $Y$ acting fixed point freely in codimension one. By arguments analogous to those in the proof of \cite[Theorem 1.3]{gkpt}, (specifically, step (1.3.3)$\implies$(1.3.1)), there is an exact sequence of groups $1\to\Gamma'\to\Gamma\to H\to1$, where $\Gamma$ is a discrete, cocompact subgroup of $\textrm{Aut}(\mathcal{D})$ acting properly discontinuously, and fixed point freely in codimension one, such that $X\cong\mathcal{D}/\Gamma$. 
\end{proof}




\section{Uniformization by a polydisk}

We are now ready to state the first uniformisation result, which is an application of Theorem \ref{genthm}. This is an extension of \cite[Corollary 9.7]{simp}, to the klt setting. We make the steps of the proof of Theorem \ref{genthm} explicit in this section. 

\subsection{Sufficient conditions}

\begin{theorem}\label{thm1}
Let $X$ be an $n$-dimensional projective klt variety of general type whose canonical divisor $K_X$ is big and nef. Suppose that the tangent sheaf $\mathcal{T}_{X_{reg}}$ of the smooth locus $X_{reg}$ splits as a direct sum
\begin{align}
    \mathcal{T}_{X_{reg}}=\mathcal{L}_1\oplus\dots\oplus\mathcal{L}_n \label{split2}
\end{align}
where each direct summand $\mathcal{L}_i$ is a line bundle. Then the canonical model $X_{can}$ of $X$ admits a Galois, quasi-\'{e}tale cover $\gamma:Y\to X_{can}$, where $Y$ is a smooth projective variety whose universal cover is the polydisk $\mathbb{H}^n$. 
\end{theorem}

In order to prove Theorem \ref{thm1}, we would like to define the Atiyah class of a reflexive sheaf $\mathcal{L}$ of rank one on $X$. This turns out to be the extension class of an exact sequence, whose restriction to $X_{reg}$ is the exact sequence corresponding to the classical Atiyah class $at(\mathcal{L}|_{X_{reg}})$ of the line bundle $\mathcal{L}|_{X_{reg}}$. This construction is used to prove an analog of \cite[Lemma 3.1]{bea}, and may also be of independent interest. \\
We denote by $\mathcal{E}nd(\mathcal{F})$ the endomorphism sheaf of a coherent sheaf $\mathcal{F}$, and by $\mathcal{D}^1(\mathcal{F})$ the sheaf of differential operators of $\mathcal{F}$ of order $\le1$. 

\begin{proposition}\label{pex}
Let $X$ be a projective klt variety of dimension $n$ and let $\mathcal{L}$ be a line bundle on the smooth locus $X_{reg}$. Consider the following short exact sequence of sheaves
\begin{align}\label{exseq1}
    0\to \mathcal{E}nd(\mathcal{L})\to\mathcal{D}^1(\mathcal{L})\to\mathcal{T}_{X_{reg}}\to0
\end{align}
on $X_{reg}$, known as the Atiyah sequence of $\mathcal{L}$. Let $j:X_{reg}\to X$ be the natural inclusion map. Then, the sequence
\begin{align}\label{exseq2}
    0\to j_*\mathcal{E}nd(\mathcal{L})\to j_*\mathcal{D}^1(\mathcal{L})\to j_*\mathcal{T}_{X_{reg}}\to0
\end{align}
is exact on $X$.
\end{proposition}
\begin{proof}
Note that since $\mathcal{L}$ is a line bundle on $X_{reg}$, we have $\mathcal{E}nd(\mathcal{L})\cong\mathcal{O}_{X_{reg}}$. Moreover, since $\mathcal{O}_X$ and $j_*\mathcal{O}_{X_{reg}}$ are both reflexive and agree on $X_{reg}$, we have $j_*\mathcal{O}_{X_{reg}}\cong\mathcal{O}_X$. Let $Z\subset X$ be the subset of $X$ consisting of points which are not quotient singularities. Since $X$ is klt, it follows that $Z$ has codimension $\ge3$ in $X$. Let $X'=X\setminus Z$ denote the complement of $Z$ in $X$ and let $i:X_{reg}\to X'$ denote the inclusion map. Note that $X'$ is $\mathbb{Q}$-factorial and hence $i_*\mathcal{L}$ is a $\mathbb{Q}$-Cartier sheaf on $X'$.\\
The extension class $at(\mathcal{L})$ of the Atiyah sequence (\ref{exseq1}) on $X_{reg}$ is an element of $\textrm{Ext}^1(\mathcal{T}_{X_{reg}},\mathcal{O}_{X_{reg}})\cong H^1(X_{reg},\Omega^1_{X_{reg}})$, and coincides with the first Chern class $c_1(\mathcal{L})$ of $\mathcal{L}$. Since $i_*\mathcal{L}$ is $\mathbb{Q}$-Cartier on $X'$, we can associate to it a cohomology class $c_1(i_*\mathcal{L})\in H^1(X',\Omega^1_{X'})\cong \textrm{Ext}^1(\mathcal{O}_{X'},\Omega^1_{X'})$, which we call the first Chern class of $i_*\mathcal{L}$, (see \cite[Section 4]{gkp}). Moreover, $c_1(i_*\mathcal{L})$  corresponds to the extension class of a short exact sequence in $\textrm{Ext}^1(\mathcal{O}_{X'},\Omega^1_{X'})$ which is locally split, hence the dual sequence of this exact sequence is also exact and locally split (see \cite[Section 4]{gkp}). Since this dual sequence is precisely the pushforward along $i_*$ of the Atiyah sequence (\ref{exseq1}), we have a short exact sequence
\begin{align}\label{exseq3}
    0\to\mathcal{O}_{X'}\to i_*\mathcal{D}^1(\mathcal{L})\to\mathcal{T}_{X'}\to0
\end{align}
on $X'$. Let $i':X'\to X$ denote the inclusion map. To complete the proof we want the pushforward of the exact sequence (\ref{exseq3}) along $i'_*$ to stay exact. \\
Since the direct image functor is left exact it follows that the sequence 
\begin{align*}
    0\to\mathcal{O}_X\to j_*\mathcal{D}^1(\mathcal{L})\to\mathcal{T}_X\to R^1j_*\mathcal{O}_{X_{reg}}
\end{align*}
is exact on $X$.\\
From \cite[Corollary 1.9]{loco}, we know that $\mathcal{H}^p_Z(\mathcal{O}_X)\cong R^pi'_*\mathcal{O}_{X'}$, for all $p>0$, so in particular, we have $\mathcal{H}^1_Z(\mathcal{O}_X)\cong R^1j_*\mathcal{O}_{X_{reg}}$. In order to show that the sequence (\ref{exseq2}) is exact it suffices to work affine locally, so we may assume $X=\textrm{Spec}(A)$, for some Noetherian ring $A$. By definition, we have $\textrm{depth}_Z(\mathcal{O}_X)=inf_{x\in Z}(\textrm{depth}(\mathcal{O}_{X,x}))$, where $\textrm{depth}(\mathcal{O}_{X,x})$ is the depth of the local ring $\mathcal{O}_{X,x}$ considered as a module over itself. By the assumptions on $X$, we know that $X$ is a Cohen-Macaulay scheme, which means that every local ring $\mathcal{O}_{X,x}$ of $X$ is Cohen-Macaulay. Let $\mathfrak{p}_x\subset A$ denote the prime ideal associated to the point $x$. The Cohen-Macaulay condition implies that $\textrm{depth}(\mathcal{O}_{X,\mathfrak{p}_x})=\textrm{dim}(\mathcal{O}_{X,\mathfrak{p}_x})=\textrm{height}(\mathfrak{p}_x)$ for all $x\in X$. Since the codimension of $Z$ in $X$ is $\ge3$, we have $\textrm{dim}(\mathcal{O}_{X,\mathfrak{p}_x})=\textrm{height}(\mathfrak{p}_x)\ge3$, i.e., $\textrm{depth}(\mathcal{O}_{X,\mathfrak{p}_x})\ge3$ for all $x\in Z$. It follows that $\textrm{depth}_Z(\mathcal{O}_X)=inf_{x\in Z}(\textrm{depth}(\mathcal{O}_{X,x}))\ge3$. Thus by \cite[Theorem 3.8]{loco}, we have $\mathcal{H}^p_Z(\mathcal{O}_X)=0$ for $p<3$, i.e., in particular $\mathcal{H}^1_Z(\mathcal{O}_X)=0$. This implies that $R^1i'_*\mathcal{O}_{X'}=R^1j_*\mathcal{O}_{X_{reg}}=0$, and hence that the sequence (\ref{exseq2}) is exact on X.
\end{proof}

\begin{lemma}\label{keylem}
Suppose $X$ satisfies the assumptions in Theorem \ref{thm1}, and let $\mathcal{L}_i$ be a direct summand of $\mathcal{T}_X$ for each $1\le i\le n$. Then we have $\widehat{c}_1(\mathcal{L}_i)^2=0$ for each $1\le i\le n$. 
\end{lemma}
\begin{proof}
Recall from \cite[Section 3.7]{gkpt}, that for a reflexive sheaf $\mathcal{L}_i$ on $X$, $\widehat{c}_1(\mathcal{L}_i)^2:N_1(X)_{\mathbb{Q}}^{\times(n-2)}\to\mathbb{Q}$ is a $\mathbb{Q}$-multilinear form which maps a tuple $(\alpha_1,...,\alpha_{n-2})\in N^1(X)_{\mathbb{Q}}^{\times(n-2)}$ to $\widehat{c}_1(\mathcal{L}_i)^2\cdot\alpha_1\dots\alpha_{n-2}\in\mathbb{Q}$, such that the properties listed in \cite[Theorem 3.13]{gkpt}, hold. Recall also that $K_X$ is big and nef, so that we can choose a sufficiently increasing and divisible sequence of numbers $0<<m_1<<\dots<<m_{n-2}$ and a general tuple of elements $(H_1,...,H_{n-2})\in\prod_i|m_iK_X|$, and set $S=H_1\cap\dots\cap H_{n-2}$. Note that $S$ has only quotient singularities and hence is contained entirely in $X'=X\setminus Z$, where $Z\subset X$ is the subset of $X$ containing points which are not quotient singularities.\\
In order to show $\widehat{c}_1(\mathcal{L}_i)^2=0$, it suffices to show $\widehat{c}_1(\mathcal{L}_i)^2\cdot S=0$ for every surface $S$ of complete intersection constructed as above. Since $S\subset X'$, the latter equality is equivalent to $\widehat{c}_1(\mathcal{L}_i|_{X'})^2\cdot S=0$. Since $X'$ is $\mathbb{Q}$-factorial, $\mathcal{L}_i|_{X'}$ is $\mathbb{Q}$-Cartier and hence has an associated first Chern class $c_1(\mathcal{L}_i|_{X'})\in H^1(X',\Omega^1_{X'})$. This cohomology class can be identified with an element in $H^2(X',\mathbb{Q})$, which we also denote by $c_1(\mathcal{L}_i|_{X'})$. Thus we have $\widehat{c}_1(\mathcal{L}_i|_{X'})^2\cdot S=c_1(\mathcal{L}_i|_{X'})^2\cdot S$, where the latter intersection product is computed as the cup product of $c_1(\mathcal{L}_i|_{X'})^2\in H^4(X',\mathbb{Q})$ and $[S]\in H^{2(n-2)}(X',\mathbb{Q})$. Choose an integer $m$ large enough so that $(\mathcal{L}_i|_{X'})^{\otimes m}$ is locally free. Then $c_1(\mathcal{L}_i|_{X'})^2\cdot S=0$ implies that $c_1((\mathcal{L}_i|_{X'})^{\otimes m})^2\cdot S=0$ and conversely, hence we may assume for the remainder of the proof that $\mathcal{L}_i|_{X'}$ is locally free.\\
Recall from the proof of Proposition \ref{pex} that $c_1(\mathcal{L}_i|_{X'})$ corresponds to the extension class $\widehat{at}(\mathcal{L}_i|_{X'})$ of the exact sequence 
\begin{align*}
    0\to\mathcal{O}_{X'}\to i_*\mathcal{D}^1(\mathcal{L}_i|_{X_{reg}})\to\mathcal{T}_{X'}\to0
\end{align*}
in $H^1(X',\Omega^1_{X'})$. We know from \cite[Lemma 3.1]{bea}, that the above exact sequence restricted to $X_{reg}$ is exact and splits over the sub-bundle $\mathcal{F}=\bigoplus_{j\neq i}\mathcal{L}_j|_{X_{reg}}\subset\mathcal{T}_{X_{reg}}$, hence $at(\mathcal{L}_i|_{X_{reg}})\in H^1(X_{reg},\Omega^1_{X_{reg}})$ vanishes in $H^1(X_{reg},\mathcal{F}^*)$ and comes from $H^1(X_{reg},\mathcal{L}^*_i|_{X_{reg}})$.\\
The pushforward of the $\mathcal{O}_{X_{reg}}$-linear map $\mathcal{F}\to\mathcal{D}^1(\mathcal{L}_i|_{X_{reg}})$ along the inclusion $i:X_{reg}\to X$ gives an $\mathcal{O}_{X'}$-linear map $i_*\mathcal{F}\to i_*\mathcal{D}^1(\mathcal{L}_i|_{X_{reg}})$. Over an open subset $U\subset X'$, this map sends a local section $u\in H^0(U,i_*\mathcal{F})\cong H^0(U_{reg},\mathcal{F})$ to $D_u\in H^0(U,i_*\mathcal{D}^1(\mathcal{L}_i|_{X_{reg}}))\cong H^0(U_{reg},\mathcal{D}^1(\mathcal{L}_i|_{X_{reg}}))$, where $U_{reg}=i^{-1}U=U\cap X_{reg}$. Moreover, the pushed forward symbol map $i_*\sigma:i_*\mathcal{D}^1(\mathcal{L}_i|_{X_{reg}})\to\mathcal{T}_{X'}$ maps $D_u$ to $u$, for every local section $u$ of $i_*\mathcal{F}$. Hence we see that the above exact sequence splits over the subsheaf $i_*\mathcal{F}$ of $\mathcal{T}_{X'}$. Thus its extension class $\widehat{at}(\mathcal{L}_i|_{X'})\in H^1(X',\Omega^1_{X'})$ vanishes in $H^1(X',(i_*\mathcal{F})^*)$, hence comes from $H^1(X',(\mathcal{L}_i|_{X'})^*)$. Since $(\mathcal{L}_i|_{X'})^*$ is a locally free sheaf of rank 1 on $X'$, we have $\widehat{at}(\mathcal{L}_i|_{X'})^2\in H^2(X',\bigwedge^2(\mathcal{L}_i|_{X'})^*)=0$. This implies $c_1(\mathcal{L}_i|_{X'})^2=0$, and hence $\widehat{c}_1(\mathcal{L}_i)^2\cdot S=0$, for every surface $S\subset X$ of complete intersection constructed as earlier.
\end{proof}

The proof of Theorem \ref{thm1} is based on the following key auxiliary Propositions.

\begin{proposition}\label{prop1}
Let $X$ be an $n$-dimensional projective, klt variety of general type whose canonical divisor $K_X$ is big and nef. Suppose that the tangent sheaf $\mathcal{T}_{X_{reg}}$ of the smooth locus $X_{reg}$ splits as in (\ref{split2}). Then, the tangent sheaf of the smooth locus of the canonical model $X_{can}$ of $X$ also splits as in (\ref{split2}), where each direct summand is a line bundle of negative degree.
\end{proposition}
\begin{proof}
Let $\pi:X\to X_{can}$ be the birational crepant morphism to the canonical model $X_{can}$. We know from \cite{gkpt2} that $X_{can}$ is also projective and klt and its canonical divisor $K_{X_{can}}$ is ample.\\
Let $U\subset X$ be the largest open subset of $X$ restricted to which $\pi$ is an isomorphism. Let $V\subset X_{can}$ be the open subset defined as the intersection $V=\pi(U)\cap X_{can,reg}$, where $X_{can,reg}$ denotes the smooth locus of $X_{can}$. Since $X_{can}\setminus\pi(U)$ is a codimension $\ge2$  subset of $X_{can}$, it follows that $X_{can,reg}\setminus V$ is a codimension $\ge2$ subset of $X_{can,reg}$. The subset $X_{reg}\setminus\pi^{-1}(V)$ is a codimension $\ge1$ subset of $X_{reg}$. \\
The restriction $\mathcal{T}_{X_{reg}}|_{\pi^{-1}(V)}=\mathcal{T}_{\pi^{-1}(V)}$ splits as a direct sum of line bundles because we have by assumption that $\mathcal{T}_{X_{reg}}\cong\mathcal{L}_1\oplus\dots\oplus\mathcal{L}_n$, where each $\mathcal{L}_i$ is a line bundle. Since we have $\pi^{-1}(V)\cong V$, it follows that the corresponding tangent bundle $\mathcal{T}_V$ of $V$ also splits as a direct sum of line bundles, say $\mathcal{T}_V=\mathcal{M}_1\oplus\dots\oplus\mathcal{M}_n$. Let $\mathcal{M}'_i$ denote the unique extension of $\mathcal{M}_i$ as a reflexive sheaf of rank 1 over $X_{can,reg}$ for all $1\le i\le n$. Then by uniqueness of the extensions, we have that $\mathcal{M}'_1\oplus\dots\oplus\mathcal{M}'_n\cong\mathcal{T}_{X_{can,reg}}$. Since $X_{can,reg}$ is smooth, $\mathcal{T}_{X_{can,reg}}$ must be locally free. This implies that the summand $\mathcal{M}'_i$ must be locally free, i.e. a line bundle, for all $1\le i\le n$. \\
We know from Proposition \ref{tanprop} that the tangent sheaf $\mathcal{T}_{X_{can}}$ is semistable with respect to $K_{X_{can}}$. Since $K_{X_{can}}$ is ample, we have $[K_{X_{can}}]^n>0$ and it follows that 
\begin{align*}
    \mu_{K_{X_{can}}}(\mathcal{T}_{X_{can}})=\frac{[\textrm{det}(\mathcal{T}_{X_{can}})]\cdot[K_{X_{can}}]^{n-1}}{\textrm{rank}(\mathcal{T}_{X_{can}})}=-\frac{[K_{X_{can}}]^{n}}{n}<0.
\end{align*}
Since $\mathcal{T}_{X_{can,reg}}$ decomposes as a direct sum of line bundles, it follows that $\mathcal{T}_{X_{can}}$ decomposes as a direct sum of rank 1 reflexive sheaves on $X_{can}$. Namely, $\mathcal{T}_{X_{can}}\cong\bigoplus_{i=1}^nj_*\mathcal{M'}_i$, where $j:X_{can,reg}\to X_{can}$ denotes the inclusion. For each direct summand $j_*\mathcal{M}'_i$ of $\mathcal{T}_{X_{can}}$, we have $\mu_{K_{X_{can}}}(j_*\mathcal{M}'_i)\le\mu_{K_{X_{can}}}(\mathcal{T}_{X_{can}})<0$. This implies that $c_1(j_*\mathcal{M}'_i)\cdot[K_{X_{can}}]^{n-1}<0$ for all $1\le i\le n$. Note that $[K_{X_{can}}]^{n-1}$ corresponds to the class of a smooth curve in $X_{can}$ because $K_{X_{can}}$ is ample. Hence $c_1(\mathcal{M}'_i)\cdot[K_{X_{can}}]^{n-1}=c_1(j_*\mathcal{M}'_i)\cdot[K_{X_{can}}]^{n-1}<0$ for all $1\le i\le n$. Thus the tangent bundle $\mathcal{T}_{X_{can,reg}}$ of $X_{can,reg}$ splits as a direct sum of line bundles of negative degree.
\end{proof}

\begin{proposition}\label{prop}
Let $X$ be an $n$-dimensional projective klt variety such that the tangent sheaf $\mathcal{T}_{X_{reg}}$ of the regular locus $X_{reg}$ splits as a direct sum of line bundles as in (\ref{split2}). Let $\gamma:Y\to X$ be a Galois, quasi-\'{e}tale cover. Then, the tangent sheaf $\mathcal{T}_{Y_{reg}}$ of the regular locus $Y_{reg}$ also splits as a direct sum of line bundles. 
\end{proposition}

\begin{proof}
Note that since $\gamma:Y\to X$ is quasi-\'{e}tale, $Y$ is again projective and klt of dimension $n$. By purity of branch locus, we know that $\gamma$ branches only over the singular locus of $X$. It follows that $\gamma^{-1}(X_{reg})=Y^o$, where $Y^o$ denotes the open subset of $Y$ restricted to which the map $\gamma$ is \'{e}tale. Note that $Y^o\subset Y_{reg}$, because any singular point of $Y$ necessarily belongs to the branching locus of $\gamma$. Since $\gamma$ is finite, the codimension of the branching locus of $\gamma$ in $Y$ is equal to the codimension of the singular locus of $X$ in $X$. This implies that $Y^o$ is a big open subset of $Y$, and in particular of the smooth locus $Y_{reg}$. \\
Since the restricted map $\gamma:Y^0\to X_{reg}$ is finite and \'{e}tale, we have $K_{Y^o}\cong\gamma^*K_{X_{reg}}$, and $\mathcal{T}_{Y^o}\cong\gamma^*\mathcal{T}_{X_{reg}}\cong\gamma^*\mathcal{L}_1\oplus\dots\oplus\gamma^*\mathcal{L}_n$, where the $\mathcal{L}_i$'s are the line bundles appearing in the direct sum decomposition of $\mathcal{T}_{X_{reg}}$. Let $\mathcal{L}'_i$ denote the unique extension of $\gamma^*\mathcal{L}_i$ as a reflexive sheaf over $Y_{reg}$, for all $1\le i\le n$. Then by uniqueness of the reflexive extension, we have $\mathcal{L}'_1\oplus\dots\oplus\mathcal{L}'_n\cong\mathcal{T}_{Y_{reg}}$. Since $Y_{reg}$ is smooth, $\mathcal{T}_{Y_{reg}}$ is locally free, which implies that each $\mathcal{L}'_i$ is locally free, i.e. a line bundle for all $1\le i\le n$. 
\end{proof}

\begin{proposition}\label{prop2}
Let $X$ be a projective, klt variety of dimension n, and let the canonical divisor $K_X$ be ample. Suppose the tangent sheaf $\mathcal{T}_{X_{reg}}$ of the smooth locus of $X$ splits as a direct sum of line bundles as in (\ref{split2}), and suppose that $X$ is maximally quasi-\'{e}tale, i.e., there is an isomorphism of \'{e}tale fundamental groups $\widehat{\pi}_1(X_{reg})\cong\widehat{\pi}_1(X)$. Then $X$ is smooth.
\end{proposition}

\begin{proof}
Since $\mathcal{T}_{X_{reg}}\cong\mathcal{L}_1\oplus\dots\oplus\mathcal{L}_n$, and we have assumed $K_X$ to be ample, each line bundle $\mathcal{L}_i$ appearing in the direct sum decomposition of $\mathcal{T}_{X_{reg}}$ has negative degree i.e., $c_1(\mathcal{L}_i)\cdot[K_X]^{n-1}<0$ for all $1\le i\le n$. Moreover, since $\mathcal{T}_X$ is semistable with respect to $K_X$, all the line bundles $\mathcal{L}_i$ have the same slope.\\
Let $G_0=SL(2,\mathbb{R})^n$, then $G_0$ is a connected Hodge group corresponding to the polydisk $\mathbb{H}^n$, which is a Hermitian symmetric space of noncompact type. Let $G=SL(2,\mathbb{C})^n$, a complexification of $G_0$, and let $\mathfrak{g}=\mathfrak{sl}(2,\mathbb{C})^n$, the Lie algebra of $G$. Since $G_0$ is a Hodge group of \emph{Hermitian type}, $\mathfrak{g}$ has the following Hodge decomposition
\begin{align*}
\mathfrak{g}=\mathfrak{g}^{-1,1}\oplus\mathfrak{g}^{0,0}\oplus\mathfrak{g}^{1,-1},
\end{align*}
with $[\mathfrak{g}^{p,-p},\mathfrak{g}^{q,-q}]\subset\mathfrak{g}^{p+q,-p-q}$ for $p,q\in\{-1,0,1\}$, where $[\cdot,\cdot]$ denotes the Lie bracket of $\mathfrak{g}$. An element of $\mathfrak{g}$ can be written as an $n$-tuple of trace zero $2\times2$ matrices with entries in $\mathbb{C}$. Note that $\mathbb{H}^n\cong G_0/K_0$, where $K_0=U(1)^n$ is a maximal compact subgroup of $G_0$. Let $K$ denote the complexification of $K_0$, then we have $K\cong(\mathbb{C}^*)^n$.\\
Since $\mathcal{T}_{X_{reg}}$ splits as a direct sum of line bundles, it admits a reduction in structure group from $GL(n,\mathbb{C})$ to $K=(\mathbb{C}^*)^n$. Let $P$ denote the principle $K=(\mathbb{C}^*)^n$-bundle over $X_{reg}$ associated to $\mathcal{T}_{X_{reg}}$, i.e., $P$ has the same locally trivializing cover and the same transition functions as $\mathcal{T}_{X_{reg}}$. Then since $\mathbb{C}^{\oplus n}$ and $\mathfrak{g}^{-1,1}$ are isomorphic as $K$-representations, it follows from the associated bundle construction that there is an isomorphism of holomorphic vector bundles
\begin{align}\label{map}
    \theta:\mathcal{T}_{X_{reg}}\longrightarrow P\times_K\mathfrak{g}^{-1,1}
\end{align}
such that $[\theta(u),\theta(v)]=0$ for all local sections $u,v$ of $\mathcal{T}_{X_{reg}}$. Thus $(P,\theta)$ is a uniformizing system of Hodge bundles on $X_{reg}$. Since $\mathfrak{g}$ is a polarized Hodge representation of $G$, $E=P\times_K\mathfrak{g}$ is a system of Hodge bundles on $X_{reg}$, so in particular $E$ is a Higgs bundle on $X_{reg}$. It has a Hodge decomposition $E=E^{-1,1}\oplus E^{0,0}\oplus E^{1,-1}$ induced by the Hodge decomposition of $\mathfrak{g}$, i.e., $E^{i,-i}\cong P\times_K\mathfrak{g}^{i,-1}$ for $i\in\{-1,0,1\}$.\\ 
The isomorphism $\theta$ induces a map from $\mathcal{T}_{X_{reg}}$ to the endomorphism bundle $\mathcal{E}nd(E)$. This map is given locally by sending an element $a\in\mathcal{T}_{X_{reg},x}\cong\mathfrak{g}^{-1,1}$ to the endomorphism $b\mapsto[\theta(a),b]$ in $\mathcal{E}nd(E)_x$. Hence $\theta$ corresponds to an element in $H^0(X_{reg},\Omega^1_{X_{reg}}\otimes \mathcal{E}nd(E))$ via the isomorphism $\mathcal{H}om(\mathcal{T}_{X_{reg}},\mathcal{E}nd(E))\cong\Omega^1_{X_{reg}}\otimes \mathcal{E}nd(E)$. Since $\Omega^1_{X_{reg}}\otimes \mathcal{E}nd(E)\cong\Omega^1_{X_{reg}}\otimes E\otimes E^*\cong \mathcal{H}om(E,E\otimes\Omega^1_{X_{reg}})$, we see that $\theta$ eventually corresponds to a morphism $\widehat{\theta}:E\to E\otimes\Omega^1_{X_{reg}}$. The map $\widehat{\theta}$ is given locally by sending an element $a\in E_x\cong\mathfrak{g}$ to the map $b\mapsto[\theta(b),a]\in \mathcal{H}om(\mathcal{T}_{X_{reg},x},E_x)\cong\Omega^1_{X_{reg},x}\otimes E_x$. Thus on the direct summands $E^{p,q}$ appearing in the Hodge decomposition of $E$, we have $\widehat{\theta}:E^{p,q}\to E^{p-1,q+1}\otimes\Omega^1_{X_{reg}}$. Moreover, it is straightforward to see that $\widehat{\theta}\wedge\widehat{\theta}=0$, where $\widehat{\theta}\wedge\widehat{\theta}$ denotes the composition
\begin{align*}
\widehat{\theta}\wedge\widehat{\theta}:E\xrightarrow{\widehat{\theta}}E\otimes\Omega^1_{X_{reg}}\xrightarrow{\widehat{\theta}\otimes Id}E\otimes\Omega^1_{X_{reg}}\otimes\Omega^1_{X_{reg}}\xrightarrow{Id\otimes[\wedge]}E\otimes\Omega^2_{X_{reg}}.
\end{align*}
Thus $\widehat{\theta}$ is in fact a Higgs field on $E$. We can decompose $\mathfrak{g}$ as as direct sum of simple ideals $\mathfrak{g}\cong\bigoplus_i\mathfrak{g}_i$, where $\mathfrak{g}_i\cong\mathfrak{sl}(2,\mathbb{C})$ for all $1\le i\le n$. Since $G=SL(2,\mathbb{C})^n$ is connected, we have a decomposition $E=P\times_K\mathfrak{g}=\bigoplus_iP\times_K\mathfrak{g}_i$, where $E_i=P\times_K\mathfrak{g}_i\cong\mathcal{L}_i\oplus\mathcal{L}_i^*\oplus\mathcal{O}_{X_{reg}}$ for all $1\le i\le n$, and each $E_i$ is a subsystem of Hodge sheaves of $E$.\\
We want to show that $E$ is $K_X$-polystable as a Higgs bundle on $X_{reg}$, and by Lemma \ref{lemsta} it is sufficient to show that $E$ is $K_X$-polystable as a system of Hodge bundles. It follows from \cite[Corollary 9.3]{simp} that if $W\subset E$ is a subsystem of Hodge sheaves then $\textrm{rank}(W^{-1,1})\ge\textrm{rank}(W^{1,-1})$. We want to show that $E_i$ is a $K_X$-stable as a system of Hodge bundles, and hence as a Higgs bundle on $X_{reg}$, for all $1\le i\le n$.\\
Let $\mathcal{F}\subset E_i$ be a subsystem of Hodge sheaves with $0<\textrm{rank}(\mathcal{F})\le\textrm{rank}(E_i)$. For each local section $a\in\mathcal{F}$, we can identify the image $\widehat{\theta}(a)$ with a local section $(b\mapsto[\theta(b),a])\in \mathcal{H}om(\mathcal{T}_{X_{reg}},E_i)$, with $b\in\mathcal{T}_{X_{reg}}\cong P\times_K\mathfrak{g}^{-1,1}$. For each non-zero local section $b\in\mathcal{T}_{X_{reg}}$, $\theta(b)$ can be expressed as an $n$-tuple of trace zero upper-diagonal $2\times2$ matrices of the form 
$$
\left(\begin{bmatrix}0&\beta_j\\0&0\end{bmatrix}\right)_j,
$$
$\beta_j\in\mathbb{C}$ for all $1\le j\le n$. If $\mathcal{F}$ is a subsystem of $E_i$ of rank one, then $\mathcal{F}$ is a subsheaf of one of the direct summands of $E_i$. If $\mathcal{F}\subset\mathcal{L}_i$, then $c_1(\mathcal{F})\cdot[K_X]^{n-1}<0$ because $c_1(\mathcal{L}_i)\cdot[K_X]^{n-1}<0$ and $\mathcal{L}_i$ is semistable. If $\mathcal{F}\subset\mathcal{L}^*_i$, then any non-zero local section $a\in\mathcal{F}$ can be represented by an $n$-tuple of the form 
$$
\left(0,\dots,0,\begin{bmatrix}0&0\\ \alpha_i&0\end{bmatrix},0,\dots,0\right),
$$
with $\alpha_i\in\mathbb{C}^*$. The bracket $[\theta(b),a]$ gives 
$$
\left(0,\dots,0,\begin{bmatrix}\alpha_i\beta_i&0\\0&-\alpha_i\beta_i\end{bmatrix},0,\dots,0\right).
$$
Thus the image $\widehat{\theta}(a)$ is a non-zero local section in $\mathcal{O}_{X_{reg}}\otimes\Omega^1_{X_{reg}}$, for all non-zero $a\in\mathcal{F}$, which implies that $\widehat{\theta}(\mathcal{F})$ is not contained in $\mathcal{F}\otimes\Omega^1_{X_{reg}}$. It follows that $\mathcal{F}$ is not $\widehat{\theta}$-invariant, hence not a Higgs subsheaf of $E_i$. Similarly, we see that if $\mathcal{F}\subset\mathcal{O}_{X_{reg}}$, then $\mathcal{F}$ is not $\widehat{\theta}$-invariant and hence not a Higgs subsheaf of $E_i$. If $\mathcal{F}$ is a rank two subsystem of $E_i$ then $\mathcal{F}$ is a subsheaf of either $\mathcal{L}_i\oplus\mathcal{L}^*_i$, or $\mathcal{L}_i\oplus\mathcal{O}_{X_{reg}}$, or $\mathcal{L}^*_i\oplus\mathcal{O}_{X_{reg}}$. If $\mathcal{F}\subset\mathcal{L}_i\oplus\mathcal{O}_{X_{reg}}$, then it is straightforward to see that $c_1(\mathcal{F})\cdot[K_{X_{reg}}]^{n-1}<0$. If $\mathcal{F}\subset\mathcal{L}^*_i\oplus\mathcal{O}_{X_{reg}}$, then any non-zero local section $a\in\mathcal{F}$ can be expressed as a $n$-tuple of the form 
$$
\left(0,\dots,0,\begin{bmatrix}\gamma_i&0\\\alpha_i&-\gamma_i\end{bmatrix},0,\dots,0\right)
$$
with $\alpha_i,\gamma_i\in\mathbb{C}^*$. For a non-zero local section $b\in\mathcal{T}_{X_{reg}}$ as earlier, the bracket $[\theta(b),a]$ is then 
$$
\left(0,\dots,0,\begin{bmatrix}\alpha_i\beta_i&-2\gamma_i\beta_i\\0&-\alpha_i\beta_i\end{bmatrix},0,\dots,0\right).
$$
Thus the image $\widehat{\theta}(a)$ is a non-zero local section of $(\mathcal{L}_i\oplus\mathcal{O}_{X_{reg}})\otimes\Omega^1_{X_{reg}}$, i.e., the image $\widehat{\theta}(\mathcal{F})$ is not contained in $\mathcal{F}\otimes\Omega^1_{X_{reg}}$. It follows that $\mathcal{F}$ is not $\widehat{\theta}$-invariant and hence not a Higgs subsheaf of $E_i$. It can similarly be verified that if $\mathcal{F}\subset\mathcal{L}_i\oplus\mathcal{L}^*_i$, then $\mathcal{F}$ is not $\widehat{\theta}$-invariant and hence not a Higgs subsheaf of $E_i$. If $\mathcal{F}=\mathcal{F}^{-1,1}\oplus\mathcal{F}^{0,0}\oplus\mathcal{F}^{1,-1}$ is a subsystem of Hodge sheaves of $E_i$ of rank 3, then we know from \cite[Corollary 9.3]{simp} that $\textrm{rank}(\mathcal{F}^{-1,1})\ge\textrm{rank}(\mathcal{F})^{1,-1}$. If the strict inequality holds, then $c_1(\mathcal{F})\cdot[K_X]^{n-1}<0$, and if equality holds then $\mathcal{F}=E_i$. These are all the cases, hence we conclude that any proper subsystem of Hodge sheaves of $E_i$ has slope strictly less than zero, which implies that $E_i$ is $K_X$-stable as a Higgs bundle for all $1\le i\le n$. Since $E$ is a direct sum of all the $E_i$'s, it follows that $E$ is $K_X$-polystable as a Higgs bundle on $X_{reg}$.\\
Note that $E\cong\mathcal{T}_{X_{reg}}\oplus\Omega^1_{X_{reg}}\oplus\mathcal{O}^{\oplus n}_{X_{reg}}$, and let $\mathcal{F}_X$ be the unique extension of $E$ as a reflexive sheaf on $X$. Then we have $\mathcal{F}_X\cong\mathcal{T}_X\oplus\Omega^{[1]}_X\oplus\mathcal{O}^{\oplus n}_X$, and by construction it follows that $c_1(\mathcal{F})=0$. Note that $\mathcal{T}_X=\bigoplus_{i}j_*\mathcal{L}_i$, where $j:X_{reg}\to X$ denotes the natural inclusion. Therefore $\widehat{c}_2(\mathcal{F}_X)=\sum_i\widehat{c}_1(j_*\mathcal{L}_i)^2$. It follows from Lemma \ref{keylem} that $\widehat{c}_1(j_*\mathcal{L}_i)^2=0$ for all $1\le i\le n$ and hence $\widehat{c}_2(\mathcal{F}_X)=0$. In particular, we have $\widehat{ch}_2(\mathcal{F}_X)\cdot[K_X]^{n-2}=\widehat{c}_2(\mathcal{F}_X)\cdot[K_X]^{n-2}=0$. Following the notation of \cite{gkpt2}, \cite[Theorem 5.1]{gkpt2} thus implies that $(E,\widehat{\theta})\in\textrm{TPI-Higgs}_{X_{reg}}$.\\ Then \cite[Proposition 3.17]{gkpt2} together with the assumption that $X$ is maximally quasi-\'{e}tale implies that $\mathcal{F}_X$ is locally free. Since $\mathcal{T}_X$ is a direct summand of $\mathcal{F}_X=\mathcal{T}_X\oplus\Omega^1_X\oplus\mathcal{O}_X^{\oplus n}$, it follows that $\mathcal{T}_X$ must be locally free. The solution of the Lipman-Zariski conjecture for klt spaces (\cite[Theorem 16.1]{gkkp}) thus asserts that $X$ is smooth, and we are done.
\end{proof}

\textbf{Proof of Theorem \ref{thm1}}\\
Suppose $X$ is an $n$-dimensional projective, klt variety of general type which satisfies the assumptions of Theorem \ref{thm1}. Recall from Proposition \ref{prop1} that the canonical model $X_{can}$ of $X$ is projective, klt, and the tangent sheaf of its regular locus splits as a direct sum of line bundles of negative degree. Moreover, the canonical divisor $K_{X_{can}}$ of $X_{can}$ is ample. Let $\gamma:Y\to X_{can}$ be a Galois, maximally quasi-\'{e}tale cover. Note that the existence of such a cover has been established in \cite[Theorem 1.14]{gkp1}. Since $\gamma$ branches only over the singular set of $X_{can}$, it follows that $Y$ is again klt, and since $\gamma$ is finite, the canonical divisor $K_Y=\gamma^*K_{X_{can}}$ is again ample. Moreover, by Proposition \ref{prop} and the semistability of $\mathcal{T}_Y$ with respect to $K_Y$ (Proposition \ref{tanprop}), we have that $\mathcal{T}_{Y_{reg}}$ splits as a direct sum of line bundles of negative degree. Thus $Y$ satisfies the assumptions of Proposition \ref{prop2} and we conclude that $Y$ is smooth. It follows that $Y$ also satisfies the assumptions of \cite[Corollary 9.7]{simp}, hence $Y$ is uniformized by the polydisk $\mathbb{H}^n$. \qed  

\begin{remark}\label{rem}
While the splitting of the tangent bundle $\mathcal{T}_{X_{reg}}$ of the smooth locus $X_{reg}$ of $X$ as a direct sum of line bundles is a sufficient condition for (the canonical model of) $X$ to be uniformized by the polydisk, it is not a necessary condition. The reason, as pointed out by Catanese and Di Scala in \cite{catdi}, is that the group $Aut(\mathbb{H}^n)$ of holomorphic automorphisms of the polydisk is not $PSL(2,\mathbb{R})^n$, but is the semidirect product of $PSL(2,\mathbb{R})^n$ and the symmetric group $\sigma_n$. Thus there is a split sequence
\begin{align*}
    1\to Aut(\mathbb{H})^n\hookrightarrow Aut(\mathbb{H}^n)\to\sigma_n\to1
\end{align*}
where $Aut(\mathbb{H})^n=Aut^0(\mathbb{H}^n)=PSL(2,\mathbb{R})^n$. The group $Aut(\mathbb{H}^n)$ is a disconnected Hodge group of Hermitian type. The splitting of the tangent bundle $\mathcal{T}_{\mathbb{H}^n}\cong\bigoplus_{i=1}^n\mathcal{T}_{\mathbb{H}}$ descends to a quotient $X\cong\mathbb{H}^n/\Gamma$, where $\Gamma\subset Aut(\mathbb{H}^n)$ acts discretely and cocompactly on $\mathbb{H}^n$, only if $\Gamma$ acts diagonally on $\mathbb{H}^n$, i.e., only if $\Gamma\subset Aut(\mathbb{H})^n=PSL(2,\mathbb{R})^n$. However, since any subgroup $\Gamma\subset Aut(\mathbb{H}^n)$ can be expressed as an extension of a subgroup $\Gamma'$ of $PSL(2,\mathbb{R})^n$ by a subgroup $H$ of $\sigma_n$ because of the above exact sequence, every projective, klt quotient $X$ of a polydisk admits a Galois, quasi-\'{e}tale cover $\gamma:Y\to X$ such that $Y$ is also a polydisk quotient whose tangent sheaf $\mathcal{T}_Y$ splits as a direct sum of reflexive sheaves of rank one.
\end{remark}

\subsection{Necessary conditions}

Next we would like to formulate a necessary condition for a projective klt variety of general type to be uniformized by the polydisk. This will be a more precise version of Remark \ref{rem}.

\begin{proposition}
Let $X$ be a $n$-dimensional projective klt variety with $K_X$ is ample, and such that $X\cong\mathbb{H}^n/\widehat{\Gamma}$, where $\widehat{\Gamma}\subset PSL(2,\mathbb{R})^n\rtimes\sigma_n=\textrm{Aut}(\mathbb{H}^n)$ is a discrete cocompact subgroup acting fixed point freely in codimension one. Then $X$ admits a smooth quasi-\'{e}tale cover $\gamma:Y\to X$ such that $K_Y$ is ample, and the tangent bundle $\mathcal{T}_Y$ splits as a direct sum of line bundles.
\end{proposition}
\begin{proof}
Due to the structure of the automorphism group $Aut(\mathbb{H}^n)=PSL(2,\mathbb{R})^n\rtimes\sigma_n$, we know that $\widehat{\Gamma}$ can be expressed as an extension
\begin{align*}
    1\to\Gamma'\to\widehat{\Gamma}\to H\to1,
\end{align*}
where $\Gamma'\subset PSL(2,\mathbb{R})^n$ is normal of finite index in $\widehat{\Gamma}$, and $H\subset\sigma_n$ is finite. The quotient map $\pi:\mathbb{H}^n\to X$ thus factors as
\begin{align*}
    \mathbb{H}^n\to\mathbb{H}^n/\Gamma'\to\mathbb{H}^n/\widehat{\Gamma}=X,
\end{align*}
where $Y'=\mathbb{H}^n/\Gamma'$ is also klt and projective, and the map $Y'\to X$ is Galois and quasi-\'{e}tale. Following the argument in Step I of the proof of Theorem \ref{genthm}, the fundamental group of the smooth locus $Y_{reg}$ is finitely generated and isomorphic to $\Gamma'$. Then using Selberg's lemma we conclude that $\Gamma'$ has a torsion free normal subgroup $\Gamma$ of finite index. The quotient map $\pi':\mathbb{H}^n\to Y'$ thus factors as
\begin{align*}
    \mathbb{H}^n\to\mathbb{H}^n/\Gamma\to\mathbb{H}^n/\Gamma'=Y'.
\end{align*}
Since $\Gamma$ acts freely and cocompactly on $\mathbb{H}^n$, the quotient $\mathbb{H}^n/\Gamma$ is a smooth projective variety. The map $Y\to Y'$ is Galois and quasi-\'{e}tale, and after composing with the map $Y'\to X$, we get a Galois, quasi-\'{e}tale map $Y\to X$. Recall that the tangent bundle of $\mathbb{H}^n$ can be expressed as $\mathcal{T}_{\mathbb{H}^n}\cong P\times_K\mathfrak{g}^{-1,1}$, where $K=(\mathbb{C}^*)^n$ is maximal compact inside $SL(2,\mathbb{C})^n$, and $\mathfrak{g}=\mathfrak{sl}(2,\mathbb{C})^n$. The action of $\Gamma$ on $\mathbb{H}^n$ lifts to a left action of $\Gamma$ on $\mathcal{T}_{\mathbb{H}^n}$ via pushforward of tangent spaces, so it follows that $\mathcal{T}_Y\cong P'\times_K\mathfrak{g}^{-1,1}$, where $P'\cong P/\Gamma$ as principal $K$-bundles on $Y$. Since $\mathcal{T}_Y$ also has structure group $K=(\mathbb{C}^*)^n$, it splits as a direct sum of line bundles. Since the map $\gamma:Y\to X$ is finite, and $K_X$ is ample by assumption, it follows that $K_Y=\gamma^*K_X$ is also ample. The semistability of $\mathcal{T}_Y$ with respect to $K_Y$ then implies that the line bundles appearing in the direct sum decomposition of $\mathcal{T}_Y$ all have negative degree.
\end{proof}

\begin{remark}
If $X$ is a smooth projective quotient of the polydisk, then we \emph{do not} in general have a representation $\pi_1(X)\to PSL(2,\mathbb{R})^n$. Therefore the action of $\pi_1(X)$ on $\mathbb{H}^n$ is not in general compatible with the reduction in structure group of $\mathcal{T}_{\mathbb{H}^n}$ from $GL(n,\mathbb{C})$ to $(\mathbb{C}^*)^n=K$. Hence the splitting of the tangent bundle into a direct sum of line bundles does not in general descend to the quotient variety $X$.\\
There is however always a representation $\pi_1(X)\to PSL(2,\mathbb{R})^n\rtimes\sigma_n$, and the action of $\pi_1(X)$ on $\mathbb{H}^n$ is compatible with the reduction in structure group of $\mathcal{T}_{\mathbb{H}^n}$ from $GL(n,\mathbb{C})$ to $(\mathbb{C}^*)^n\rtimes\sigma_n=K'$. The tangent bundle of $X$ thus always admits a reduction in structure group from $GL(n,\mathbb{C})$ to $K'$ i.e., we can write $\mathcal{T}_{X}\cong P\times_{K'}\mathfrak{g}^{-1,1}$, for some principal $K'=(\mathbb{C}^*)^n\rtimes\sigma_n$-bundle $P$ on $X$. 
\end{remark}

\section{Uniformization by Hermitian symmetric space of type \texorpdfstring{$CI$}{}}

For any $n\in\mathbb{Z}_{\ge1}$, the Siegel upper half space $\mathcal{H}_n$ is a Hermitian symmetric space of non-compact type, of dimension $n(n+1)/2$. The bounded symmetric domain realization of $CI$ is as follows (see \cite[Ch.4, Sec.2]{mok}).
\begin{align*}
    CI=\{Z\in\mathcal{D}_{n,n}:Z^T=Z\},
\end{align*}
where the domain $\mathcal{D}_{n,n}$ is given by
\begin{align}\label{mat}
    \mathcal{D}_{n,n}=\{Z\in M(n,n,\mathbb{C})\cong\mathbb{C}^{n^2}:I_n-\Bar{Z}^TZ>0\}.
\end{align}
In the above definition, the condition in the bracket is that the matrix $I_n-\Bar{Z}^TZ$ is positive definite. We know, for example from \cite[p.108]{take}, that the automorphism group of $\mathcal{H}_n$ is $PSp(2n,\mathbb{R})$, which is a Hodge group of Hermitian type. Another Hodge group associated to $\mathcal{H}_n$ is $G_0=Sp(2n,\mathbb{R})$, which is a cover of $PSp(2n,\mathbb{R})$ and has maximal compact subgroup $K_0=U(n)$. Thus we have $\mathcal{H}_n\cong G_0/K_0=Sp(2n,\mathbb{R})/U(n)$. The complexified Lie algebra $\mathfrak{g}$ of $G_0$ has complex dimension $n(2n+1)$, and it follows from the list on \cite[p.341]{helga}, that $\mathfrak{g}$ can be expressed as the set of complex $2n\times2n$ trace zero matrices $\begin{bmatrix}A&B\\C&-A^T\end{bmatrix}$, where $A$ is a complex $n\times n$ matrix, and $B$ and $C$ are complex symmetric $n\times n$ matrices. By definition, the Lie algebra $\mathfrak{g}$ admits a Hodge decomposition given by $\mathfrak{g}=\mathfrak{g}^{-1,1}\oplus\mathfrak{g}^{0,0}\oplus\mathfrak{g}^{1,-1}$, where $\mathfrak{g}^{-1,1}$ consists of matrices of the form $\begin{bmatrix}0&B\\0&0\end{bmatrix}$, $\mathfrak{g}^{0,0}$ consists of matrices of the form $\begin{bmatrix}A&0\\0&-A^T\end{bmatrix}$, and $\mathfrak{g}^{1,-1}$ consists of matrices of the form $\begin{bmatrix}0&0\\C&0\end{bmatrix}$. The Lie algebras $\mathfrak{g}^{-1,1}$, $\mathfrak{g}^{0,0}$, and $\mathfrak{g}^{1,-1}$ have complex dimensions $n(n+1)/2$, $n^2$, and $n(n+1)/2$ respectively.\\
The compact dual of $\mathcal{H}_n$ is the Lagrangian Grassmannian $Y_n$, which is a homogeneous complex projective variety of dimension $n(n+1)/2$ for the action of $G=Sp(2n,\mathbb{C})$. It parametrizes all Lagrangian (i.e. maximally isotropic) subspaces of a complex symplectic vector space of dimension $2n$. We know from \cite{vdg} that there is an open embedding $j:\mathcal{H}_n\hookrightarrow Y_n$, and the tangent bundle of $Y_n$ is given by $\mathcal{T}_{Y_n}\cong \textrm{Sym}^2(\mathcal{E})$, where $\mathcal{E}$ is the tautological vector bundle on $Y_n$, and corresponds to the standard representation of $K_0=U(n)$. It follows that the tangent bundle of $\mathcal{H}_n$ can be expressed as $\mathcal{T}_{\mathcal{H}_n}\cong\textrm{Sym}^2(\mathcal{E}')$, where $\mathcal{E}'$ denotes the restriction of $\mathcal{E}$ to $\mathcal{H}_n$.  \\
\\
We want to formulate necessary and sufficient conditions for a complex projective klt variety of dimension $n(n+1)/2$ with ample canonical divisor to be uniformized by the Siegel upper half space $\mathcal{H}_n$. To do this, we apply Theorem \ref{genthm}.

\subsection{Sufficient conditions}

\begin{proposition}\label{prop6}
Let $X$ be a projective, klt variety of dimension $n(n+1)/2$ for some $n\in\mathbb{Z}_{\ge1}$ with $K_X$ ample. Suppose that the tangent bundle of the regular locus $X_{reg}$ of $X$ satisfies $\mathcal{T}_{X_{reg}}\cong\textrm{Sym}^2(\mathcal{E})$, where $\mathcal{E}$ is a vector bundle of rank $n$ on $X_{reg}$. Let $\mathcal{E}'$ denote the reflexive extension of $\mathcal{E}$ to $X$, and suppose that the Chern class equality 
\begin{align}\label{cceq}
    [2\widehat{c}_2(X)-\widehat{c}_1(X)^2+2n\widehat{c}_2(\mathcal{E}')-(n-1)\widehat{c}_1(\mathcal{E}')^2]\cdot[K_X]^{n-2}=0
\end{align}
holds on $X$. Then $X\cong\mathcal{H}_n/\Gamma$, where $\Gamma\subset PSp(2n,\mathbb{R})$ is a discrete, cocompact subgroup, and acts fixed point freely in codimension one on the Siegel upper half space $\mathcal{H}_n$.
\end{proposition}
\begin{proof}
Let $j:X_{reg}\to X$ denote the natural inclusion. Then we have $\mathcal{T}_X\cong j_*\mathcal{T}_{X_{reg}}\cong\textrm{Sym}^{[2]}(\mathcal{E}')=\textrm{Sym}^2(\mathcal{E})^{**}$ by the uniqueness of reflexive extension. Let $G_0=Sp(2n,\mathbb{R})$, $K_0=U(n)$ a maximal compact subgroup of $G_0$, Then complexifications $G$ and $K$ of $G_0$ and $K_0$ respectively, are $G=Sp(2n,\mathbb{C})$ and $K=GL(n,\mathbb{C})$.\\
Let $P$ be the frame bundle of $\mathcal{E}$. Then $P$ is a principal $K=GL(n,\mathbb{C})$-bundle on $X_{reg}$ and we can write $\mathcal{E}\cong P\times_KV$, where $V$ denotes the typical fiber of $\mathcal{E}$. It follows that $\mathcal{T}_{X_{reg}}\cong \textrm{Sym}^2(P\times_KV)=P\times_K\textrm{Sym}^2(V)$, so in particular $\mathcal{T}_{X_{reg}}$ admits a reduction in structure group from $GL(n(n+1)/2,\mathbb{C})$ to $K$. From \cite[Table 2 on p.193]{maxs}, we see that $\textrm{Sym}^2(V)$ and $\mathfrak{g}^{-1,1}$ are isomorphic as $K$-representations. Thus there is an isomorphism
\begin{align*}
    \theta:\mathcal{T}_{X_{reg}}\to P\times_K\mathfrak{g}^{-1,1}
\end{align*}
such that $[\theta(u),\theta(v)]=0$ for all local sections $u,v$ of $\mathcal{T}_{X_{reg}}$. The latter statement is clear because $G_0$ is a Hodge group, so the Lie bracket of any two elements of $\mathfrak{g}^{-1,1}$ lands in $\mathfrak{g}^{-2,2}$, which is zero. It follows that $(P,\theta)$ is a uniformizing system of Hodge bundles on $X_{reg}$.\\
Hence we can form the system of Hodge bundles $E=P\times_K\mathfrak{g}=P\times_K\mathfrak{g}^{-1,1}\oplus P\times_K\mathfrak{g}^{0,0}\oplus P\times_K\mathfrak{g}^{1,-1}$, where we have $P\times_K\mathfrak{g}^{1,-1}\cong \textrm{Sym}^2(\mathcal{E}^{\vee})\cong\Omega^1_{X_{reg}}$, and $P\times_K\mathfrak{g}^{0,0}\cong\mathcal{E}nd(\mathcal{E})$. The last isomorphism follows from observing that $\mathfrak{g}^{0,0}$ and $\textrm{End}(V)$ are isomorphic as $K$-representations. Indeed, they are isomorphic as $\mathbb{C}$-vector spaces, and the adjoint action of $K$ on $\mathfrak{g}^{0,0}$ is compatible with the action of $K$ on $\textrm{End}(V)$ by conjugation.\\
Note that the system of Hodge bundles $E=P\times_K\mathfrak{g}\cong\mathcal{T}_{X_{reg}}\oplus \mathcal{E}nd(\mathcal{E})\oplus\Omega^1_{X_{reg}}$ is in particular a Higgs bundle with Higgs field $\widehat{\theta}:E\to E\otimes\Omega^1_{X_{reg}}$ given by sending a local section $u$ of $E$ to the local section $v\mapsto[\theta(v),u]$ of $\mathcal{H}om(\mathcal{T}_{X_{reg}},E)\cong E\otimes\Omega^1_{X_{reg}}$. It is clear that $E$ has slope zero with respect to $K_X$. By Proposition \ref{propstab}, we know that $E$ is $K_X$-polystable as a Higgs bundle on $X_{reg}$. In fact, $E$ is $K_X$-stable as a Higgs bundle on $X_{reg}$ because $\mathfrak{g}=\mathfrak{sp}(2n,\mathbb{C})$ is a simple Lie algebra.\\
Let $\mathcal{F}_X$ denote the reflexive extension of $E$ to $X$, then $\mathcal{F}_X\cong\mathcal{T}_X\oplus \mathcal{E}nd(\mathcal{E}')\oplus\Omega^{[1]}_X$. It is clear that $\widehat{c}_1(\mathcal{F}_X)\cdot[K_X]^{n-1}=0$. Moreover, we compute
\begin{align*}
\widehat{c}_2(\mathcal{F}_X)=2\widehat{c}_2(X)-\widehat{c}_1(X)^2+2n\widehat{c}_2(\mathcal{E}')-(n-1)\widehat{c}_1(X)^2, 
\end{align*}
so by the assumption of the Proposition we have $\widehat{c}_2(\mathcal{F}_X)\cdot[K_X]^{n-2}=\widehat{ch}_2(\mathcal{F}_X)\cdot[K_X]^{n-2}=0$.\\
Therefore, $X$ satisfies the two conditions of Theorem \ref{genthm}, and we conclude that $X\cong\mathcal{H}_n/\Gamma$, where $\Gamma\subset PSp(2n,\mathbb{R})$ is a discrete, cocompact subgroup acting fixed point freely in codimension one on $\mathcal{H}_n$.
\end{proof}


\subsection{Necessary conditions}

Since the automorphism group $\textrm{Aut}(\mathcal{H}_n)=PSp(2n,\mathbb{R})$ is a connected Lie group, we choose the associated Hodge group to be $G_0=Sp(2n,\mathbb{R})$, and its complexification to be $G=Sp(2n,\mathbb{C})$, which is also connected. Thus the sufficient conditions of Proposition \ref{prop6} are also necessary conditions for a projective klt variety with ample canonical divisor to be uniformized by $\mathcal{H}_n$.

\begin{proposition}\label{propn}
Let $X$ be a projective klt variety of dimension $n(n+1)/2$ with $K_X$ ample, such that $X=\mathcal{H}_n/\widehat{\Gamma}$, where $\widehat{\Gamma}\subset\textrm{Aut}(\mathcal{H}_n)$ acts discretely, cocompactly, and fixed point freely in codimension one on $\mathcal{H}_n$. Then the tangent bundle of the regular locus of $X$ satisfies $\mathcal{T}_{X_{reg}}\cong\textrm{Sym}^2(\mathcal{E})$, where $\mathcal{E}$ is a rank $n$ vector bundle on $X$. Moreover, $X$ satisfies the Chern class equality \ref{cceq}.
\end{proposition}
\begin{proof}
We know from Theorem \ref{genthm} that the smooth locus $X_{reg}$ of $X$ admits a uniformizing system of Hodge bundles $(P,\theta)$ corresponding to the Hodge group $G_0=Sp(2n,\mathbb{R})$, such that $E=P\times_K\mathfrak{g}$ is $K_X$-polystable as a Higgs bundle on $X_{reg}$, and $\widehat{c}_2(E')\cdot[K_X]^{n-2}=0$, where $E'$ denotes the unique reflexive extension of $E$ to $X$. Thus there is an isomorphism $\theta:\mathcal{T}_{X_{reg}}\cong P\times_K\mathfrak{g}^{-1,1}$. Recall from the proof of Proposition \ref{prop6} that $\mathfrak{g}^{-1,1}$ is isomorphic as a $K$-representation to $\textrm{Sym}^2(V)$, where $V$ is a complex $n$-dimensional vector space. So we can write $\mathcal{T}_{X_{reg}}\cong P\times_K\textrm{Sym}^2(V)= \textrm{Sym}^2(\mathcal{E})$, where $\mathcal{E}=P\times_K V$ is a rank $n$ vector bundle on $X_{reg}$. We also have $\Omega^1_{X_{reg}}\cong \textrm{Sym}^2(\mathcal{E}^{\vee})\cong P\times_K\mathfrak{g}^{1,-1}$, and $\mathcal{E}nd(\mathcal{E})\cong P\times_K\mathfrak{g}^{0,0}$, from which it follows that the system of Hodge bundles $P\times_K\mathfrak{g}$ is isomorphic to $\mathcal{T}_{X_{reg}}\oplus\mathcal{E}nd(\mathcal{E})\oplus\Omega^1_{X_{reg}}$.\\
Let $\mathcal{E}'$ denote the reflexive extension of $\mathcal{E}$ to $X$. Then $E'=\mathcal{T}_X\oplus\mathcal{E}nd(\mathcal{E}')\oplus\Omega^{[1]}_X$, and
the equality $\widehat{c}_2(E')\cdot[K_X]^{n-2}=0$ is equivalent to
\begin{align*}
    \widehat{c}_2(\mathcal{T}_X\oplus \mathcal{E}nd(\mathcal{E}')\oplus\Omega^{[1]}_X)\cdot[K_X]^{n-2}=0.
\end{align*}    
Using the formula for the second $\mathbb{Q}$-Chern class of a direct sum of reflexive sheaves, the above equality can be rephrased as
\begin{align*}
    [2\widehat{c}_2(X)-\widehat{c}_1(X)^2+2n\widehat{c}_2(\mathcal{E}')-(n-1)\widehat{c}_1(\mathcal{E}')^2]\cdot[K_X]^{n-2}=0.
\end{align*}
Thus it follows that $X$ satisfies the required Chern class equality (\ref{cceq}), this concludes the proof.
\end{proof}

Putting together Propositions \ref{prop6} and \ref{propn}, we arrive at Theorem \ref{equivthm}.

\section{Uniformization by Hermitian symmetric space of type \texorpdfstring{$DIII$}{}}

This example is very similar to the Siegel upper half space. The Hermitian symmetric space of type $DIII$, which we denote by $\mathcal{D}_n$, can be expressed as the quotient $G_0/K_0$, where the associated Hodge group is $G_0=SO^*(2n)$ defined as
\begin{align*}
    SO^*(2n)=\{M\in SL(2n,\mathbb{C}): M^TM=I_{2n}, \Bar{M}^TJ_nM=J_n\},
\end{align*}
where $J_n$ is the matrix given by $J_n=\begin{bmatrix}0&I_n\\-I_n&0\end{bmatrix}$. We know from \cite[Ch.4, Sec.2]{mok} that the maximal compact subgroup $K_0$ of $G_0$ is given by
\begin{align*}
    K_0=\left\{M=\begin{bmatrix}U&0\\0&\Bar{U}\end{bmatrix}:M\in SU(n,n)\right\},
\end{align*}
from which it is clear that $U\in U(n)\cong K_0$, for $n\in\mathbb{Z}_{\ge1}$. The bounded symmetric domain realization of $DIII$ is as follows (see again \cite[Ch.4, Sec.2]{mok}).
\begin{align*}
    DIII=\{Z\in\mathcal{D}_{n,n}:Z^T=-Z\},
\end{align*}
where the domain $\mathcal{D}_{n,n}$ is as defined in the expression (\ref{mat}). From \cite[p.111]{take}, the group of holomorphic automorphisms of $\mathcal{D}_n$ is $\textrm{Aut}(\mathcal{D}_n)=PSO^*(2n)$, so $\textrm{Aut}(\mathcal{D}_n)$ is a quotient of $G_0$ by a discrete central subgroup. From the list on \cite[p.341]{helga}, the complexified Lie algebra $\mathfrak{g}$ of $G_0$ has complex dimension $n(2n-1)$, and can be expressed as the set of complex $2n\times2n$ trace zero matrices $\begin{bmatrix}A&B\\C&-A^T\end{bmatrix}$, where $A$ is a complex $n\times n$ matrix, and $B$ and $C$ are complex skew-symmetric $n\times n$ matrices. Since $G_0$ is a Hodge group of Hermitian type, the Lie algebra $\mathfrak{g}$ has a Hodge decomposition given by $\mathfrak{g}=\mathfrak{g}^{-1,1}\oplus\mathfrak{g}^{0,0}\oplus\mathfrak{g}^{1,-1}$, where $\mathfrak{g}^{-1,1}$ consists of matrices of the form $\begin{bmatrix}0&B\\0&0\end{bmatrix}$, $\mathfrak{g}^{0,0}$ consists of matrices of the form $\begin{bmatrix}A&0\\0&-A^T\end{bmatrix}$, and $\mathfrak{g}^{1,-1}$ consists of matrices of the form $\begin{bmatrix}0&0\\C&0\end{bmatrix}$. The Lie algebras $\mathfrak{g}^{-1,1}$, $\mathfrak{g}^{0,0}$, and $\mathfrak{g}^{1,-1}$ have dimensions $n(n-1)/2$, $n^2$, and $n(n-1)/2$ respectively.\\
The compact dual of $\mathcal{D}_n$ is the Isotropic Grassmannian $I_n$, which is a homogeneous complex projective variety of dimension $n(n-1)/2$ for the action of $G=SO^*(2n,\mathbb{C})$. It parametrizes all isotropic subspaces of a complex vector space of dimension $2n$ equipped with an inner product. Just as in the case of the Lagrangian Grassmannian, there is an open embedding $j:\mathcal{D}_n\hookrightarrow I_n$. The tangent bundle of $I_n$ is given by $\mathcal{T}_{I_n}\cong \bigwedge^2(\mathcal{E})$, where $\mathcal{E}$ is the tautological vector bundle on $I_n$, and corresponds to the standard representation of $K_0=U(n)$. It follows that the tangent bundle of $\mathcal{D}_n$ can be expressed as $\mathcal{T}_{\mathcal{D}_n}\cong \bigwedge^2(\mathcal{E}')$, where $\mathcal{E}'$ denotes the restriction of $\mathcal{E}$ to $\mathcal{D}_n$.  \\
\\
We want to formulate necessary and sufficient conditions for a complex projective variety of dimension $n(n-1)/2$ with klt singularities and ample canonical divisor to be uniformized by the Hermitian symmetric space $\mathcal{D}_n$ of type DIII.

\subsection{Sufficient conditions}

A version of Proposition \ref{prop6} also holds in this case, and the proof is essentially the same. For completeness, we repeat the proof.

\begin{proposition}\label{prop6b}
Let $X$ be a projective, klt variety of dimension $n(n-1)/2$ for some $n\in\mathbb{Z}_{\ge1}$ with $K_X$ ample. Suppose that the tangent bundle of the regular locus $X_{reg}$ of $X$ satisfies $\mathcal{T}_{X_{reg}}\cong \bigwedge^2(\mathcal{E})$, where $\mathcal{E}$ is a vector bundle of rank $n$ on $X_{reg}$. Let $\mathcal{E}'$ denote the reflexive extension of $\mathcal{E}$ to $X$, and suppose that the Chern class equality 
\begin{align}\label{cceq2}
    [2\widehat{c}_2(X)-\widehat{c}_1(X)^2+2n\widehat{c}_2(\mathcal{E}')-(n-1)\widehat{c}_1(\mathcal{E}')^2]\cdot[K_X]^{n-2}=0
\end{align}
holds on $X$. Then $X\cong\mathcal{D}_n/\Gamma$, where $\Gamma\subset PSO^*(2n,\mathbb{R})$ is a discrete, cocompact subgroup, and acts fixed point freely in codimension one on the domain $\mathcal{D}_n$ of type $DIII$. 
\end{proposition}
\begin{proof}
Let $j:X_{reg}\to X$ denote the natural inclusion. Then we have $\mathcal{T}_X\cong j_*\mathcal{T}_{X_{reg}}\cong \bigwedge^{[2]}(\mathcal{E}')=(\bigwedge^2(\mathcal{E}'))^{**}$ by the uniqueness of reflexive extension. Let $G_0=SO^*(2n,\mathbb{R})$, $K_0=U(n)$ a maximal compact subgroup of $G_0$. Then we choose complexifications $G$ and $K$ of $G_0$ and $K_0$ respectively, to be $G=SO^*(2n,\mathbb{C})$ and $K=GL(n,\mathbb{C})$.\\
Let $P$ be the frame bundle of $\mathcal{E}$. Then $P$ is a principal $K=GL(n,\mathbb{C})$-bundle on $X_{reg}$ and we can write $\mathcal{E}\cong P\times_KV$, where $V\cong\mathbb{C}^n$ denotes the typical fiber of $\mathcal{E}$. It follows that $\mathcal{T}_{X_{reg}}\cong \bigwedge^2(P\times_KV)=P\times_K\bigwedge^2(V)$, so in particular $\mathcal{T}_{X_{reg}}$ admits a reduction in structure group from $GL(n(n-1)/2,\mathbb{C})$ to $K$. Note that $\bigwedge^2(V)$ and $\mathfrak{g}^{-1,1}$ are isomorphic as $K$-representations, again by \cite[Table 2]{maxs}. Thus there is an isomorphism
\begin{align*}
    \theta:\mathcal{T}_{X_{reg}}\to P\times_K\mathfrak{g}^{-1,1}
\end{align*}
such that $[\theta(u),\theta(v)]=0$ for all local sections $u,v$ of $\mathcal{T}_{X_{reg}}$. It follows that $(P,\theta)$ is a uniformizing system of Hodge bundles on $X_{reg}$.\\
Hence we can form the system of Hodge bundles $E=P\times_K\mathfrak{g}=P\times_K\mathfrak{g}^{-1,1}\oplus P\times_K\mathfrak{g}^{0,0}\oplus P\times_K\mathfrak{g}^{1,-1}$, where we have $P\times_K\mathfrak{g}^{1,-1}\cong \bigwedge^2(\mathcal{E}^{\vee})\cong\Omega^1_{X_{reg}}$, and $P\times_K\mathfrak{g}^{0,0}\cong\mathcal{E}nd(\mathcal{E})$. The last isomorphism follows from observing, as in Proposition \ref{prop6}, that $\mathfrak{g}^{0,0}$ and $\textrm{End}(V)$ are isomorphic as $K$-representations. Hence we have $P\times_K\mathfrak{g}^{0,0}\cong P\times_K\textrm{End}(V)=\mathcal{E}nd(P\times_KV)$.\\
Note that the system of Hodge bundles $E=P\times_K\mathfrak{g}\cong\mathcal{T}_{X_{reg}}\oplus \mathcal{E}nd(\mathcal{E})\oplus\Omega^1_{X_{reg}}$ is in particular a Higgs bundle with Higgs field $\widehat{\theta}:E\to E\otimes\Omega^1_{X_{reg}}$ given by sending a local section $u$ of $E$ to the local section $v\mapsto[\theta(v),u]$ of $\mathcal{H}om(\mathcal{T}_{X_{reg}},E)\cong E\otimes\Omega^1_{X_{reg}}$. It is clear that $E$ has slope zero with respect to $K_X$. We know from Proposition \ref{propstab} that $E$ is $K_X$-polystable as a Higgs bundle on $X_{reg}$. In fact, $E$ is $K_X$-stable as a Higgs bundle on $X_{reg}$ because again $\mathfrak{g}$ is a simple Lie algebra.\\
Let $\mathcal{F}_X$ denote the reflexive extension of $E$ to $X$, then $\mathcal{F}_X\cong\mathcal{T}_X\oplus \mathcal{E}nd(\mathcal{E}')\oplus\Omega^{[1]}_X$. It is clear that $\widehat{c}_1(\mathcal{F}_X)\cdot[K_X]^{n-1}=0$. Moreover, we compute 
\begin{align*}
\widehat{c}_2(\mathcal{F}_X)=2\widehat{c}_2(X)-\widehat{c}_1(X)^2+2n\widehat{c}_2(\mathcal{E}')-(n-1)\widehat{c}_1(X)^2,
\end{align*}
so by the assumption of the Proposition we have $\widehat{c}_2(\mathcal{F}_X)\cdot[K_X]^{n-2}=\widehat{ch}_2(\mathcal{F}_X)\cdot[K_X]^{n-2}=0$. \\
Again we see that $X$ satisfies the two conditions of Theorem \ref{genthm}, and we conclude that $X\cong\mathcal{D}_n/\Gamma$, where $\Gamma\subset PSO^*(2n,\mathbb{R})$ is a discrete, cocompact subgroup, acting fixed point freely in codimension one on $\mathcal{D}_n$.
\end{proof}

\subsection{Necessary conditions}

Since $\textrm{Aut}(\mathcal{D}_n)=PSO^*(2n)$ is a connected Lie group, we take the associated Hodge group to be $G_0=SO^*(2n)$, and its complexification to be $G=SO^*(2n,\mathbb{C})$, which is also connected. Moreover, $K_0=U(n)$ and its complexification $K=GL(n,\mathbb{C})$ are connected too, so the sufficient conditions of Proposition \ref{prop6b} are also necessary conditions for a projective klt variety with ample canonical divisor to be uniformized by $\mathcal{D}_n$.
The proof is essentialy the same as that of Proposition \ref{propn}.

\begin{proposition}\label{propnb}
Let $X$ be a projective klt variety of dimension $n(n-1)/2$ with $K_X$ ample, such that $X=\mathcal{D}_n/\widehat{\Gamma}$, where $\widehat{\Gamma}\subset\textrm{Aut}(\mathcal{D}_n)$ acts discretely, cocompactly, and fixed point freely in codimension one on $\mathcal{D}_n$. Then the tangent bundle of the regular locus of $X$ satisfies $\mathcal{T}_{X_{reg}}\cong \bigwedge^2(\mathcal{E})$, where $\mathcal{E}$ is a rank $n$ vector bundle on $X_{reg}$. Moreover, $X$ satisfies the Chern class equality \ref{cceq2}.
\end{proposition}
\begin{proof}
From Theorem \ref{genthm}, it follows that $X_{reg}$ admits a uniformizing system of Hodge bundles $(P,\theta)$ corresponding to the Hodge group $G_0=SO^*(2n)$, such that the system of Hodge bundles $E=P\times_K\mathfrak{g}$ is $K_X$-polystable as a Higgs bundle on $X_{reg}$. Moreover, the equality $\widehat{c}_2(E')\cdot[K_X]^{n-2}=0$. holds, where $E'$ denotes the extension of $E$ to $X$ as a reflexive sheaf. Hence there is an isomorphism $\theta:\mathcal{T}_Y\cong P\times_K\mathfrak{g}^{-1,1}$. Recall from the proof of Proposition \ref{prop6b} that $\mathfrak{g}^{-1,1}$ is isomorphic as a $K$-representation to $\bigwedge^2(V)$, where $V$ is a complex $n$-dimensional vector space. So we can write $\mathcal{T}_{X_{reg}}\cong P\times_K\bigwedge^2(V)= \bigwedge^2(\mathcal{E})$, where $\mathcal{E}=P\times_K V$ is a rank $n$ vector bundle on $Y$. We also have $\Omega^1_{X_{reg}}\cong \bigwedge^2(\mathcal{E}^{\vee})\cong P\times_K\mathfrak{g}^{1,-1}$, and $\mathcal{E}nd(\mathcal{E})\cong P\times_K\mathfrak{g}^{0,0}$, from which it follows that the system of Hodge bundles $P\times_K\mathfrak{g}$ is isomorphic to $\mathcal{T}_{X_{reg}}\oplus \mathcal{E}nd(\mathcal{E})\oplus\Omega^1_{X_{reg}}$.\\
Let $\mathcal{E}'$ denote the reflexive extension of $\mathcal{E}$ to $X$. Then $E'=\mathcal{T}_X\oplus\mathcal{E}nd(\mathcal{E}')\oplus\Omega^{[1]}_X$, and the Chern class equality $\widehat{c}_2(E')\cdot[K_X]^{n-2}=0$ is equivalent to
\begin{align*}
    \widehat{c}_2(\mathcal{T}_X\oplus \mathcal{E}nd(\mathcal{E}')\oplus\Omega^{[1]}_X)\cdot[K_X]^{n-2}=0.
\end{align*}    
Using the formula for the second $\mathbb{Q}$-Chern class of a direct sum of vector bundles, the above equality can be rephrased as
\begin{align*}
    [2\widehat{c}_2(X)-\widehat{c}_1(X)^2+2n\widehat{c}_2(\mathcal{E})-(n-1)\widehat{c}_1(\mathcal{E})^2]\cdot[K_X]^{n-2}=0.
\end{align*}
Thus it follows that $X$ satisfies the Chern class equality (\ref{cceq2}), and we are done.
\end{proof}
Putting together Propositions \ref{prop6b} and \ref{propnb}, we arrive at the following necessary and sufficient condition for a projective klt variety $X$ with ample canonical divisor to be uniformized by the Hermitian symmetric space $\mathcal{D}_n$ of type $DIII$. 

\begin{theorem}\label{equivthmb}
Let $X$ be a projective klt variety of dimension $n(n-1)/2$ such that the canonical divisor $K_X$ is ample. Then $X\cong\mathcal{D}_n/\Gamma$, where $\Gamma\subset PSO^*(2n,\mathbb{R})$ is a discrete, cocompact subgroup, and acts fixed point freely in codimension one on $\mathcal{D}_n$, if and only if $X$ satisfies
\begin{itemize}
    \item $\mathcal{T}_{X_{reg}}\cong \bigwedge^2(\mathcal{E})$
    \item $[2\widehat{c}_2(X)-\widehat{c}_1(X)^2+2n\widehat{c}_2(\mathcal{E}')-(n-1)\widehat{c}_1(\mathcal{E}')^2]\cdot[K_X]^{n-2}=0$,
\end{itemize}
where $\mathcal{E}$ is a vector bundle of rank $n$ on $X_{reg}$, and $\mathcal{E}'$ denotes the reflexive extension of $\mathcal{E}$ to $X$.
\end{theorem}

\section{Uniformization by Hermitian symmetric space of type \texorpdfstring{$BDI$}{}}

The Hermitian symmetric space of type $BDI$, which we denote by $\mathcal{B}_n$, can be expressed as the quotient $G_0/K_0$, where we may take the associated Hodge group to be $G_0=SO(2,n)$, and $K_0=SO(2)\times SO(n)$ a maximal compact subgroup of $G_0$, for $n\in\mathbb{Z}_{\ge2}$. The bounded symmetric domain realization of $BDI$ is more involved than the other examples, we refer the reader to \cite[p.75]{mok} for details. When $n=2$, the domain $\mathcal{B}_2$ is isomorphic to the bidisk $\mathbb{H}^2$, which we have already studied. Thus we may assume $n\ge3$. The Lie group $G_0$ has two connected components, and the automorphism group $PSO(2,n)$ of $\mathcal{B}_n$ (see again \cite{take}) is a quotient of $SO(2,n)$ by a discrete central subgroup. The dimension of $\mathcal{B}_n$ as a complex manifold is $n$. The complexified Lie algebra of $G_0$ is $\mathfrak{g}=\mathfrak{so}(2+n,\mathbb{C})$. The Lie algebra $\mathfrak{g}$ has complex dimension $(2+n)(1+n)/2$, and using \cite[p.341]{helga}, can be expressed as the set of complex $(n+2)\times(n+2)$ matrices $\begin{bmatrix}A&B\\C&D\end{bmatrix}$, where $A$ is a $2\times2$ orthogonal matrix, $D$ is a $n\times n$ orthogonal matrix, and $B$ and $C$ are $2\times n$ and $n\times2$ matrices respectively. The group $G_0=SO(2,n)$ is in fact a Hodge group of Hermitian type, therefore the Lie algebra $\mathfrak{g}$ has a Hodge decomposition given by $\mathfrak{g}=\mathfrak{g}^{-1,1}\oplus\mathfrak{g}^{0,0}\oplus\mathfrak{g}^{1,-1}$, where $\mathfrak{g}^{-1,1}$ consists of those matrices of $\mathfrak{g}$ of the form $\begin{bmatrix}0&B\\0&0\end{bmatrix}$, $\mathfrak{g}^{0,0}$ consists of matrices of the form $\begin{bmatrix}A&0\\0&D\end{bmatrix}$, and $\mathfrak{g}^{1,-1}$ consists of matrices of the form $\begin{bmatrix}0&0\\C&0\end{bmatrix}$. The Lie algebras $\mathfrak{g}^{-1,1}$, $\mathfrak{g}^{0,0}$, and $\mathfrak{g}^{1,-1}$ are of dimensions $n$, $n(n-1)/2+1$, and $n$ respectively. \\
By Lemma \ref{lemtan}, we know there is be a principal $K=SO(2,\mathbb{C})\times SO(n,\mathbb{C})$-bundle $P$ on $\mathcal{B}_n$ such that there is an isomorphism $\theta:\mathcal{T}_{\mathcal{B}_n}\to P\times_K\mathfrak{g}^{-1,1}$. Then there is an orthogonal vector bundle $\mathcal{W}$ of rank $n$, and a line bundle $\mathcal{L}$ such that the system of Hodge bundles $P\times_K\mathfrak{g}$ is isomorphic to $\bigwedge^2(\mathcal{W}\oplus\mathcal{L}\oplus\mathcal{L}^\vee)$. From the Hodge decomposition of the Lie algebra $\mathfrak{g}$, it follows that $P\times_K\mathfrak{g}^{-1,1}\cong\mathcal{T}_{\mathcal{B}_n}\cong \mathcal{H}om(\mathcal{W},\mathcal{L})$, $P\times_K\mathfrak{g}^{0,0}\cong\bigwedge^2(\mathcal{W})\oplus\mathcal{O}_{\mathcal{B}_n}$, and $P\times_K\mathfrak{g}^{1,-1}\cong\Omega^1_{\mathcal{B}_n}\cong \mathcal{H}om(\mathcal{W},\mathcal{L}^\vee)$. Note that, since $\mathcal{W}$ is an orthogonal bundle, it is self-dual.\\
\\
We want to determine necessary and sufficient conditions for a projective, klt variety of dimension $n$ with ample canonical divisor to be uniformized by $\mathcal{B}_n$. 

\subsection{Sufficient conditions}

We apply Theorem \ref{genthm} to prove the following.

\begin{proposition}\label{prop6d}
Let $X$ be a projective, klt variety of dimension $n\ge3$ with $K_X$ ample. Suppose that the tangent bundle of the regular locus $X_{reg}$ of $X$ satisfies $\mathcal{T}_{X_{reg}}\cong \mathcal{H}om(\mathcal{W},\mathcal{L})$, where $\mathcal{W}$ is an orthogonal vector bundle of rank $n$, and $\mathcal{L}$ is a line bundle on $X_{reg}$. Let $\mathcal{W}'$ and $\mathcal{L}'$ denote the reflexive extensions of $\mathcal{W}$ and $\mathcal{L}$ respectively to $X$, and suppose that the Chern class equality 
\begin{align}\label{cceq4}
    [\widehat{c}_2(\mathcal{W}'\wedge\mathcal{W}')+2\widehat{c}_2(\mathcal{W}')-n\widehat{c}_1(\mathcal{L}')^2]\cdot[K_X]^{n-2}=0
\end{align}
holds on $X$. Then $X\cong\mathcal{B}_n/\Gamma$, where $\Gamma\subset\textrm{Aut}(\mathcal{B}_n)$ is a discrete, cocompact subgroup, and acts fixed point freely in codimension one on $\mathcal{B}_n$.
\end{proposition}
\begin{proof}
Let $j:X_{reg}\to X$ denote the natural inclusion. Then we have $\mathcal{T}_X\cong j_*\mathcal{T}_{X_{reg}}\cong j_*\mathcal{H}om(\mathcal{W},\mathcal{L})\cong \mathcal{H}om(\mathcal{W}',\mathcal{L}')$ by the uniqueness of reflexive extension. Let $G_0=SO(2,n)$, $K_0=SO(2)\times SO(n)$ the maximal compact subgroup of $G_0$, and let $G$ and $K$ denote complexifications of $G_0$ and $K_0$ respectively. We choose $G=SO(2+n,\mathbb{C})$, then $K=SO(2,\mathbb{C})\times SO(n,\mathbb{C})$.\\
We can deduce again from \cite[Table 2]{maxs} that the Lie algebra $\mathfrak{g}=\mathfrak{so}(2+n,\mathbb{C})$ and the typical fiber of the vector bundle $\bigwedge^2(\mathcal{W}\oplus\mathcal{L}\oplus\mathcal{L}^\vee)$ are isomorphic as $K$-representations. Moreover, the vector bundle $\bigwedge^2(\mathcal{W}\oplus\mathcal{L}\oplus\mathcal{L}^\vee)$ admits a reduction in structure group from $GL((n+2)(n+1)/2,\mathbb{C})$ to $K=SO(2,\mathbb{C})\times SO(n,\mathbb{C})$. Thus there is a principal $K$-bundle $P$ on $X_{reg}$ such that $P\times_K\mathfrak{g}\cong \bigwedge^2(\mathcal{W}\oplus\mathcal{L}\oplus\mathcal{L}^\vee)$. Recall from the earlier discussion that the Hodge decomposition of $\mathfrak{g}$ gives $P\times_K\mathfrak{g}^{-1,1}\cong \mathcal{H}om(\mathcal{W},\mathcal{L})$, $P\times_K\mathfrak{g}^{0,0}\cong\bigwedge^2\mathcal{W}\oplus\mathcal{O}_{X_{reg}}$, and $P\times_K\mathfrak{g}^{1,-1}\cong 
\mathcal{H}om(\mathcal{W},\mathcal{L}^\vee)$. Since $\mathcal{T}_{X_{reg}}\cong \mathcal{H}om(\mathcal{W},\mathcal{L})$ by assumption, we have the following isomorphism
\begin{align*}
    \theta:\mathcal{T}_{X_{reg}}\to P\times_K\mathfrak{g}^{-1,1}
\end{align*}
such that $[\theta(u),\theta(v)]=0$ for all local sections $u,v$ of $\mathcal{T}_{X_{reg}}$. It follows that $(P,\theta)$ gives a uniformizing system of Hodge bundles on $X_{reg}$.\\
The system of Hodge bundles $E=P\times_K\mathfrak{g}\cong\mathcal{T}_{X_{reg}}\oplus \bigwedge^2\mathcal{W}\oplus\mathcal{O}_{X_{reg}}\oplus\Omega^1_{X_{reg}}$ is in particular a Higgs bundle with Higgs field $\widehat{\theta}:E\to E\otimes\Omega^1_{X_{reg}}$ given by sending a local section $u$ of $E$ to the local section $v\mapsto[\theta(v),u]$ of $\mathcal{H}om(\mathcal{T}_{X_{reg}},E)\cong E\otimes\Omega^1_{X_{reg}}$. Since $\mathcal{W}$ is self-dual, we have that $\textrm{det}(\mathcal{W})\cong\mathcal{O}_{X_{reg}}$, and hence $c_1(E)=c_1(P\times_K\mathfrak{g})=0$. We know again from Proposition \ref{propstab} that the system of Hodge bundles $E$ is $K_X$-polystable as a Higgs bundle on $X_{reg}$. In fact, $E$ is $K_X$-stable as a Higgs bundle on $X_{reg}$ because $\mathfrak{so}(2+n,\mathbb{C})$ is a simple Lie algebra for $n\ge3$.\\
Let $\mathcal{F}_X$ denote the reflexive extension of $E$ to $X$. It is clear that $\widehat{c}_1(\mathcal{F}_X)\cdot[K_X]^{n-1}=0$. Moreover, we have
\begin{align*}
\widehat{c}_2(\mathcal{F}_X)\cdot[K_X]^{n-2}=\widehat{c}_2(\bigwedge^2(\mathcal{W}'\oplus\mathcal{L}'\oplus\mathcal{L}'^\vee))\cdot[K_X]^{n-2}. 
\end{align*}
Expanding the right hand side of the above equality gives $[\widehat{c}_2(\mathcal{W}'\wedge\mathcal{W}')+2\widehat{c}_2(\mathcal{W}')-n\widehat{c}_1(\mathcal{L}')^2]\cdot[K_X]^{n-2}$, which by assumption is zero. Therefore, we have $\widehat{ch}_2(\mathcal{F}_X)\cdot[K_X]^{n-2}=0$.\\
It follows from Theorem \ref{genthm} that $X\cong\mathcal{B}_n/\Gamma$, where $\Gamma\subset\textrm{Aut}(\mathcal{B}_n)$ is a discrete, cocompact subgroup, and acts fixed point freely in codimension one on $\mathcal{B}_n$. This concludes the proof. 
\end{proof}

\subsection{Necessary conditions}

Note that the full automorphism group of $\mathcal{B}_n$ is a quotient of $G_0=SO(2,n)$ by a discrete central subgroup, and thus the complexification of the maximal compact subgroup of $\textrm{Aut}(\mathcal{B}_n)$ is a quotient of $K=SO(2,\mathbb{C})\times SO(n,\mathbb{C})$ by a discrete central subgroup. Since $K$ is connected, the sufficient criteria for uniformization by $\mathcal{B}_n$ are also necessary.

\begin{proposition}\label{propnd}
Let $X$ be a projective klt variety of dimension $n\ge3$ with $K_X$ ample, such that $X=\mathcal{B}_n/\widehat{\Gamma}$, where $\widehat{\Gamma}\subset Aut(\mathcal{B}_n)$ is a discrete, cocompact subgroup which acts fixed point freely in codimension one on $\mathcal{B}_n$. Then the tangent bundle of the regular locus of $X$ satisfies $\mathcal{T}_{X_{reg}}\cong \mathcal{H}om(\mathcal{V},\mathcal{M})$, where $\mathcal{V}$ is an orthogonal vector bundle of rank $n$, and $\mathcal{M}$ is a line bundle on $X_{reg}$. Moreover, $X$ satisfies the Chern class equality \ref{cceq4}.
\end{proposition}
\begin{proof}
Again from Theorem \ref{genthm}, it follows that the smooth locus $X_{reg}$ of $X$ admits a uniformizing system of Hodge bundles $(P,\theta)$ corresponding to the Hodge group $G_0=SO(2,n)$, such that $E=P\times_K\mathfrak{g}$ is $K_X$-polystable as a Higgs bundle on $X_{reg}$. Furthermore, the equality $\widehat{c}_2(E')\cdot[K_X]^{n-2}=0$ is satisfied, where $E'$ denotes the reflexive extension of $E$ to $X$. Hence there is an isomorphism $\theta:\mathcal{T}_{X_{reg}}\cong P\times_K\mathfrak{g}^{-1,1}$. Recall that from the proof of Proposition \ref{prop6d} that $\mathfrak{g}^{-1,1}$ is isomorphic as a $K$-representation to $\textrm{Hom}(W,L)$, where $W$ and $L$ are complex $n$ and 1-dimensional vector spaces respectively. So we can write $\mathcal{T}_{X_{reg}}\cong P\times_K\textrm{Hom}(W,L)=\mathcal{H}om(\mathcal{W},\mathcal{L})$, where $\mathcal{W}$ is an orthogonal rank $n$ vector bundle, and $\mathcal{L}$ is a line bundle on $X_{reg}$. We also have $\Omega^1_{X_{reg}}\cong \mathcal{H}om(\mathcal{W},\mathcal{L}^\vee)\cong P\times_K\mathfrak{g}^{1,-1}$, and $\bigwedge^2\mathcal{W}\oplus\mathcal{O}_{X_{reg}}\cong P\times_K\mathfrak{g}^{0,0}$, from which it follows that the system of Hodge bundles $P\times_K\mathfrak{g}$ is isomorphic to $\bigwedge^2(\mathcal{W}\oplus\mathcal{L}\oplus\mathcal{L}^\vee)$.\\
Let $E'$, $\mathcal{W}'$, and $\mathcal{L}'$ denote the reflexive extensions of $E$, $\mathcal{W}$, and $\mathcal{L}$ respectively to $X$. Then the equality $\widehat{c}_2(E')\cdot[K_X]^{n-2}=0$ is equivalent to
\begin{align*}
    \widehat{c}_2(\bigwedge^2(\mathcal{W}'\oplus\mathcal{L}'\oplus\mathcal{L}'^\vee))\cdot[K_X]^{n-2}=0.
\end{align*}    
Expanding the above expression using formulae for $\mathbb{Q}$-Chern classes, we arrive at the following expression.
\begin{align*}    
    [\widehat{c}_2(\mathcal{W}'\wedge\mathcal{W}')+2\widehat{c}_2(\mathcal{W}')-n\widehat{c}_1(\mathcal{L}')^2]\cdot[K_X]^{n-2}=0.
\end{align*}
Therefore, $X$ satisfies the $\mathbb{Q}$-Chern class equality (\ref{cceq4}).
\end{proof}

Putting together Propositions \ref{prop6d} and \ref{propnd}, we arrive at the following necessary and sufficient condition for a projective klt variety $X$ with ample canonical divisor to be uniformized by the Hermitian symmetric space $\mathcal{B}_n$ of type $BDI$. 

\begin{theorem}\label{equivthmd}
Let $X$ be a projective klt variety of dimension $n\ge3$, such that the canonical divisor $K_X$ is ample. Then $X\cong\mathcal{B}_n/\Gamma$, where $\Gamma\subset\textrm{Aut}(\mathcal{B}_n)$ is a discrete, cocompact subgroup acting fixed point freely in codimension one on $\mathcal{B}_n$, if and only if $X$ satisfies
\begin{itemize}
    \item $\mathcal{T}_{X_{reg}}\cong\mathcal{H}om(\mathcal{W},\mathcal{L})$
    \item $[\widehat{c}_2(\mathcal{W}'\wedge\mathcal{W}')+2\widehat{c}_2(\mathcal{W}')-n\widehat{c}_1(\mathcal{L}')^2]\cdot[K_X]^{n-2}=0$,
\end{itemize}
where $\mathcal{W}$ is an orthogonal vector bundle of rank $n$, $\mathcal{L}$ is a line bundle on $X_{reg}$, and $\mathcal{W}'$ and $\mathcal{L}'$ denote the reflexive extensions of $\mathcal{W}$ and $\mathcal{L}$ to $X$.
\end{theorem}

\section{Uniformization by Hermitian symmetric space of type \texorpdfstring{$AIII$}{}}

The Hermitian symmetric space of type $AIII$, denoted $\mathcal{A}_{pq}$, can be expressed as the quotient $G_0/K_0$, where we may take $G_0=SU(p,q)$, and $K_0=S(U(p)\times U(q))$ its maximal compact subgroup, for $p,q\in\mathbb{Z}_{\ge1}$, $pq>1$. The domain $\mathcal{A}_{pq}$ can also be realized as follows (see \cite{mok}) 
\begin{align*}
    \mathcal{A}_{pq}=\{Z\in M(p,q,\mathbb{C})\cong\mathbb{C}^{pq}:I_q-\Bar{Z}^TZ>0\},
\end{align*}
where the "$>0$" again means positive definiteness. The complex dimension of $\mathcal{A}_{pq}$ is $pq$. The complexified Lie algebra of $SU(p,q)$ is $\mathfrak{g}=\mathfrak{sl}(p+q,\mathbb{C})$. The Lie algebra $\mathfrak{g}$ has complex dimension $(p+q)^2-1$, and again using the list on \cite[p.341]{helga}, can expressed as the set of complex $(p+q)\times(p+q)$ matrices $\begin{bmatrix}A&B\\C&D\end{bmatrix}$, where $A$ is a $p\times p$ matrix, $B$ is a $p\times q$ matrix, $C$ is a $q\times p$ matrix, and $D$ is a $q\times q$ matrix, and $Tr(A)+Tr(D)=0$. Note that $G_0=SU(p,q)$ is a Hodge group of Hermitan type, and $\mathfrak{g}$ has a Hodge decomposition $\mathfrak{g}=\mathfrak{g}^{-1,1}\oplus\mathfrak{g}^{0,0}\oplus\mathfrak{g}^{1,-1}$ as in Definition \ref{def7}, where $\mathfrak{g}^{-1,1}$ consists of matrices of the form $\begin{bmatrix}0&B\\0&0\end{bmatrix}$, $\mathfrak{g}^{0,0}$ consists of matrices of the form $\begin{bmatrix}A&0\\0&D\end{bmatrix}$, and $\mathfrak{g}^{1,-1}$ consists of matrices of the form $\begin{bmatrix}0&0\\C&0\end{bmatrix}$. The Lie algebras $\mathfrak{g}^{-1,1}$, $\mathfrak{g}^{0,0}$, and $\mathfrak{g}^{1,-1}$ have complex dimensions $pq$, $p^2+q^2-1$, and $pq$ respectively.\\
Let $P$ be a principal $K=S(GL(p,\mathbb{C})\times GL(q,\mathbb{C}))$-bundle on $\mathcal{A}_{pq}$ such that $\theta:\mathcal{T}_{\mathcal{A}_{pq}}\cong P\times_K\mathfrak{g}^{-1,1}$, which we know exists by Lemma \ref{lemtan}. Then there are vector bundles $\mathcal{V}$ and $\mathcal{W}$ of ranks $p$ and $q$ on $\mathcal{A}_{pq}$, such that the system of Hodge bundles $P\times_K\mathfrak{g}$ is isomorphic to $\mathcal{E}nd_0(\mathcal{V}\oplus\mathcal{W})$, the bundle of trace zero endomorphisms of $\mathcal{V}\oplus\mathcal{W}$. From the Hodge decomposition of $\mathfrak{g}$, we get $P\times_K\mathfrak{g}^{-1,1}\cong \mathcal{H}om(\mathcal{V},\mathcal{W})\cong\mathcal{T}_{\mathcal{A}_{pq}}$, $P\times_K\mathfrak{g}^{0,0}\cong(\mathcal{E}nd(\mathcal{V})\oplus \mathcal{E}nd(\mathcal{W}))_0$, and $P\times_K\mathfrak{g}^{1,-1}\cong \mathcal{H}om(\mathcal{W},\mathcal{V})\cong\Omega^1_{\mathcal{A}_{pq}}$. The domain $\mathcal{A}_q$ corresponding to $p=1$ is the unit ball $\mathbb{B}^q\subset\mathbb{C}^q$.\\
\\
As usual, the goal is to formulate necessary and sufficient conditions for a projective, klt variety of dimension $pq$ with ample canonical divisor to be uniformized by $\mathcal{A}_{pq}$. 

\subsection{Sufficient conditions}

We again apply Theorem \ref{genthm} to make the following characterization.

\begin{proposition}\label{prop6c}
Let $X$ be a projective, klt variety of dimension $pq$ for some $p,q\in\mathbb{Z}_{\ge1}$, $pq>1$, with $K_X$ ample, such that the tangent bundle of the regular locus $X_{reg}$ of $X$ satisfies $\mathcal{T}_{X_{reg}}\cong\mathcal{H}om(\mathcal{V},\mathcal{W})$, where $\mathcal{V}$ and $\mathcal{W}$ are vector bundles of ranks $p$ and $q$ on $X_{reg}$. Let $\mathcal{V}'$ and $\mathcal{W}'$ denote the reflexive extensions of $\mathcal{V}$ and $\mathcal{W}$ to $X$, and suppose that the Chern class equality 
\begin{align}\label{cceq3}
    [2(p+q)(\widehat{c}_2(\mathcal{V}')+\widehat{c}_2(\mathcal{W}'))-(p+q-1)(\widehat{c}_1(\mathcal{V}')^2+\widehat{c}_1(\mathcal{W}')^2)+2\widehat{c}_1(\mathcal{V}')\widehat{c}_1(\mathcal{W}')]\cdot[K_X]^{n-2}=0
\end{align}
holds on $X$. Then $X\cong\mathcal{A}_{pq}/\Gamma$, where $\Gamma\subset\textrm{Aut}(\mathcal{A}_{pq})$ is a discrete, cocompact subgroup, and acts fixed point freely in codimension one on the domain $\mathcal{A}_{pq}$ of type $AIII$.
\end{proposition}
\begin{proof}
Let $j:X_{reg}\to X$ denote the natural inclusion. Then we have $\mathcal{T}_X\cong j_*\mathcal{T}_{X_{reg}}\cong j_*\mathcal{H}om(\mathcal{V},\mathcal{W})\cong \mathcal{H}om(\mathcal{V}',\mathcal{W}')$ by the uniqueness of reflexive extension. Let $G_0=SU(p,q)$, $K_0=S(U(p)\times U(q))$ the maximal compact subgroup of $G_0$, and let $G$ and $K$ denote complexifications of $G_0$ and $K_0$ respectively. We choose $G=SL(p+q,\mathbb{C})$ and $K=S(GL(p,\mathbb{C})\times GL(q,\mathbb{C}))$.\\
Let $\mathcal{E}nd_0(\mathcal{V}\oplus\mathcal{W})$ denote the bundle of trace zero endomorphisms of $\mathcal{V}\oplus\mathcal{W}$. We deduce from \cite[Table 2]{maxs}, that the Lie algebra $\mathfrak{g}=\mathfrak{sl}(p+q,\mathbb{C})$ and the typical fiber $\textrm{End}_0(\mathbb{C}^p,\mathbb{C}^q)$ of $\mathcal{E}nd_0(\mathcal{V}\oplus\mathcal{W})$ are isomorphic as $K$-representations. Moreover, the vector bundle $\mathcal{E}nd_0(\mathcal{V}\oplus\mathcal{W})$ admits a reduction in structure group from $GL((p+q)^2-1,\mathbb{C})$ to $K=S(GL(p,\mathbb{C})\times GL(q,\mathbb{C}))$. Thus there is a principal $K$-bundle $P$ on $X_{reg}$ associated to $\mathcal{E}nd_0(\mathcal{V}\oplus\mathcal{W})$ i.e., we have $P\times_K\mathfrak{g}\cong \mathcal{E}nd_0(\mathcal{V}\oplus\mathcal{W})$. The Hodge decomposition of $\mathfrak{g}$ together with identifying isomorphic $K$-representations using \cite[Table 2]{maxs} gives $P\times_K\mathfrak{g}^{-1,1}\cong\mathcal{H}om(\mathcal{V},\mathcal{W})$, $P\times_K\mathfrak{g}^{0,0}\cong(\mathcal{E}nd(\mathcal{V})\oplus \mathcal{E}nd(\mathcal{W}))_0$, and $P\times_K\mathfrak{g}^{1,-1}\cong \mathcal{H}om(\mathcal{W},\mathcal{V})$. Since $\mathcal{T}_{X_{reg}}\cong\mathcal{H}om(\mathcal{V},\mathcal{W})$ by assumption, we have the following isomorphism
\begin{align*}
    \theta:\mathcal{T}_{X_{reg}}\to P\times_K\mathfrak{g}^{-1,1}
\end{align*}
such that $[\theta(u),\theta(v)]=0$ for all local sections $u,v$ of $\mathcal{T}_{X_{reg}}$. It follows that $(P,\theta)$ gives a uniformizing system of Hodge bundles on $X_{reg}$.\\
The system of Hodge bundles $E=P\times_K\mathfrak{g}\cong\mathcal{T}_{X_{reg}}\oplus (\mathcal{E}nd(\mathcal{V})\oplus\mathcal{E}nd(\mathcal{W}))_0\oplus\Omega^1_{X_{reg}}$ is in particular a Higgs bundle with Higgs field $\widehat{\theta}:E\to E\otimes\Omega^1_{X_{reg}}$ given by sending a local section $u$ of $E$ to the local section $v\mapsto[\theta(v),u]$ of $\mathcal{H}om(\mathcal{T}_{X_{reg}},E)\cong E\otimes\Omega^1_{X_{reg}}$. Note that $E$ fits into the following short exact sequence
\begin{align}\label{exseq}
    0\to E=\mathcal{E}nd_0(\mathcal{V}\oplus\mathcal{W})\to \mathcal{E}nd(\mathcal{V}\oplus\mathcal{W})\to \mathcal{L}\to0
\end{align}
where $\mathcal{L}$ is a line bundle of degree zero. Thus we have $c_1(E)=c_1(\mathcal{E}nd_0(\mathcal{V}\oplus\mathcal{W}))=0$. From Proposition \ref{propstab}, it follows that $E$ is $K_X$-polystable as a Higgs bundle on $X_{reg}$, and in fact $E$ is $K_X$-stable as a Higgs bundle on $X_{reg}$ because $\mathfrak{g}=\mathfrak{sl}(p+q,\mathbb{C})$ is a simple Lie algebra. \\
Let $\mathcal{F}_X$ denote the reflexive extension of $E$ to $X$. Then $\mathcal{F}_X$ also fits into an exact sequence
\begin{align*}
    0\to\mathcal{F}_X\to \mathcal{E}nd(\mathcal{V}'\oplus\mathcal{W}')\to\mathcal{L}'\to0
\end{align*}
where $\mathcal{L}'$ is a coherent sheaf of degree 0. Thus $\widehat{c}_1(\mathcal{F}_X)\cdot[K_X]^{n-1}=\widehat{c}_1(\mathcal{E}nd(\mathcal{V}'\oplus\mathcal{W}'))\cdot[K_X]^{n-1}=0$, and $\widehat{c_2}(\mathcal{F}_X)\cdot[K_X]^{n-2}=\widehat{c}_2(\mathcal{E}nd(\mathcal{V}'\oplus\mathcal{W}'))\cdot[K_X]^{n-2}$. Moreover, we compute
\begin{align*}
\widehat{c}_2(\mathcal{F}_X)\cdot[K_X]^{n-2}=[2(p+q)(\widehat{c}_2(\mathcal{V}')+\widehat{c}_2(\mathcal{W}'))-(p+q-1)(\widehat{c}_1(\mathcal{V}')^2+\widehat{c}_1(\mathcal{W}')^2)+2\widehat{c}_1(\mathcal{V}')\widehat{c}_1(\mathcal{W}')]\cdot[K_X]^{n-2},
\end{align*}
so by the assumption of the Proposition we have $\widehat{c}_2(\mathcal{F}_X)\cdot[K_X]^{n-2}=\widehat{ch}_2(\mathcal{F}_X)\cdot[K_X]^{n-2}=0$.\\
Thus the conditions of Theorem \ref{genthm} are satisfied, and it follows that $X\cong\mathcal{A}_{pq}/\Gamma$, where $\Gamma\subset\textrm{Aut}(\mathcal{A}_{pq})$ is a discrete, cocompact subgroup acting fixed point freely in codimension one on the domain $\mathcal{A}_{pq}$. This concludes the proof.
\end{proof}

In order to determine the necessary conditions for uniformization by $\mathcal{A}_{pq}$, we consider two separate cases, namely when (i) $p\neq q$, and (ii) $p=q$. This is because the group of automorphisms $\textrm{Aut}(\mathcal{A}_{pq})$ of $\mathcal{A}_{pq}$ is connected when $p\neq q$, and has two connected components when $p=q$ (see \cite{take}).

\subsection{Necessary conditions when \texorpdfstring{$p\neq q$}{}}

From \cite[p.114]{take}, we know that the automorphism group of $\mathcal{A}_{pq}$ when $p\neq q$ is $\textrm{Aut}(\mathcal{A}_{pq})=PSU(p,q)$. This is a connected Lie group and is a quotient of $SU(p,q)$ by a discrete central subgroup. Thus we can choose the associated Hodge group to be $G_0=SU(p,q)$, and its complexification as $G=SL(p+q,\mathbb{C})$, which is also connected. Thus the sufficient conditions of Proposition \ref{prop6c} are also necessary conditions for a projective klt variety with ample canonical divisor to be uniformized by $\mathcal{A}_{pq}$.
The proof is essentially the same as that of Proposition \ref{propn}.

\begin{proposition}\label{propnc}
Let $X$ be a projective klt variety of dimension $pq$, $p\neq q$ with $K_X$ ample, such that $X=\mathcal{A}_{pq}/\widehat{\Gamma}$, where $\widehat{\Gamma}\subset Aut(\mathcal{A}_{pq})$ is a discrete, cocompact subgroup acting fixed point freely in codimension one on $\mathcal{A}_{pq}$. Then the tangent bundle of the regular locus of $X$ satisfies $\mathcal{T}_{X_{reg}}\cong \mathcal{H}om(\mathcal{E},\mathcal{F})$, where $\mathcal{E}$ and $\mathcal{F}$ are vector bundles of ranks $p$ and $q$ on $X_{reg}$. Moreover, $X$ satisfies the Chern class equality \ref{cceq3}.
\end{proposition}
\begin{proof}
By Theorem \ref{genthm}, the smooth locus $X_{reg}$ of $X$ admits a uniformizing system of Hodge bundles $(P,\theta)$ corresponding to the Hodge group $G_0=SU(p,q)$, such that the system of Hodge bundles $E=P\times_K\mathfrak{g}$ is $K_X$-polystable as a Higgs bundle on $X_{reg}$. Moreover, $X$ satisfies the $\mathbb{Q}$-Chern class equality $\widehat{c}_2(E')\cdot[K_X]^{n-2}=0$, where $E'$ denotes the unique reflexive extension of $E$ to $X$. Hence there is an isomorphism $\theta:\mathcal{T}_{X_{reg}}\cong P\times_K\mathfrak{g}^{-1,1}$. Recall from the proof of Proposition \ref{prop6c} that $\mathfrak{g}^{-1,1}$ is isomorphic as a $K$-representation to $\textrm{Hom}(V,W)$, where $V$ and $W$ are complex $p$ and $q$-dimensional vector spaces. So we can write $\mathcal{T}_{X_{reg}}\cong P\times_K\textrm{Hom}(V,W)=\mathcal{H}om(\mathcal{V},\mathcal{W})$, where $\mathcal{V}$ and $\mathcal{W}$ are vector bundles of ranks $p$ and $q$ on $X_{reg}$. We also have $\Omega^1_{X_{reg}}\cong\mathcal{H}om(\mathcal{W},\mathcal{V})\cong P\times_K\mathfrak{g}^{1,-1}$, and $(\mathcal{E}nd(\mathcal{V})\oplus \mathcal{E}nd(\mathcal{W}))_0\cong P\times_K\mathfrak{g}^{0,0}$, from which it follows that the system of Hodge bundles $P\times_K\mathfrak{g}$ is isomorphic to $\mathcal{E}nd_0(\mathcal{V}\oplus\mathcal{W})$. Let $\mathcal{V}'$ and $\mathcal{W}'$ denote the reflexive extensions of $\mathcal{V}$ and $\mathcal{W}$ to $X$. Then the Chern class equality $\widehat{c}_2(E')\cdot[K_X]^{n-2}=0$ is equivalent to
\begin{align*}
    \widehat{c}_2(\mathcal{E}nd_0(\mathcal{V}'\oplus\mathcal{W}'))\cdot[K_X]^{n-2}=0.
\end{align*}    
Using the short exact sequence \ref{exseq}, and the formula for the second $\mathbb{Q}$-Chern class of a direct sum of reflexive sheaves, the above equality can be rephrased as
\begin{align*}    
    [2(p+q)(\widehat{c}_2(\mathcal{V}')+\widehat{c}_2(\mathcal{W'}))-(p+q-1)(\widehat{c}_1(\mathcal{V}')^2+\widehat{c}_1(\mathcal{W}')^2)+2\widehat{c}_1(\mathcal{V}')\widehat{c}_1(\mathcal{W}')]\cdot[K_X]^{n-2}=0.
\end{align*}
Thus it follows that $X$ satisfies the Chern class equality (\ref{cceq3}), which completes the proof.
\end{proof}

Putting together Propositions \ref{prop6c} and \ref{propnc}, we arrive at the following necessary and sufficient condition for a projective klt variety $X$ with ample canonical divisor to be uniformized by the Hermitian symmetric space $\mathcal{A}_{pq}$ $(p\neq q)$ of type $AIII$. 

\begin{theorem}\label{equivthmc}
Let $X$ be a projective klt variety of dimension $pq\ge2$, $p,q\in\mathbb{Z}_{\ge1}$, $p\neq q$, such that the canonical divisor $K_X$ is ample. Then $X\cong\mathcal{A}_{pq}/\Gamma$, where $\Gamma\subset\textrm{Aut}(\mathcal{A}_{pq})$ is a discrete, cocompact subgroup, and acts fixed point freely in codimension one on $\mathcal{A}_{pq}$, if and only if $X$ satisfies
\begin{itemize}
    \item $\mathcal{T}_{X_{reg}}\cong\mathcal{H}om(\mathcal{V},\mathcal{W})$
    \item $[2(p+q)(\widehat{c}_2(\mathcal{V}')+\widehat{c}_2(\mathcal{W}'))-(p+q-1)(\widehat{c}_1(\mathcal{V}')^2+\widehat{c}_1(\mathcal{W}')^2)+2\widehat{c}_1(\mathcal{V}')\widehat{c}_1(\mathcal{W}')]\cdot[K_X]^{n-2}=0$,
\end{itemize}
where $\mathcal{V}$ and $\mathcal{W}$ are vector bundles of ranks $p$ and $q$ on $X_{reg}$, and $\mathcal{V}'$ and $\mathcal{W}'$ denote the reflexive extensions of $\mathcal{V}$ and $\mathcal{W}$ to $X$.
\end{theorem}

Note that when $p=1$ and $q=n$, the associated domain is the unit ball $\mathbb{B}^n$. In this case we have $\mathcal{V}=\Omega^1_{X_{reg}}$, $\mathcal{W}=\mathcal{O}_{X_{reg}}$, and the system of Hodge bundles is given by $P\times_K\mathfrak{g}=\mathcal{E}nd_0(\Omega^1_{X_{reg}}\oplus\mathcal{O}_{X_{reg}})$. The condition on the structure of the tangent bundle is tautological, and the Chern class condition is the well known Bogomolov-Miyaoka-Yau equality satisfied by ball quotients.

\subsection{Necessary conditions when \texorpdfstring{$p=q$}{}}

In this case the group of holomorphic automorphisms has two connected components. Again from \cite{take} we know that $Z\mapsto Z^T$ is an automorphism of $\mathcal{A}_{pp}$ of order 2 that is not contained in the connected component of $\textrm{Aut}(\mathcal{A}_{pp})$. This defines a group homomorphism $\mathbb{Z}_2\to\textrm{Aut}(\mathcal{A}_{pp})$, i.e., there is a splitting $\textrm{Aut}(\mathcal{A}_{pp})=\textrm{Aut}^0(\mathcal{A}_{pp})\rtimes\mathbb{Z}_2$, where $\textrm{Aut}^0(\mathcal{A}_{pp})=PSU(p,p)$. It follows that $\textrm{Aut}(\mathcal{A}_{pp})$ is a quotient of $SU(p,p)\rtimes\mathbb{Z}_2$ by a discrete central subgroup. Let $G$ denote the chosen complexification of $SU(p,p)\rtimes\mathbb{Z}_2$, and $K$ the complexification of $K_0$. Then the tangent bundle of the regular locus of a projective, klt quotient of $\mathcal{A}_{pp}$ admits a reduction in structure group to $K$. However, as in the polydisc case, such a reduction in structure group does not give a meaningful condition on the tangent sheaf. Therefore, we state the result up to a 2:1 quasi-\'{e}tale cover. The proof is essentially the same as that of Proposition \ref{propn}.

\begin{proposition}\label{propn2}
Let $X$ be a projective klt variety of dimension $p^2$ with $K_X$ ample, such that $X$ is a quotient of $\mathcal{A}_{pp}$ by a discrete, cocompact subgroup $\Gamma$ of $\textrm{Aut}(\mathcal{A}_{pp})$ acting fixed point freely in codimension one. Then $X$ admits a 2:1 quasi-\'{e}tale cover $X'$ such that the tangent bundle of the regular locus $X'_{reg}\subset X'$ satisfies $\mathcal{T}_{X'_{reg}}\cong \mathcal{H}om(\mathcal{E},\mathcal{F})$, where $\mathcal{E}$ and $\mathcal{F}$ are vector bundles of rank $p$ on $X_{reg}$. Moreover, $X'$ satisfies the Chern class equality \ref{cceq3} for $p=q$.
\end{proposition}

\begin{proof}
Since the automorphism group $\textrm{Aut}(\mathcal{A}_{pp})$ has two connected components, the group $\Gamma$ can be expressed as an extension of a normal subgroup $\Gamma'$ by a subgroup of $\mathbb{Z}_2$. Thus the quotient map $\mathcal{A}_{pp}\to X$ factors as
\begin{align*}
    \mathcal{A}_{pp}\to\mathcal{A}_{pp}/\Gamma'\to\mathcal{A}_{pp}/\Gamma=X,
\end{align*}
where $X'=\mathcal{A}_{pp}/\Gamma'$, and the map $X'\to X$ is a Galois, 2:1 quasi-\'{e}tale cover. Note that $X'$ is again projective, klt, and $K_{X'}$ is ample. Again from Theorem \ref{genthm}, it follows that the smooth locus $X'_{reg}$ of $X'$ admits a uniformizing system of Hodge bundles $(P,\theta)$ corresponding to the Hodge group $G_0=SU(p,p)$, such that the system of Hodge bundles $E=P\times_K\mathfrak{g}$ is $K_{X'}$-polystable as a Higgs bundle on $X'_{reg}$. Moreover, $X'$ satisfies the $\mathbb{Q}$-Chern class equality $\widehat{c}_2(E')\cdot[K_{X'}]^{n-2}=0$, where $E'$ is the reflexive extension of $E$ to $X$. Therefore, there is an isomorphism $\theta:\mathcal{T}_{X'_{reg}}\cong P\times_K\mathfrak{g}^{-1,1}$ of vector bundles. Recall from the proof of Proposition \ref{prop6c} that $\mathfrak{g}^{-1,1}$ is isomorphic as a $K$-representation to $\textrm{Hom}(V,W)$, where $V$ and $W$ are $p$-dimensional complex vector spaces. So we can write the tangent bundle as $\mathcal{T}_{X'_{reg}}\cong P\times_K\textrm{Hom}(V,W)=\mathcal{H}om(\mathcal{V},\mathcal{W})$, where $\mathcal{V}$ and $\mathcal{W}$ are vector bundles of rank $p$ on $X'_{reg}$. Moreover, we have $\Omega^1_{X'_{reg}}\cong P\times_K\mathfrak{g}^{1,-1}\cong\mathcal{H}om(\mathcal{W},\mathcal{V})$, and $P\times_K\mathfrak{g}^{0,0}\cong(\mathcal{E}nd(\mathcal{V})\oplus\mathcal{E}nd(\mathcal{W}))_0$. Thus the system of Hodge bundles $P\times_K\mathfrak{g}$ is isomorphic to $\mathcal{E}nd_0(\mathcal{V}\oplus\mathcal{W})$.\\
Let $\mathcal{V}'$ and $\mathcal{W}'$ denote the reflexive extensions of $\mathcal{V}$ and $\mathcal{W}$ to $X'$. Then $E'\cong\mathcal{E}nd_0(\mathcal{V}'\oplus\mathcal{W}')$, and the $\mathbb{Q}$-Chern class equality $\widehat{c}_2(E')\cdot[K_{X'}]^{n-2}=0$ is equivalent to
\begin{align*}
[\widehat{c}_2(\mathcal{E}nd_0(\mathcal{V}',\mathcal{W}'))]\cdot[K_{X'}]^{n-2}=0.
\end{align*}
Expanding this using Chern class formulae, we arrive at the following expression
\begin{align*}
 [4p(\widehat{c}_2(\mathcal{V})'+\widehat{c}_2(\mathcal{W}'))-(2p-1)(\widehat{c}_1(\mathcal{V}')^2)+\widehat{c}_1(\mathcal{W}')^2)+2\widehat{c}_1(\mathcal{V}')\widehat{c}_1(\mathcal{W}')]\cdot[K_{X'}]^{n-2}=0.
\end{align*}
Thus $X'$ satisfies the Chern class equality as claimed, and this completes the proof.
\end{proof}

\end{document}